\documentclass[10pt]{article}
\usepackage{amssymb,amsmath,textcomp}
\usepackage{enumerate}
\usepackage{hyperref}
\usepackage{mathrsfs}
\usepackage{amsthm}
\usepackage{xcolor}
\title{A note on application of mean-field limit to non-exchangeable non-conservative systems}

\author{Piotr Gwiazda\footnote{Institute of Mathematics, Polish Academy of Sciences
\'{S}niadeckich 8, 00-656 Warsaw, Poland,\\
 Interdisciplinary Centre for Mathematical and Computational Modelling, 
University of Warsaw, 
Tyniecka 15/17, 02-630 Warsaw, Poland, }, Katarzyna Ryszewska \footnote{Department of Mathematics and Information Sciences,
Warsaw University of Technology,
Koszykowa 75, 00-662 Warsaw, Poland, \\
 Interdisciplinary Centre for Mathematical and Computational Modelling, 
University of Warsaw, 
Tyniecka 15/17, 02-630 Warsaw, Poland, 
Katarzyna.Ryszewska@pw.edu.pl \\
Both authors were partially supported by National Science Center, Poland through 
UMO-2024/54/A/ST1/00159 Grant.
}}
\newtheorem{theo}{Theorem}
\newtheorem{defi}{Definition}
\newtheorem{prop}{Proposition}
\newtheorem{lemma}{Lemma}
\newtheorem{coro}{Corollary}
\newtheorem{remark}{Remark}

\def\divv{\operatorname {div}}

\def\diam{\operatorname{diam}}
\DeclareMathOperator*{\esssup}{ess~sup}

\newcommand{\eqq}[2]{\begin{equation}  #1  \label{#2}\end{equation}    }
\newcommand{\hd}{\hspace{0.2cm}}

\newcommand{\m}[1]{\mbox{#1}}
\newcommand*{\norm}[1]{\left\Vert{#1}\right\Vert}
\newcommand*{\abs}[1]{\left\vert{#1}\right\vert}
\newcommand{\vf}{\varphi}
\newcommand{\ve}{\varepsilon}
\newcommand{\R}{\mathbb{R}}
\newcommand{\nic}[1]{ }

\newcommand{\mt}{\mathcal{M}_{t_{*}}}
\newcommand{\mb}{\overline{\mathcal{M}}}
\newcommand{\mub}{\overline{\mu}}
\newcommand{\nub}{\overline{\nu}}
\newcommand{\mubd}{\overline{\mu}_{.}}
\newcommand{\nubd}{\overline{\nu}_{.}}

\newcommand{\izj}{\int_{0}^{1}}
\newcommand{\ird}{\int_{\mathbb{R}^d}}
\newcommand{\al}{\alpha}
\newcommand{\ini}{[\frac{i-1}{n},\frac{i}{n}]}

\newcommand{\tx}{\tilde{X}}
\newcommand{\nx}{X}

\newcommand{\nc}{M}
\newcommand{\sn}{\frac{1}{N}\sum_{i=1}^N}
\newcommand{\xb}{\bar{X}}
\newcommand{\Mb}{\bar{M}}

\newcommand{\p}{\mathbb{P}}
\newcommand{\E}{\mathbb{E}}
\newcommand{\Tree}{\mathcal{T}}

\begin{document}
\date{}
\maketitle
\vspace{1cm}

\begin{abstract}
In this paper we apply the newly developed theory of extended graphons to a certain class of non-exchangeable, non-conservative problems. We establish the mean-field limit by proving the convergence of the associated generalized weighted empirical measures.
\end{abstract}
\vspace{0.7em}
\begin{center}
{\bf AMS subject classification:} 35Q83, 35Q70, 35R02, 35Q49, 35R06
\end{center}

\noindent{\bf Keywords:} mean-field limit, interacting particles, non-exchangeable systems, extended graphons, non-conservative problems

\section{Introduction}

Recently, models of interacting particle dynamics have gained considerable attention, due to their vast applications in various fields, including flocking, cell organization in tissues, neural networks and opinion formation. A common assumption in the literature is that $N$ particles exhibit pairwise interactions which are modeled by a connectivity matrix $(w_{ij}^N)_{i,j=1}^N$. Since modeling a  system with a huge number of particles is often computationally intractable, it is common to study the mean-field limit as $N\rightarrow \infty$ instead. In this setting we do not keep track of individual particles, but we approximate the evolution of density of the whole population by a limiting density, as the number of particles tends to infinity. In this paper we focus on, so called, non-exchangeable setting, which accounts for particle heterogeneity. It means that the particles are not treated as identical, each of them may have individual characteristics and interactions, while the exchangeable setting corresponds to $w_{ij}^N=1/N$. Consequently, the non-exchangeable setting presents a significantly more challenging scenario and the mathematical methods are still in development. The first results regarding establishing the mean-field limit in non-exchangeable setting were proven in case of a priori knowledge of the limiting graphon  \cite{CM2019} - \cite{CMM2018}, \cite{KVM2018}. Conceptually, a graphon is a limit object $w:[0,1]\times[0,1]\rightarrow \R$ that encapsulates the asymptotic connectivity of a graph sequence. Subsequently, in \cite{BCN} the results were extended to accommodate dense graphs, without a priori knowledge of the limiting object. This extension is possible because the sequence of underlying discrete interaction matrices in dense networks naturally admits a graphon as its continuous limit. We refer to  \cite{Lovasz}-\cite{LovaszSzegedy} for a comprehensive study of graph limits. Several extensions of graphons have been developed to encompass broader classes of graphs. Here we mention $L^p$ - graphons \cite{BCCZ}-\cite{BCCZ2},  graphops and digraph measures, which were adopted across multiple studies in the context of mean-field limit with notable examples including \cite{Kuehn}, \cite{GC2022}. For a comprehensive review of the development of the theory of mean-field limit we refer to \cite{Aover} and \cite{ADover}. \\
One of the breakthrough results in the area of non-exchangeable mean-field limit
with possibly sparse graphs have been recently obtained in \cite{JPS2025}. The authors introduce the new concept of observables indexed by trees. It allows to pass to the mean-field limit under very general assumptions on the sequence of connectivities $(w_{ij}^N)_{i,j=1}^N$. The mathematical tools developed in \cite{JPS2025} (where a toy model is considered) are subsequently employed in \cite{JZ2025} to establish the mean-field limit for biological neuron networks based
on the so-called stochastic integrate-and-fire  dynamics. 

The aim of this paper is to extend the mathematical methods introduced in \cite{JPS2025} to cover also non-conservative systems. We anticipate that our approach may find valuable applications in biology, particularly in  problems governed by balance laws. Other potential applications arise in opinion dynamics. On the canonical probabilistic space $(\R^d,\mathcal{B}(\R^d),\p)$ we consider the following system of equations for fixed $t_{*}>0$
\begin{equation}\label{part}
\begin{array}{lll}
(i) & \partial_t X_i(t) = \sum_{j=1}^N M_j(t)w^N_{ij}K_{1}(X_i(t)-X_j(t)) & t \in (0,t_*),\\
(ii) & \partial_t M_i(t) = M_i(t)\left[A-\sum_{j=1}^N M_j(t)w^N_{ij}K_{2}(X_i(t)-X_j(t))\right] & t \in  (0,t_*),\\
(iii) & X_i(0) = X_{i}^0, \hd  M_i(0) = M_{i}^0,
\end{array}
\end{equation}
for $i=1,\dots, N$, $N \in \mathbb{N}$. Here, a constant $A>0$, kernels $K_1,K_2$ and a sequence of matrices $\{w_{ij}^N\}_{i,j=1}^N$ are given deterministic data and the initial conditions  $X_i^0:\mathbb{R}^d \rightarrow \mathbb{R}^d $ and $M_i^0:\mathbb{R}^d\rightarrow \mathbb{R}_{+}$ are given random variables. Here and henceforth we denote $\R_{+}:=[0,\infty)$. Note that, we suppress the index $N$ in the notation $X_i,M_i$ for the sake of readability.

The system (\ref{part}) with identical $w_{ij}^N = 1/N$ is a particular case of exchangeable problems studied in \cite{AD2021}, \cite{D2022} and \cite{MPD2019}. In this context, $X_i$ represents the vector of opinions of $i$-th agent and $M_i$ its charisma (the influence which it has on other agents). 
 The problem discussed in \cite{AD2021}, \cite{D2022} and \cite{MPD2019} may be also regarded as a particular case of a problem with adaptive dynamical network, i.e. $\tilde{w}_{ij}(t) =\frac{1}{N} M_j(t)$. For the results regarding mean-field limit for general non-exchangeable adaptive network we refer to recent papers \cite{A2026}, \cite{Poyato2026}, \cite{Throm2},  \cite{Zhou} for dense graph setting and to \cite{GKX2025}, where the authors employ the theory of digraph measures to tackle a more general framework.

While existing literature on systems of type (\ref{part}) (where the evolution of $M_i$ has been modeled using diverse dynamics across different studies)
 is restricted either to the exchangeable regime or to somewhat strong assumptions on the underlying graph, our primary novelty lies in proving the mean-field limit under fairly general conditions, matching the framework introduced in \cite{JPS2025}.   Secondly, unlike in \cite{AD2021} we do not assume that the total weight of all $M_i$ is constant, which leads us to non-conservative equation in the mean field limit (\ref{meana}). This allows us to capture phenomena that cannot be described by existing models. In the context of opinion formation, it may be interpreted as allowing the possibility of polarization and a sharp intensification of sentiment. Although system (\ref{part}) represents an idealized toy model, it establishes a framework for future generalizations. Our next target, is the application of the developed ideas to problems with general adaptive weights \cite{inprep}.

In this paper we closely follow the approach introduced in seminal paper \cite{JPS2025} and adapt it to the new non-conservative framework. 
We establish the convergence of empirical measure  $\sn M_{i}(t,x)\delta_{X_i(t,x)}$ to the solution of continuous problem (\ref{meana}). We point out that such empirical measure (in this case  probabilistic),  has been used already in \cite{AD2021}.   Discussing the mixed object $M_{i}(t,x)\delta_{X_i(t,x)}$  enables us to obtain the propagation of independence, a property that is highly desirable, yet generally challenging to achieve. To arrive at this result, we establish the extension of Glivenko-Cantelli lemma for empirical measures with non-conservative mass. 

 Below we collect the assumptions that we impose on deterministic data and random initial conditions. Note that, with respect to regularity of $K_1$ and the Sobolev regularity of initial conditions, we adopt the assumptions from  \cite{JPS2025}, since improving these baseline conditions lies beyond the scope of our work. \\

Let us assume for every $N \in \mathbb{N}$ that  the family of pairs $(X_{i}^0,M_i^0)_{i=1}^N$ is jointly independent.     Furthermore, let
\eqq{K_1 \in W^{1,\infty}(\R^d;\R^d) \cap W^{1,1}(\R^d;\R^d), \hd \hd K_2 \in W^{1,\infty}(\R^d;\R) \cap L^{1}(\R^d;\R), \hd \hd  K_2 \geq 0.}{asK}

Regarding interaction weights $w_{ij}^N$, we impose the minimal assumptions introduced in \cite{JPS2025} supplemented with nonnegativity assumption: 
\eqq{\max_{1\leq i,j \leq N}w^N_{ij} \rightarrow 0 \m{ as } N\rightarrow \infty \m{ and } w^N_{ij} \geq 0 \m{ for every } i,j = i,\dots, N,}{aswl}
there exists $C_w$ such that 
\eqq{\sup_{N \in \mathbb{N}}\max_{1\leq i\leq N}\sum_{j=1}^Nw^N_{ij} \leq C_w, \hd \hd \sup_{N \in \mathbb{N}}\max_{1\leq j\leq N}\sum_{i=1}^Nw^N_{ij} \leq C_w. }{aswb}

Concerning the initial data we assume that there exist finite numbers $m,M > 0$ such that 
\eqq{
\sup_{N \in \mathbb{N}}\max_{1\leq i \leq N}\esssup_{x\in \R^d}M_i^0(x) \leq M, \hd \hd \inf_{N \in \mathbb{N}}\min_{1\leq i \leq N} M_i^0 \geq 0  \hd \p \hd a.e. 
}{M}

Finally we assume that for every $i=1,\dots, N$, the densities $f_{X_i^0}$ of $X^0_i$ exist, belong to $(W^{1,1}\cap W^{1,\infty})(\R^d)$ and satisfy the following:
\eqq{
  \sup_{N\in \mathbb{N}}\max_{1\leq i \leq N} \norm{g_i}_{W^{1,1}(\R^d)\cap W^{1,\infty}(\R^d)} < \infty, \hd \m{ where } \hd  g_i(x):= f_{X_i^0}(x)\E[M^0_i|X^0_i = x].
}{asf}
\begin{remark}
    Note that, the assumption (\ref{asf}) simplifies significantly if we assume that for every $i \in {1,\dots, N}$, $N\in \mathbb{N}$, the random variables $ X^0_i$ and $M^0_i$ are independent. Then,  
    \[
    g_i(x) =  f_{X^0_i}(x)  \mathbb{E} M^0_i
    \]
    and due to 
(\ref{M}), the assumption (\ref{asf}) reduces to 
    \[
    \sup_{N\in \mathbb{N}}\max_{1\leq i \leq N} \norm{f_{X^0_i}}_{W^{1,1}(\R^d)\cap W^{1,\infty}(\R^d)} < \infty,
    \]
    which coincides with the assumption for conservative problem imposed in \cite{JPS2025}.
\end{remark}

Let us denote by $d_F$ the flat metric on the space of positive Radon measures defined on $\R^d$, i.e.
\[
d_F(\mu,\nu) = \sup \left\{\abs{\ird \phi(x)d(\mu(x)-\nu(x))}: \phi \in W^{1,\infty}(\R^d); \norm{\phi}_{L^{\infty}(\R^d)} \leq 1, \norm{\nabla\phi}_{L^{\infty}(\R^d)}\leq 1\right\}.
\]

In the whole paper the time derivative is always denoted by $\partial_t$ and all other differential operators refer to the space variable, thus we omit the subscript $x$ and denote simply $\nabla,\divv$ etc. Furthermore, similarly as in \cite{JPS2025} while integrating with respect to measure, we use the short notation   $f(dx)$ which we understand as $df(x)$. It is particularly useful if we deal with multi variable objects.
Below we present the main result of this paper. Note that the space $L^\infty_\xi \mathcal{M}_{\zeta}$ is introduced in Definition \ref{space}.

  \begin{theo} \label{maintheorem}
Assume (\ref{asK}) - (\ref{asf}) and let   the family of pairs $(X_{i}^0,M_i^0)_{i=1}^N$ be jointly independent for every $N \in \mathbb{N}$. Moreover, let  
\eqq{
\sup_{N\in \mathbb{N}} \sup_{1\leq i\leq N}\mathbb{E}[\vert X_i^0\vert^2]<\infty.
}{secmo}
Finally, let $(X_i(t), M_i(t))_{i=1}^N$ be a solution to (\ref{part}), given by Lemma \ref{exi}. 
Then, there exist nonnegative $w\in L^\infty_\xi \mathcal{M}_\zeta\cap L^\infty_\zeta \mathcal{M}_\xi$ and $f\in L^\infty((0, t_*)\times(0, 1);W^{1,1}\cap W^{1,\infty}(\R^d))$, such that $f$ is a weak solution to
\eqq{
\partial_t f(t,x,\xi) + \divv ( f(t,x,\xi) V_1[f](t,x,\xi)) =  f(t,x,\xi) \left[A-V_2[f](t,x,\xi)\right], 
}{meana}
where for $j=1,2$
\[
V_j[f](t,x,\xi) = \izj \ird K_{j}(x-y)f(t,y,\zeta)dy w(\xi,d\zeta)
\]
 and, up to the extraction of a subsequence,
\eqq{
\esssup_{0\leq t\leq t_*}\,\mathbb{E}\,d_{F}\left(\int_0^1 f(t,\cdot,\xi)\,d\xi,\;\sn M_{i}(t,x)\delta_{X_i(t,x)}(\cdot)\right)\to 0, \m{ as } N\rightarrow \infty.
}{finalfinal}
 \end{theo}

\begin{remark}\label{wlrem}
    From the proof of \cite[Theorem 5.1]{JPS2025} (pages 722-724) it follows that the extended graphon $w$ in the theorem above is a weak$^*$ limit in $L^\infty_\xi \mathcal{M}_\zeta\cap L^\infty_\zeta \mathcal{M}_\xi$ of a properly selected subsequence of $w^N(\Phi^N(\cdot),\Phi^{N}(\cdot))$, where
    \[
    w^N(\xi,\eta):=N \sum_{i,j=1}^N w_{ij}^N \mathbb{I}_{[\frac{i-1}{N},\frac{i}{N})}(\xi) \mathbb{I}_{[\frac{j-1}{N},\frac{j}{N})}(\eta)
    \]
and $\Phi^{N}:[0,1]\rightarrow [0,1]$ is a properly chosen sequence of a.e. injective, measure preserving maps.
\end{remark}

Before tackling the full complexity of the mean-field limit via the extended graphon framework, we first present a simplified sketch of the argument under the assumption of an a priori known, bounded graphon in purely deterministic case. We find it instructive to first examine the underlying mechanisms of the result within a simplified setting. Consequently, we dedicate Chapter 2 to the exposition of the simpler case. Note that in this setting, even if we assume that the initial conditions are independent random variables rather than deterministic, we do not need to apply a propagation of independence argument. Indeed, since we assume the existence of an a priori known bounded graphon, the empirical measure (\ref{empfin}) is well approximated by the auxiliary measure $\overline{\nu}_{n,m,t}$, which solves the closed equation (\ref{fixed}) with $w=w^n$ being a discretization of a bounded graphon. Consequently, the stability results for the limiting equation given in Proposition \ref{inistab} and Proposition \ref{kernelstab} yield the desired result. In contrast, in the general framework of Theorem \ref{maintheorem}, we rely on a tree-indexed hierarchy where propagation of independence is necessary to factorize the joint measures.
The proof of Theorem \ref{maintheorem} will be carried in two steps presented in Chapter~3 and Chapter~4.  Chapter~3 is devoted to the proof of propagation of independence. In Chapter~4 we adapt the theory of extended graphons introduced in \cite{JPS2025} to our non-conservative setting, which leads to the proof of Theorem~\ref{maintheorem}. We postpone the proofs of technical results to the Appendix in Chapter~5.   

\section{The result in the case of a bounded graphon}
Let us consider the system 
\begin{equation}\label{part1V}
\begin{array}{lll}
(i) & \partial_t X_i(t) = \frac{1}{N}\sum_{j=1}^N M_j(t)w^N_{ij}K_{1}(X_i(t)-X_j(t)) & t \in (0,t_*),\\
(ii) & \partial_t M_i(t) = M_i(t)\left[A-\frac{1}{N}\sum_{j=1}^N M_j(t)w^N_{ij}K_{2}(X_i(t)-X_j(t))\right] & t \in  (0,t_*),\\
(iii) & X_i(0) = X_{i}^0, \hd  M_i(0) = M_{i}^0,
\end{array}
\end{equation}
where for $i=1,\dots,N$ the initial conditions $X^0_i \in \R^d$, $M^0_i \in \R_{+}$ are given. Before we establish the result we introduce suitable for our setting metric space $(\overline{\mathcal{M}},d)$:  \begin{align*}
\overline{\mathcal{M}}:=\{&\overline{\mu}:[0,1]\rightarrow \mathcal{M}(\R^d)\},
\m{ where } \overline{\mu} \m{ denotes the measurable function}  \\
&\overline{\mu}:\xi\mapsto \mu^{\xi} \in \mathcal{M}(\R^d) \m{ such that } \izj\ird d\mu^{\xi}(x)d\xi < \infty \}
\end{align*}
and
\[
d(\overline{\mu},\overline{\nu}) = \izj d_F(\mu^{\xi},\nu^{\xi})d\xi.
\]
Furthermore, we define for every $\al \geq 0$  metric space $(\mathcal{M}_{t_{*}},d_{\alpha})$, where
\[
\mathcal{M}_{t_{*}}:=C([0,t_{*}];\overline{\mathcal{M}}) \hd \m{ and } \hd d_\al(\overline{\mu}_{.},\overline{\nu}_{.}) = \sup_{t \in [0,t_{*}]}e^{-\al t}d(\overline{\mu}_{t},\overline{\nu}_{t}).
\]
\begin{theo}\label{maineasy}
Assume that there exists nonnegative  function $w:[0,1]^2 \rightarrow [0,1]$, such that 
\[
w^{N}_{ij}:=N^2\int_{[\frac{i-1}{N}, \frac{i}{N})}\int_{[\frac{j-1}{N}, \frac{j}{N})} w(\xi,\eta)d\xi d\eta \hd \m{ for } \hd i,j=1,\dots, N
\]
and
\eqq{\sup_{(\xi,\eta) \in [0,1]^2} |w(\xi,\eta)| \leq 1.} {as1}
Furthermore, we assume that there exists $M>0$ such that
\eqq{
\sup_{N\in \mathbb{N}}\max_{1\leq i \leq N} M^0_i \leq M.
}{Mz}
Let $(X_i,M_i)_{i=1}^N$ be a solution to (\ref{part1V}) with $w^N_{ij}$ as above and the kernels $K_1,K_2$ satisfying the assumption~(\ref{asK}).
Let $\mu_t^{\xi}$ be a unique measure solution to 
\eqq{
\partial_t \mu_t^\xi(x) + \divv \left( \mu_t^\xi(x) \izj \ird K_{1}(x-y)d\mu^\eta_t(y) w(\xi,\eta)d\eta\right) =  \mu_t^\xi(x) \left(A - \izj \ird K_{2}(x-y)d\mu_t^\eta(y) w(\xi,\eta)d\eta\right),
}{meanaV}
with the initial condition $\mub_0 = \rho_0^\xi \mathcal{L}^d \in \mb$ satisfying $\sup_{\xi \in (0,1)} \ird d\mu^\xi_0(x)<\infty$. Denote $\mathbb{I}_{i}^n := \mathbb{I}_{[\frac{i-1}{n},\frac{i}{n})}$ and 
define 
the weighted empirical measure  by the formula
\eqq{
\mu_{n,m,t}^{\xi} = \frac{1}{m}\sum_{i=1}^n\sum_{j=1}^m M_{(i-1)m+j}(t)\delta_{X_{(i-1)m+j}(t)}\mathbb{I}_{i}^n(\xi) \hd  n,m>0, \hd N=nm.
}{empfin}
If the initial conditions for the system (\ref{part1V}) are well approximated by a $\mub_0 \in \overline{\mathcal{M}}$, i.e. for every $\ve > 0$ there exists $ n_0\in \mathbb{N}$ such that, for every $n \geq n_0$ exists $m_0 = (\ve,n)$ such that, for every $m \geq m_0$
\[
d(\overline{\mu}_{0},\overline{\mu}_{n,m,0}) < \ve, \hd  \m{ where } \mu_{n,m,0}^{\xi} = \frac{1}{m}\sum_{i=1}^n\sum_{j=1}^m M_{(i-1)m+j}^0\delta_{X_{(i-1)m+j}^0}\mathbb{I}_{i}^n(\xi), 
\]
then for every $\ve > 0$ exists $\tilde{n} > 0$ dependent only on norms $K_1,K_2$ in $W^{1,\infty}$ and $\ve,A,t_{*}$, $\izj \ird d\mu^\xi_0(x)d\xi$, $\sup_{\xi \in (0,1)} \ird d\mu^\xi_0(x)$ and $m_1 > 0$ dependent additionally on $\tilde{n}$ such that for any $n \geq \tilde{n}$  and $m > m_1$
\[
\sup_{t\in (0,t_{*})}d(\mub_t,\mub_{n,m,t}) \leq \ve.
\]
\end{theo}

Below we recall the definition of measure solution to (\ref{meanaV}) adopted in this paper.

\begin{defi}
A family of measures ${\mub_t}:[0,t_{*}]\rightarrow \mb$ is a measure
solution to (\ref{meanaV}) provided that $t\rightarrow \mu^\xi_t$ is narrowly continuous for almost all $\xi \in (0,1)$ and, for any test
function $\Psi \in  C^1_c((0, t_{*}) \times \R^d) \cap W^{1,\infty}([0, t_{*}] \times \R^d)$,  we have
\begin{align*}
 & \int_{0}^{t_{*}}\ird \partial_t \Psi(t,x) d\mu^\xi_t(x)dt =\\  
 &-\int_{0}^{t_{*}}\ird (\nabla \Psi(t,x) \cdot V_1[w,\mubd](t,x,\xi) +  \Psi(t,x)\cdot (A-V_2[w,\mubd](t,x,\xi)))d\mu^\xi_t(x)dt \m{ for } a.a. \hd \xi \in (0,1),
\end{align*}
where
\[
 V_i[w,\overline{\mu}_{.}](t,x,\xi) := \izj \ird K_{i}(x-y)d\mu_t^\eta(y) w(\xi,\eta)d\eta \hd \m{ for } \hd i = 1,2.
\]
\end{defi}

In the proof we follow the approach from \cite{KVM2018}. We also adopt the notation $\mu^\xi_t(x)$ and $d\mu^\xi_t(x)$ from  \cite{KVM2018} contrary to $f(t,x,\xi)$ and $f(t,dx,\xi)$ from the extended graphon setting in the rest of the paper.

\subsection{Solvability and stability of (\ref{meanaV})}
The purpose of this section is to prove the solvability and stability of (\ref{meanaV}) in $\mathcal{M}_{t_{*}}$, utilizing the method of characteristics. First, we note that due to Lipschitz continuity of the kernels $K_1,K_2$ we have the following result.
\begin{lemma}\cite[Lemma 2.2 and Lemma 2.3]{KVM2018}\label{vreg}
For every $\overline{\mu}_{.} \in \mathcal{M}_{t_{*}}$ the functions
\[
 V_i[w,\overline{\mu}_{.}](t,x,\xi) = \izj \ird K_{i}(x-y)d\mu_t^\eta(y) w(\xi,\eta)d\eta \hd \m{ for } \hd i = 1,2
\]
are Lipschitz continuous in $x$ uniformly with respect to $\xi,t$, namely for every $x_0,x_1 \in \R^d$ 
\[
\abs{ V_i[w,\overline{\mu}_{.}](t,x_1,\xi) - V_i[w,\overline{\mu}_{.}](t,x_0,\xi) } \leq \norm{K_i}_{W^{1,\infty}(\R^d)} \izj \ird d\mu_t^\xi(x)d\xi |x_1-x_0|
\]
and continuous in $t$, i.e. for every $t,\tau \in (0,t_{*})$
\[
\abs{ V_i[w,\overline{\mu}_{.}](t,x,\xi) - V_i[w,\overline{\mu}_{.}](\tau,x,\xi) } \leq  \norm{K_i}_{W^{1,\infty}(\R^d)} d(\mub_t,\mub_\tau).
\]
Furthermore, for every $\overline{\mu}_{.}, \overline{\nu}_{.} \in \mathcal{M}_{t_{*}}$
\eqq{
\abs{V_i[w,\overline{\mu}_{.}](t,x,\xi)-V_i[w,\overline{\nu}_{.}](t,x,\xi)} \leq  \norm{K_i}_{W^{1,\infty}(\R^d)}  d(\overline{\mu}_t,\overline{\nu}_t).
}{stabv}
\end{lemma}

Lipschitz continuity of velocity allows us to obtain the existence of the solutions to our problem by the method of characteristics. Note that actually it is not necessary to assume any continuity assumption on $w$ (cf. \cite{KVM2018}).

\begin{coro}
Fix $\mubd \in \mt$ and discuss the problem 
\eqq{
\partial_t X^\xi(t,s,x) =  V_1[w,\overline{\mu}_{.}](t,X^{\xi}(t,s,x),\xi), \hd  X^\xi(s,s,x) = x \in \mathbb{R}^d,
}{char}
where $\xi \in (0,1)$ is a fixed parameter, $0\leq s \leq t\leq t_{*}$. From Lemma \ref{vreg}, function $V_1[w,\overline{\mu}_{.}]$ satisfies the assumptions of \cite[Lemma 8.1.4.]{AGS2005} and for every $x \in \mathbb{R}^d$ and almost all $\xi \in (0,1)$, there exists a unique solution $X^{\xi}(t,s,x)$ to (\ref{char}) which for fixed $s\in [0,t_{*})$ belongs to $L^{\infty}(s,t_{*};W^{1,\infty}(\mathbb{R}^d)) \cap W^{1,1}(s,t_{*};L^{\infty}(\mathbb{R}^d))$. 
\end{coro}
\begin{prop}\label{solvability}
Let $X^{\xi}(t,\cdot):=X^{\xi}(t,0,\cdot)$ for almost all $\xi \in (0,1)$, denote the solution operator from (\ref{char}). For every $\overline{\mu}_0 \in \overline{\mathcal{M}}$, there exists $\al > 0$  dependent only on $A,t_{*}, \norm{K_1}_{W^{1,\infty}(\R^d)}, \norm{K_2}_{W^{1,\infty}(\R^d)}$ and $\izj d\mu_0^\xi d\xi$, such that the fixed point equation
\eqq{
\mu_t^\xi(x) = X^{\xi}(t,\cdot)_{\#}\left[\mu_0^\xi(x)\exp\left(At-\int_{0}^{t}V_2[w,\mubd](s,X^{\xi}(s,x),\xi)ds\right)\right]
}{fixed}
has a unique solution in $(Y_{t_{*}},d_\al)$, where 
\[
Y_{t_{*}}:=\left\{\mubd \in \mt: \sup_{t\in [0,t_{*}]}\izj\ird d\mu^\xi_t(x)d\xi \leq e^{At_{*}}\izj \ird \mu_0^{\xi}(x)d\xi\right\}.
\]
Furthermore, if $\mu_0^\xi(x) = \rho_0^{\xi}(x)d\mathcal{L}^d$, with $\rho^{\xi}_0 \in L^{1}(\R^d)$ for almost all $\xi \in (0,1)$, then the solution to (\ref{fixed}) is a measure solution of (\ref{meanaV}) with the initial condition $\mu_0^\xi(x)$. Moreover, $\mu_t^\xi(x) = \rho^{\xi}_t(x) d\mathcal{L}^d$, $\rho^{\xi}_t \in L^{1}(\R^d)$ for almost all $\xi$.
\end{prop}

The proof of Proposition \ref{solvability}, as well as two stability estimates below, is the adaptation of the reasoning from \cite{KVM2018} for the case with reaction term. We omit the proofs in these cases since they follow standard arguments.

\begin{prop}\label{inistab}
Let $\mubd,\nubd$ be two solutions to (\ref{fixed}) with initial conditions $\overline{\mu}_0, \overline{\nu}_0 \in \mb$, respectively. Then there exists a positive constant $C$ dependent only on $A,t_{*}$, $W^{1,\infty}$ norms of $K_1$ and $K_2$ and $\izj \ird\nu_0^\xi(x)d\xi$, $\izj \ird\mu_0^\xi(x)d\xi$ such that 
\[
\sup_{t \in [0,t_{*}]}d(\mub_t,\nub_t) \leq Cd(\overline{\mu}_0,\overline{\nu}_0).
\]
\end{prop}

\begin{prop}\label{kernelstab}
Let  $\mubd,\nubd$ be two solutions to (\ref{fixed}) with different kernel functions $w_1,w_2$, both satisfying (\ref{as1}), i.e.
\[
\mu_t^\xi(x) = X^{\xi}[w_1,\mubd](t,\cdot)_{\#}\left[\mu_0(\xi,x)\exp\left(At-\int_{0}^{t}V_2[w_1,\overline{\mu}](s,X^{\xi}(s,x),\xi)ds\right)\right],
\]
\[
\nu_t^\xi(x) = X^{\xi}[w_2,\nubd](t,\cdot)_{\#}\left[\mu_0(\xi,x)\exp\left(At-\int_{0}^{t}V_2[w_2,\overline{\nu}](s,X^{\xi}(s,x),\xi)ds\right)\right].
\]
Furthermore, let 
\[
\sup_{\xi \in (0,1)} \ird d\mu^\xi_0 < \infty.
\]
Then, there exists positive constant $C$ dependent on $A,t_{*},c(\overline{\mu}_0)$ and the norms of $K_1,K_2$ in $W^{1,\infty}$, such that
\[
\sup_{t\in [0,t_{*}]}d(\mub_t,\nub_t) \leq C \sup_{\xi \in (0,1)}\ird d\mu_0^\xi(x) \norm{w_1-w_2}_{L^{1}([0,1]^2)},
\]
where $c(\mub_0) := \izj \ird d\mu^\xi_0(x)dx$.
\end{prop}

\subsection{Approximation}
In order to prove Theorem \ref{maineasy} we argue by triangle inequality. Theorem \ref{maineasy} follows from Proposition \ref{pfirst} together with Proposition \ref{auxemp} and Proposition \ref{third}. 
We begin with discretization of the kernel $w$  in the equation (\ref{meanaV}).
\begin{prop}\label{pfirst}
Let $\mu_t^{\xi}$ be a unique weak solution to (\ref{meanaV}) with initial condition $ \mub_0 = \rho_0^\xi \mathcal{L}^d \in \mb$ given by Proposition \ref{solvability}.
We define the discrete approximation of the kernel $w$  by
\[
w^{n}(\xi,\eta) = \sum_{i,j=1}^n w^{n}_{ij}\mathbb{I}_{i}^n(\xi)\mathbb{I}_j^n(\eta), \hd  \m{ where } \hd w^{n}_{ij} := n^2\int_{[\frac{i-1}{n}, \frac{i}{n})}\int_{[\frac{j-1}{n}, \frac{j}{n})} w(\xi,\eta)d\xi d\eta.
\]
Let us denote by $\{\overline{\nu}^n_t\}$ the sequence of solutions to (\ref{meanaV}) with $w=w^n$ and initial condition $\overline{\mu}_0$.

Fix $\ve > 0$. Then, the exists $n_1 > 0$, dependent only on $\ve, A,t_{*},c(\overline{\mu}_0), c(\overline{\nu}_0)$ and the norms of $K_1,K_2$ in $W^{1,\infty}$, such that for any $n > n_1$
\[
\sup_{t \in [0,t_{*}]} d(\mub_t,\nub^n_t) \leq \frac{\ve}{2}.
\]
\end{prop}

\begin{proof}
By \cite[Lemma 5.3]{KM2017} (see also \cite[Lemma 3.3]{KVM2018})
\eqq{w^n \rightarrow w \hd a.e \hd \m{ and in } L^{2}([0,1]^2) \hd \m{ as } \hd  n\rightarrow \infty
}{wconv}

Hence, the claim follows from Proposition \ref{kernelstab}.

\end{proof}
Henceforth, we operate with fixed $n>n_1$. Following \cite{KVM2018} we will construct the auxiliary empirical measure $\nub_{n,m,t}$ which for $m$ big enough is arbitrarily close to $\nub^n$. Let us introduce arbitrary $m >0$ and $N=nm$. Let us introduce the system of $N$ pairs of equations describing $N$ particles whose connections are encoded in $w^n$ which is $n\times n$ matrix in the following way 
\begin{equation}\label{partV}
\begin{array}{ll}
\frac{d}{dt}\tx_{(k-1)m+l}(t) = \frac{1}{N}\sum_{i=1}^n\sum_{j=1}^m \tilde{M}_{(i-1)m+j}(t)w^{n}_{ki}K_{1}(\tx_{(k-1)m+l}(t) -\tx_{(i-1)m+j}(t) )\\
\frac{d}{dt}\tilde{M}_{(k-1)m+l}(t) =\tilde{M}_{(k-1)m+l}(t)\left[A- \frac{1}{N}\sum_{i=1}^n\sum_{j=1}^m \tilde{M}_{(i-1)m+j}(t)w^{n}_{ki}K_{2}(\tx_{(k-1)m+l}(t) -\tx_{(i-1)m+j}(t)\right]\\
\tx_{(k-1)m+l}(0) = X^0_{(k-1)m+l} \in \R^d, \hd \tilde{M}_{(k-1)m+l}(0) = M^0_{(k-1)m+l} \m{ for } k\in \{1,\dots,n\}, \hd l \in \{1,\dots,m\}.
\end{array}
\end{equation}

\begin{prop}\label{auxemp}
Let us define the empirical measure $\nub_{n,m,t} \in \mb$ by the following formula
\[
\nu_{n,m,t}^{\xi} = \frac{1}{m}\sum_{i=1}^n\sum_{j=1}^m \tilde{M}_{(i-1)m+j}(t)\delta_{\tx_{(i-1)m+j}(t)}\mathbb{I}_i^n(\xi),
\]
where $(\tilde{M}_{(i-1)m+j},\tx_{(i-1)m+j})$, $i \in \{1,\dots,n\}, j \in \{1,\dots,m\}$ is a solution to (\ref{partV}). Then, there exists a positive $C$ dependent only on $A,t_{*},M$, $W^{1,\infty}$-norms of $K_1$ and $K_2$ and $\izj \ird d\nu_0^\xi(x)d\xi$, $\izj \ird d\mu_0^\xi(x)d\xi$ such that 
\[
\sup_{t \in [0,t_{*}]}d(\nub_{t}^n,\nub_{n,m,t}) \leq C d(\mub_{0},\nub_{n,m,0}).
\]
\end{prop}
\begin{proof}
Let us show that the empirical measure satisfies equation (\ref{fixed}) with $w^n$ defined in Proposition~\ref{pfirst} and  $\mub_0 = \nub_{n,m,0}$. For $l=1,2$ we have
\begin{align*}
V_l[w^n,\nub_{n,m,.}](t,x,\xi) &= \izj \ird K_{l}(x-y)d\nu^\eta_{n,m,t}(y) w^n(\xi,\eta)d\eta\\
 &= \sum_{i=1}^n\int_{\ini}w^{n}(\xi,\eta)d\eta \frac{1}{m}\sum_{j=1}^m \tilde{M}_{(i-1)m+j}(t)K_{l}(x - \tx_{(i-1)m+j}(t))\\
 & = \frac{1}{N} \sum_{k=1}^n \mathbb{I}^n_k(\xi)\sum_{i=1}^n \sum_{j=1}^m  w^{n}_{ki} \tilde{M}_{(i-1)m+j}(t)K_{l}(x - \tx_{(i-1)m+j}(t)).
\end{align*}
Thus,
\eqq{
V_l[w^n\nub_{n,m,.}](t,\tx_{(k-1)m+l}(t),\xi)= \frac{1}{N} \sum_{k=1}^n \mathbb{I}^n_k(\xi)\sum_{i=1}^n \sum_{j=1}^m  w^{n}_{ki} \tilde{M}_{(i-1)m+j}(t)K_{l}(\tx_{(k-1)m+l}(t) - \tx_{(i-1)m+j}(t))
}{v2x}
and 
\[
\frac{d}{dt}\tx_{(k-1)m+l}(t) =V_1[w^n,\nub_{n,m,.}](t,\tx_{0,(k-1)m+l},\xi) \hd \m{ for } \hd \xi \in [\frac{k-1}{n},\frac{k}{n}).
\]
From the uniqueness of solutions to (\ref{char}), we obtain  $X^\xi[w^n,\nub_{n,m,.}](t,\tx_{0,(k-1)m+l})=\tx_{(k-1)m+l}(t)$ for $\xi \in [\frac{k-1}{n},\frac{k}{n})$.
For $\Phi \in C_c(\mathbb{R}^d)$ and $\xi \in [\frac{k-1}{n},\frac{k}{n})$ we may write
\begin{align*}
&\ird \Phi(x)dX^{\xi}[w^n,\nub_{n,m,.}](t,\cdot)_{\#}\left[\nu^\xi_{0,n,m}\exp\left(At-\int_{0}^{t}V_2[w^n,\overline{\nu}_{n,m,.}](s,X^{\xi}[w^n,\nub_{n,m,.}](s,\cdot),\xi)ds\right)\right](x)\\
&= \ird \Phi(X^{\xi}[w^n,\nub_{n,m,.}](t,x))\exp\left(At-\int_{0}^{t}V_2[w^n,\overline{\nu}_{n,m,.}](s,X^{\xi}[w^n,\nub_{n,m,.}](s,x),\xi)ds\right)d\nu^\xi_{n,m,0}(x)
\end{align*}
\begin{align*}
&=\frac{1}{m}\sum_{i=1}^n\sum_{j=1}^m M_{(i-1)m+j}^0 \Phi(X^{\xi}[w^n,\nub_{n,m,.}](t,X^0_{(i-1)m+j})) \times
\\
&\times \exp\left(At-\int_{0}^{t}V_2[w^n,\overline{\nu}_{n,m,.}](s,X^{\xi}[w^n,\nub_{n,m,.}](s,X^0_{(i-1)m+j})),\xi)ds\right)
\\
&=\frac{1}{m}\sum_{i=1}^n\sum_{j=1}^m M_{(i-1)m+j}^0 \Phi(\tx_{(i-1)m+j}(t))\exp\left(At-\int_{0}^{t}V_2[w^n,\overline{\nu}_{n,m,.}](s,\tx_{(i-1)m+j}(s)),\xi)ds\right).
\end{align*}
We note that by (\ref{v2x})
\[
\int_{0}^{t}V_2[W^n,\overline{\nu}_{n,m,.}](s,\tx_{(i-1)m+j}(s)),\xi)ds = \frac{1}{N} \sum_{a=1}^n \sum_{b=1}^m  w^{n}_{ia} \int_0^t \tilde{M}_{(a-1)m+b}(s)K_{2}(  \tx_{(i-1)m+j}(s)-\tx_{(a-1)m+b}(s))ds.
\]
Hence, from (\ref{partV})
\[
M^0_{(i-1)m+j}\exp\left(At-\int_{0}^{t}V_2[w^n,\overline{\nu}_{n,m,.}](s,\tx_{(i-1)m+j}(s)),\xi)ds\right) = \tilde{M}_{(i-1)m+j}(t),
\]
which leads to
\[
\ird \Phi(x)dX^{\xi}[w^n,\nub_{n,m,.}](t,\cdot)_{\#}\left[\nu^\xi_{0,n,m}\exp\left(At-\int_{0}^{t}V_2[w^n,\overline{\nu}_{n,m,.}](s,X^{\xi}[w^n,\nub_{n,m,.}](s,\cdot),\xi)ds\right)\right](x)
\]
\[
=\frac{1}{m}\sum_{i=1}^n\sum_{j=1}^m\tilde{M}_{(i-1)m+j}(t) \Phi(\tx_{(i-1)m+j}(t)) = \ird \Phi(x)d\nu^\xi_{n,m,t}(x)
\]
for $\xi \in \xi \in [\frac{k-1}{n},\frac{k}{n})$, $k=1,\dots,n$. Thus the empirical measure indeed satisfies (\ref{fixed}) with $w=w^n$ and $\nub_0 = \nub_{n,m,0}$. The claim follows from Proposition \ref{inistab} and the assumption (\ref{Mz}).
\end{proof}

It remains to compare the weighted empirical measure defined in (\ref{empfin}) with auxiliary $\overline{\nu}_{n,m,\cdot}$.
To this end we firstly formulate the lemma whose proof we postpone to Appendix.

\begin{lemma}\label{odestab}
Assume (\ref{asK}) and (\ref{Mz}). Let $(\nx_i(t),M_i(t))_{i=1}^{N}$ and $(\tx_i(t),\tilde{M}_i(t))_{i=1}^{N}$ be two solutions to (\ref{part1V}) with the same initial conditions $X_i(0) = \tx_i(0) = \nx_{i}^0 \in \R^d, \hd M_i(0) = \tilde{M}_i(0) = M_{i}^0 \geq 0 \m{ for } i\in \{1,\dots,N\}$ and different nonnegative weights $(w^N_{ij})_{i,j=1}^N$ and $(\tilde{w}^N_{ij})_{i,j=1}^N$, respectively. Finally, let $\max_{1\leq i,j \leq N}w^N_{ij} \leq  C_w$ and $\max_{1\leq i,j \leq N}\tilde{w}^N_{ij} \leq  C_{\tilde{w}}$ for every $N\in \mathbb{N}$. Then there exists a positive $C$ dependent only on $A,t_{*},M,C_w,C_{\tilde{w}}$, $\norm{K_1}_{W^{1,\infty}(\R^d)}, \norm{K_2}_{W^{1,\infty}(\R^d)}$ such that
\[
\frac{1}{N}\sum_{i=1}^N \abs{M_i(t) - \tilde{M}_i(t)}^2 + \frac{1}{N}\sum_{i=1}^N \abs{\nx_i(t) - \tx_i(t)}^2 \leq C\frac{1}{N^2}\sum_{i,j=1}^N \abs{w^N_{ij}-\tilde{w}^N_{ij}}^2.
\]
\end{lemma}

\begin{prop}\label{third}
For fixed $\ve > 0$ there exists $n_2>0$ dependent only on $\ve, A,t_{*}, M$, $\norm{K_1}_{W^{1,\infty}(\R^d)}, \norm{K_2}_{W^{1,\infty}(\R^d)}$ such that for any $n > n_2$ and for fixed $m > 0$
\[
\sup_{t \in (0,t_{*})}d(\mub_{n,m,t}, \nub_{n,m,t}) \leq \ve/4.
\]
\end{prop}
\begin{proof}
First note that under the assumptions (\ref{asK}) and (\ref{Mz}) there holds 
\[
\max_{1\leq i\leq N} M_i(t) \leq Me^{At_{*}}.
\]
We perform the following computations
\begin{align*}
&d(\nub_{n,m,t}, \mub_{n,m,t}) = \frac{1}{n}\sum_{i=1}^{n}d_F(\nu^{\xi}_{n,m,t}, \mu^{\xi}_{n,m,t})\mathbb{I}^n_i(\xi) = \frac{1}{n}\sum_{i=1}^{n}\mathbb{I}^n_i(\xi)\sup_{\norm{\Phi}_{BL} \leq 1}\abs{\ird\Phi(x)d(\nu^{\xi}_{n,m,t}(x)- \mu^{\xi}_{n,m,t}(x))}\\
&=\frac{1}{nm}\sum_{i=1}^{n}\mathbb{I}^n_i(\xi)\sup_{\norm{\Phi}_{BL} \leq 1}\abs{\sum_{j=1}^m \tilde{M}_{(i-1m+j)}(t)\Phi(\tx_{(i-1)m+j}(t)) - M_{(i-1m+j)}(t)\Phi(X_{(i-1)m+j}(t))}\\
&\leq \frac{1}{nm}\sum_{i=1}^{n}\mathbb{I}^n_i(\xi)\sum_{j=1}^m\abs{ \tilde{M}_{(i-1)m+j)}(t)}\abs{\tx_{(i-1)m+j}(t) - X_{(i-1)m+j}(t)} 
\\
&+  \frac{1}{nm}\sum_{i=1}^{n}\mathbb{I}^n_i(\xi)\sum_{j=1}^m\abs{ \tilde{M}_{(i-1)m+j}(t) - M_{(i-1)m+j}(t)}\leq Me^{At_{*}} \frac{1}{N}\sum_{i=1}^{N}\abs{\tx_{i}(t) - X_{i}(t)} +  \frac{1}{N}\sum_{i=1}^{N}\abs{ \tilde{M}_{i}(t) - M_{i}(t)}
\\
&\leq Me^{At_{*}}\left(\frac{1}{N}\sum_{i=1}^{N}\abs{\tx_{i}(t) - X_{i}(t)}^2\right)^{\frac{1}{2}} +  \left(\frac{1}{N}\sum_{i=1}^{N}\abs{ \tilde{M}_{i}(t) - M_{i}(t)}^2\right)^{\frac{1}{2}}.
\end{align*}
We note that in this notation $(\tx_i,\tilde{M}_i)$ satisfy the system (\ref{part1V}) with the kernel $\tilde{w}^N_{i,j} = w^n_{a,b}$ for \\$i \in \{(a-1)m+1,\dots,am\}$, $j \in \{(b-1)m+1,\dots,bm\}$, $a,b \in \{1,\dots,n\}$. Moreover,
\begin{align*}    
&\izj \izj \abs{w^N(\xi,\eta) - w^n(\xi,\eta)}^2 d\xi d\eta = \sum_{k=1}^n\sum_{l=1}^m\sum_{a=1}^n\sum_{b=1}^m \int_{\frac{(k-1)m+l-1}{N}}^{\frac{(k-1)m+l}{N}}\int_{\frac{(a-1)m+b-1}{N}}^{\frac{(a-1)m+b}{N}}\abs{w^{N}_{(k-1)m+l,(a-1)m+b} - w^n_{k,a}}^2 d\xi d\eta\\
&
=\frac{1}{N^2} \sum_{k=1}^n\sum_{l=1}^m\sum_{a=1}^n\sum_{b=1}^m \abs{w^{N}_{(k-1)m+l,(a-1)m+b} - w^n_{k,a}}^2 = \frac{1}{N^2}\sum_{i=1}^N\sum_{j=1}^N(w^N_{ij} - \tilde{w}^N_{ij})^2. 
\end{align*}
Recalling the definition of $w^n$  as well as the assumption (\ref{as1}), we note that $\max_{1\leq i,j \leq N}w^N_{ij} \leq  1$. Thus, applying Lemma \ref{odestab} we obtain
\[
d(\mub_{n,m,t}, \nub_{n,m,t}) \leq C(A,t_{*},M) \frac{1}{N}\left(\sum_{i,j=1}^N (w^N_{ij}-\tilde{w}^N_{ij})^2\right)^{\frac{1}{2}} \leq C(A,t_{*},M)\norm{w^N-w^n}_{L_{2}([0,1]^2)}
\]
\[
\leq C(A,t_{*},M)\left(\norm{w^N-w}_{L_{2}([0,1]^2)} +\norm{w^n-w}_{L_{2}([0,1]^2)}  \right).
\]
Recalling (\ref{wconv}) we arrive at the claim.
\end{proof}
Proposition \ref{third} together with Proposition \ref{pfirst} and Proposition \ref{auxemp} leads to the proof of Theorem \ref{maineasy}.

\section{Propagation of independence}
Let us come back to the proof of Theorem \ref{maintheorem}. In the first step, we introduce the auxiliary system (\ref{part2}) and we prove that the jointly independent family $(\xb_i(t),\Mb_i(t))_{i=1}^N$ appropriately approximates $(X_i(t),M_i(t))_{i=1}^N$. Our arguments adapt the reasoning from \cite[Chapter 3]{JPS2025}, which follows the standard path from the mean-field limit theory. However, in our non-conservative framework several extensions are required, notably the generalization of Glivenko-Cantelli lemma (Lemma \ref{GC} in Appendix). 

We introduce the following problem
\begin{equation}\label{part2}
\begin{array}{lll}
(i) & \partial_t \xb_i(t) = \sum_{j=1}^N w^N_{ij} \ird K_{1}(\xb_i(t)-y)\tilde{f}_j(t,dy) & t \in (0,t_*),\\
(ii) & \partial_t \Mb_i(t) = \Mb_i(t)\left[A-\sum_{j=1}^N w^N_{ij}\ird K_{2}(\xb_i(t)-y)\tilde{f}_j(t,dy)\right] & t \in  (0,t_*),\\
(iii) & \xb_i(0) = X_{i}^0, \hd \Mb_i(0) = M_{i}^0,
\end{array}
\end{equation}
for $i=1,\dots, N$, where 
\eqq{
\tilde{f}_i(t,x) = \xb_i(t,x)_{\#}[\Mb_i(t,x)d\mathbb{P}(x)].
}{defft}
Here, by $\xb_i(t,x)_{\#}$ we understand the push forward operator, i.e. for any measurable map $T:\mathcal{A}_1\rightarrow \mathcal{A}_2$ and a measure $\mu$ on $\mathcal{A}_1$ we define $(T_{\#}\mu)(C):=\mu(T^{-1}(C))$ for any measurable $C \subset \mathcal{A}_2$.

First let us establish the  existence results for (\ref{part}) and  (\ref{part2}) together with the independence for (\ref{part2}).

\begin{lemma}\label{exi}
Assuming (\ref{asK}), (\ref{aswb}),(\ref{M}), nonnegativity of $w_{ij}^N$ and 
\[
 \max_{1\leq i\leq N}\mathbb{E}|X^0_i| < \infty,
\]
for every $N \in \mathbb{N}$ the problems (\ref{part}) and (\ref{part2}) admit the unique solutions $(X_i(t),M_{i}(t))_{i=1}^N \in C^{1}((0,t_*);L^{1}(\R^d,\p))$ and $(\xb_i(t),\Mb_{i}(t))_{i=1}^N \in C^{1}((0,t_*);L^{1}(\R^{d},\p))$, respectively. 
Moreover, for each $t\in (0,t_{*})$ and $i \in \{1,\dots,N\}$ the random variables $\Mb_i(t), M_i(t)$ are nonnegative $\p$ a.e.  and
\[
\sup_{N\in \mathbb{N}}\max_{1\leq i\leq N}\sup_{t\in (0,t_*)} \Mb_i(t) \leq Me^{At_*},\hd \sup_{N\in \mathbb{N}}\max_{1\leq i\leq N}\sup_{t\in (0,t_*)} M_i(t) \leq Me^{At_*}  \hd \p \hd  a.e. \hd 
\]
Finally, if for every $N \in \mathbb{N}$ the family $(X_{i}^0,M_i^0)_{i=1}^N$ is jointly independent, then also $(\xb_i(t),\Mb_i(t))_{i=1}^N$  is jointly independent, for each $t \in (0,t_*)$ and $N \in \mathbb{N}$.
\end{lemma}
Since the proof of the lemma is rather standard we postpone it to the Appendix. 
Our goal in this section is to prove the following theorem.

\begin{theo}\label{mainprob}
Let us assume that the assumptions (\ref{asK}) -(\ref{M}) are satisfied, the family $(X_{i}^0,M_i^0)_{i=1}^N$ is jointly independent for every $N \in \mathbb{N}$ and there exists $C_1 > 0$ such 
\eqq{
\sup_{N \in \mathbb{N}}\max_{1\leq i \leq N}\E|X_i^0|^2 \leq C_1.
}{ex}
Let $(X_i(t),M_i(t))_{i=1}^N$ and $(\xb_i(t),\Mb_i(t))_{i=1}^N$ be a unique solution to (\ref{part}) and (\ref{part2}) in $C^{1}((0,t_*);L^{1}(\R^{d},\p))$, respectively.  
Then, there exists a  positive constant $C = C(C_1,t_*,C_w,M,A,\norm{K_1}_{L^{\infty} (\R^d)})$, such that  for any $N \in \mathbb{N}$ we have
\eqq{
\esssup_{t \in (0,t_{*})}\E d_{F}\left(\sn \Mb_{i}(t,x)\delta_{\xb_i(t,x)}(\cdot),\sn d\tilde{f}_i(t,\cdot)\right) \leq C N^{-\frac{1}{2+3d/2}}.
}{zGC}
Furthermore, there exists a positive constant $C$ dependent only on $M,t_*,A,C_w,\norm{K_{1}}_{W^{1,\infty}(\R^d)},\norm{K_{2}}_{W^{1,\infty}(\R^d)}$, such that
\eqq{
\esssup_{t \in (0,t_{*})}\E d_{F}\left(\sn \Mb_{i}(t,x)\delta_{\xb_i(t,x)}(\cdot),\sn M_i(t,x)\delta_{X_i(t,x)}(\cdot)\right)
\leq C \max_{1\leq i,j\leq N}\sqrt{w_{ij}^N}.
}{zEX}
Finally,
\eqq{
\esssup_{t \in (0,t_{*})}\E d_{F}\left(\sn M_{i}(t,x)\delta_{X_i(t,x)}(\cdot),\sn d\tilde{f}_i(t,\cdot)\right) \rightarrow 0 \m{ as } N\rightarrow \infty
}{finalp}
and for $i=1,\dots,N$ the measure $\tilde{f}_i$ defined in (\ref{defft}) is a measure solution to
\eqq{
\partial_t \tilde{f}_i(t,x) + \divv\left(\tilde{f}_i(t,x)\sum_{j=1}^Nw_{ij}^N\ird K_{1}(x-y)\tilde{f}_j(t,dy)\right) = \tilde{f}_i(t,x)\left[A- \sum_{j=1}^Nw_{ij}^N\ird K_{2}(x-y)\tilde{f}_j(t,dy)\right].
}{semic}
\end{theo}

Before, proving Theorem \ref{mainprob}, we establish two auxiliary results.

\begin{prop}
Under the assumptions (\ref{asK}), (\ref{aswb}), (\ref{M}) and the second part of  the assumption (\ref{aswl}), the solutions to (\ref{part2}) satisfy
\eqq{M_i^0(x)  \exp\left(-t_*\norm{K_2}_{L^{\infty}(\R^d)} C_w e^{At_*}M\right) \leq \Mb_{i}(t,x) \leq M_i^0(x) e^{At_*} \leq Me^{At_*} \hd \p \hd  a.e. \m{ for every } i=1,\dots, N.}{Mbound}
\end{prop}
\begin{proof}
The estimate from above is already established in Lemma \ref{exi}. Indeed, we note that for each $i=1,\dots, N$ the random variables $\Mb_i(t)$ satisfy
\eqq{
\Mb_i(t) = M_i^0 \exp\left(\int_0^t \left[A-\sum_{j=1}^N w^N_{ij}\ird K_{2}(\xb_i(\tau)-y)\tilde{f}_j(\tau,dy) \right] \right).
}{intM}
Thus, by the assumptions (\ref{asK}) and (\ref{aswl})
\eqq{
\Mb_i(t) \leq M_i^0 e^{At} \hd \p \hd a.e.
}{mab}
and, due to (\ref{M}), we obtain the estimate from above.
In order to obtain the estimate from below we note that from the formula (\ref{intM}) we infer that $\Mb_i(t,x) \geq 0$ $\p$ almost everywhere. Then, using (\ref{asK}), (\ref{aswl}),  (\ref{mab}) and (\ref{M}) we may estimate as follows
\[
\sum_{j=1}^N w^N_{ij}\ird K_{2}(\xb_i(\tau)-y)\tilde{f}_j(\tau,dy) \leq \norm{K_2}_{L^{\infty}(\R^d)} C_w \max_{1\leq j \leq N}\sup_{y\in \R^d}\Mb_j(\tau,y) \leq \norm{K_2}_{L^{\infty}(\R^d)} C_w e^{At_*}M.
\]
Inserting it in (\ref{intM}) we arrive at
\[
\Mb_i(t) \geq M_i^0 \exp(At)\exp\left(-t\norm{K_2}_{L^{\infty}(\R^d)} C_w e^{At_*}M\right) \geq M_i^0 \exp\left(-t_*\norm{K_2}_{L^{\infty}(\R^d)} C_w e^{At_*}M\right), \p \hd a.e.
\]
\end{proof}

In the proof of next lemma we follow the approach from the proof of \cite[Proposition 3.2.]{JPS2025} to show that $(\xb_i,\Mb_i)_{i=1}^N$ well approximate $(X_i,M_i)_{i=1}^N$.

\begin{lemma}\label{exest}
Assuming  (\ref{asK}) -(\ref{M}), (\ref{ex}) and that the family $(X_{i}^0,M_i^0)_{i=1}^N$ is jointly independent, there exists a positive constant $C$ dependent only on $M,t_*,A,C_w,\norm{K_{1}}_{W^{1,\infty}(\R^d)},\norm{K_{2}}_{W^{1,\infty}(\R^d)}$, such that 
\[
\max_{1\leq i \leq N}\E|X_i(t) - \xb_i(t)| + \max_{1\leq i \leq N} \E|M_i(t) - \Mb_i(t)| \leq C \max_{1\leq i,j\leq N}\sqrt{w_{i,j}^N}.
\]
\end{lemma}
\begin{proof}

Let us subtract the equations (\ref{part2})$_{(i)}$ from (\ref{part})$_{(i)}$. 
\eqq{
\partial_t(X_i(t) - \xb_i(t)) = \sum_{j=1}^N w^N_{ij}\left(M_j(t)K_{1}(X_i(t)-X_j(t)) -  \ird K_{1}(\xb_i(t)-y)\tilde{f}_j(t,dy)\right).
}{nax}

We integrate the equation with respect to $\p$ and make use of the triangle inequality
\begin{align*}
  \partial_t\ird (X_i(t) - \xb_i(t)) d\p  &= \sum_{j=1}^N w^N_{ij}\ird M_j(t)[K_{1}(X_i(t)-X_j(t))-K_1(\xb_i(t)-\xb_j(t))] d\p \\
  &+ \sum_{j=1}^N w^N_{ij}\ird (M_j(t)-\Mb_j(t))K_1(\xb_i(t)-\xb_j(t)) d\p\\
  &+ \sum_{j=1}^N w^N_{ij} \ird \left( \Mb_j(t)K_1(\xb_i(t)-\xb_j(t))-  \ird K_{1}(\xb_i(t)-y)\tilde{f}_j(t,dy)\right)d\p.
\end{align*}
Hence,
\begin{align*}
\partial_t \E \abs{X_i(t) - \xb_i(t)} &\leq 
\sum_{j=1}^N w^N_{ij} \E \left( |M_j(t)|\abs{K_1(X_i(t)-X_j(t)) - K_1(\xb_i(t)-\xb_j(t))}\right)  \\ 
& + \E \abs{\sum_{i=1}^Nw_{ij}^N\left(M_j(t)-\Mb_j(t)\right)K_{1}(\xb_i(t)-\xb_j(t))}\\
& + \E \abs{\sum_{j=1}^N w^N_{ij}  \left( \Mb_j(t)K_1(\xb_i(t)-\xb_j(t))-  \ird K_{1}(\xb_i(t)-y)\tilde{f}_j(t,dy)\right)} =:\sum_{k=1}^3 I_k.
\end{align*}
We estimate, the first expression using (\ref{asK}), (\ref{aswb}) and (\ref{Mbound})
\eqq{
I_1 \leq 2\norm{K_{1}}_{W^{1,\infty}(\R^d)} Me^{At_*} C_w \max_{1\leq j \leq N}\E|X_j(t)-\xb_j(t)|.
}{I1}
Moreover, again by (\ref{asK}) and (\ref{aswb}) 
\eqq{
I_2 \leq \norm{K_1}_{L^{\infty}(\R^d)} C_w \max_{1\leq j \leq N}\E|M_j(t) - \Mb_j(t)|.
}{I2}

We note that by Lemma \ref{exi} the random variable $(\xb_j(t),\Mb_j(t))$ is independent on $\xb_i(t)$ for each $t \in [0,t_*)$ if $i\neq j$, thus for $i\neq j$ we have 
\[
\E\left[\Mb_j(t)K_1(\xb_i(t)-\xb_j(t))| \xb_i(t)\right] = \ird\Mb_j(t,y) K_{1}(\xb_i(t,x)-\xb_j(t,y))d\p(y) = \ird  K_{1}(\xb_i(t,x)-y)\tilde{f}_{j}(t,dy).
\]
In order to estimate $I_3$ we firstly deal with second moment
\begin{align*}
 & \E \abs{\sum_{j=1}^N w^N_{ij}  \left( \Mb_j(t)K_1(\xb_i(t)-\xb_j(t))-  \ird K_{1}(\xb_i(t)-y)\tilde{f}_j(t,dy)\right)}^2 \\ 
  &\leq 2 \E \abs{\sum_{j=1, j\neq i}^N w^N_{ij}  \left( \Mb_j(t)K_1(\xb_i(t)-\xb_j(t))-  \ird K_{1}(\xb_i(t)-y)\tilde{f}_j(t,dy)\right)}^2\\
  &+ 2 \E \abs{ w^N_{ii}  \left( \Mb_i(t)K_1(0)-  \ird K_{1}(\xb_i(t)-y)\tilde{f}_i(t,dy)\right)}^2.
\end{align*}
With the first term we deal as follows
\[
\E \abs{\sum_{j\neq i}^N w^N_{ij}  \left( \Mb_j(t)K_1(\xb_i(t)-\xb_j(t))-  \ird K_{1}(\xb_i(t)-y)\tilde{f}_j(t,dy)\right)}^2
= \E \sum_{j\neq i}^N\sum_{k\neq i}^N w_{ij}^N w_{ik}^N  H_{i,j}(t)\cdot H_{i,k}(t),
\]
where we denoted
\[
H_{i,j}(t) = \Mb_j(t)K_1(\xb_i(t)-\xb_j(t))- \E [\Mb_j(t)K_1(\xb_i(t)-\xb_j(t))|\xb_i(t)].
\]
We will show that for $k\neq j, k\neq i, j\neq i$
\[
\E [H_{i,j}(t)\cdot H_{i,k}(t)] = 0.
\]
To that end, note that for $k\neq j, k\neq i, j\neq i$ random variables $H_{i,j}(t)$ and $H_{i,k}(t)$ are conditionally independent under the condition $\xb_i(t)$. Furthermore,  
since for any random variables $X,Y$ we have $\E(\E(X|Y)) = \E X$  and if $X$ and $Y$ are conditionally independent under condition $Z$ then $\E(XY|Z) = \E(X|Z)\E(Y|Z)$, we may write
\[
\E [H_{i,j}(t)\cdot H_{i,k}(t)] = \E \left[\E\left[H_{i,j}(t)\cdot H_{i,k}(t)|\xb_{i}(t)\right]\right] = \E \left[\E\left[H_{i,j}(t)|\xb_{i}(t)\right]\E\left[H_{i,k}(t)|\xb_{i}(t)\right]\right] = 0,
\]
where the last equality follows from $\E\left[G_{i,k}(t)|\xb_{i}(t)\right] = 0$, since for any random variables $X,Y$ there holds $\E(X-\E(X|Y)|Y) = E(X|Y) - E(X|Y) = 0.$ All in all, we obtain 
\[
\E \abs{\sum_{j=1}^N w^N_{ij}  \left( \Mb_j(t)K_1(\xb_i(t)-\xb_j(t))-  \ird K_{1}(\xb_i(t)-y)\tilde{f}_j(t,dy)\right)}^2
\]
\[
\leq  2\sum_{j=1}^N (w_{ij}^N)^{2} \E\left( \Mb_j(t)K_1(\xb_i(t)-\xb_j(t))-  \ird K_{1}(\xb_i(t)-y)\tilde{f}_j(t,dy)\right)^2
\]
\[
\leq 8 M^2e^{2At_*}\norm{K_{1}}_{L^{
\infty}(\R^d)}^2\sum_{j=1}^N (w_{ij}^N)^2 \leq  8 M^2e^{2At_*}\norm{K_{1}}_{L^{
\infty}(\R^d)}^2 C_w \max_{1\leq j\leq N}w_{ij}^N,
\]
where we applied (\ref{asK}), (\ref{Mbound}) and (\ref{aswb}).

Hence, by Jensen inequality 
\eqq{
I_3 \leq \sqrt{8} Me^{At_*}\norm{K_{1}}_{L^{
\infty}(\R^d)} \sqrt{C_w} \max_{1\leq j\leq N}\sqrt{w_{ij}^N}.
}{I3}

Combining (\ref{I1}), (\ref{I2}) and (\ref{I3}) we obtain for each $1\leq i \leq N$
\eqq{
\partial_t \E \abs{X_i(t) - \xb_i(t)} \leq C \left(\max_{1\leq i\leq N}\E \abs{X_i(t) - \xb_i(t)} +  \max_{1\leq i\leq N}\E\abs {M_i(t) - \Mb_i(t)}\right) + C \max_{1\leq i,j\leq N}\sqrt{w_{ij}^N},
}{xdif}
where $C$ is a positive constant dependent only on $M,A,t_*,C_w,\norm{K_{1}}_{W^{1,\infty}(\R^d)}$, which may change from line to line.
Similarly we subtract  (\ref{part2})$_{(ii)}$ from (\ref{part})$_{(ii)}$ and integrate with respect to $\p$. 
\begin{align*}
\partial_t \E |M_i(t)-\Mb_i(t)| &\leq A\E |M_i(t)-\Mb_i(t)|\\
&+\E\abs{\Mb_i(t)}\abs{\sum_{j=1}^Nw_{ij}^N\left(\ird K_{2}(\xb_i(t)-y)\tilde{f}_j(t,dy) - M_{j}(t)K_{2}(X_i(t)-X_j(t))\right)}.
\end{align*}

Note that the last term under expectation differs from the right hand side of $(\ref{nax})$ only by the change of kernel $K_1$ into $K_2$. Hence, repeating the arguments leading to (\ref{xdif}), we obtain the following estimate for every $1\leq i \leq N$
\eqq{
\partial_t \E |M_i(t)-\Mb_i(t)| \leq C \max_{1\leq i \leq N}\E |M_i(t)-\Mb_i(t)| +C\max_{1\leq i\leq N}\E \abs{X_i(t) - \xb_i(t)} + C\max_{1\leq i,j \leq N}\sqrt{w_{ij}^N},
}{mdif}
where $C$ is a positive constant dependent only on $M,t_*,A,C_w,\norm{K_{2}}_{W^{1,\infty}(\R^d)}$. We integrate  (\ref{xdif}) and (\ref{mdif}) w.r.t. time and then add the resulting inequalities  to the result
\begin{align*}
 &\max_{1\leq i \leq N} \E |X_i(t)-\xb_i(t)|+\max_{1\leq i \leq N}\E |M_i(t)-\Mb_i(t)|\leq \\
&C\int_0^t \left(\max_{1\leq i \leq N}\E |X_i(\tau)-\xb_i(\tau)|+\max_{1\leq i \leq N}\E |M_i(\tau)-\Mb_i(\tau)|\right) d\tau+ Ct\max_{1\leq i,j \leq N}\sqrt{w_{ij}^N}, 
\end{align*}
which by Gr\"onwall inequality gives
\[
\max_{1\leq i \leq N} \E |X_i(t)-\xb_i(t)|+ \max_{1\leq i \leq N} \E |M_i(t)-\Mb_i(t)| \leq C\max_{1\leq i,j \leq N}\sqrt{w_{ij}^N} 
\]
for some positive constant $C$ dependent only on $M,A,t_*,C_w,\norm{K_{1}}_{W^{1,\infty}(\R^d)},\norm{K_{2}}_{W^{1,\infty}(\R^d)}$.
\end{proof}

The last ingredient in the proof of Theorem \ref{mainprob} is the following extension of Glivenko-Cantelli lemma. We postpone the proof of the lemma to Appendix. 

\begin{lemma}[Glivenko-Cantelli]\label{GC}
Discuss the probabilistic space $(\R^d,\mathcal{B}(\R^d),\p)$ with  $d\geq 1$. Let $X_i:\mathbb{R}^d\rightarrow \mathbb{R}^d$ and $M_i:\mathbb{R}^d \rightarrow \mathbb{R}_{+}$ for $i=1,\dots,N$, $N \in \mathbb{N}$ be two sequences of random variables such that $(X_i,M_i)$ and $(X_j,M_j)$ are independent for every $i\neq j$. Furthermore, assume 
\eqq{\sup_{N \in \mathbb{N}}\frac{1}{N}\sum_{i=1}^N \mathbb{E}|X_i|^2 \leq C}{xas}
and that there exist a positive constant $M > 0$, such that
\eqq{\max_{1\leq i \leq N} \esssup_{x \in \R^d} |M_i(x)| \leq M.}{mas}
Let us define
\[
\mu_N := \frac{1}{N}\sum_{i=1}^N M_i\delta_{X_i}, \hd \hd \nu_N:= \frac{1}{N}\sum_{i=1}^N d\tilde{f}_i, \hd \tilde{f}_i = X_{i\#}[M_i d\mathbb{P}].
\]
Then, there exists a positive $\bar{C} = \bar{C}(M,C,d)$ such that for every $N \in \mathbb{N} $
\[
\mathbb{E} d_{F}\left(\mu_N,\nu_N\right) \leq \bar{C}(M,C,d) N^{-\frac{1}{2+3d/2}}.
\]
\end{lemma}

At last, we are ready to prove the main theorem of this chapter.

\begin{proof}[Proof of Theorem \ref{mainprob}]
Let us begin with (\ref{zGC}). Note that due to estimate (\ref{Mbound})  
\eqq{
\max_{1\leq i \leq N}\sup_{t\in (0,t_*)}\esssup_{x\in \R^d}\Mb_i(t,x) \leq e^{At_*}M.
}{Mesti}
Furthermore, integrating (\ref{part2})$_{(i)}$ w.r.t. time and using  (\ref{asK}), (\ref{aswl}),(\ref{Mbound}) we obtain
\eqq{
\E \abs{\xb_i(t)}^2 \leq \E\left(\abs{X^0_i} + t\norm{K_1}_{L^{\infty}(\R^d)}C_wM e^{At_*}\right)^2 \leq 2 \E\abs{X^0_i}^2 + 2\left(t_*\norm{K_1}_{L^{\infty}(\R^d)}C_wM e^{At_*}\right)^2.
}{secmobound}
Hence, due to (\ref{ex}), there exists $C = C(C_1,t_*,M,A,C_w,\norm{K_1}_{L^{\infty}(\R^d)})$ such that $\E \abs{\xb_i(t)}^2 \leq C$.
Thus, we may apply Lemma \ref{GC} with $X_i = \xb_i(t)$, $M_i = \Mb_i(t)$ and we arrive at (\ref{zGC}). To show (\ref{zEX}) we perform the following estimate
\[
\E \sup_{\norm{\Phi}_{L^\infty}\leq 1, \norm{\nabla \Phi}_{L^\infty}\leq 1 }\abs{\sn \left[M_i(t)\Phi(X_i(t)) - \Mb_i(t)\Phi(\xb_i(t))\right]}
\]
\[
\leq \sn \Big[\E\abs{M_i(t)-\Mb_i(t)} + \E\left[\abs{\Mb_i(t)}\abs{\xb_i(t)-X_i(t)}\right]\Big]
\leq (1+Me^{At_{*}})C \max_{1\leq i,j\leq N}\sqrt{w_{ij}^N},
\]
where we applied Lemma \ref{exest} together with (\ref{Mbound}). Thus, (\ref{zEX}) is proven. The convergence (\ref{finalp}) is a direct consequence of (\ref{zGC}), (\ref{zEX}) and the assumption (\ref{aswl}). It remains to show that $\tilde{f}_i$ is a measure solution to (\ref{semic}). To this end, note that by the definition of $\tilde{f}_{i}$  for any $\vf \in C_{c}^1(\R^d)$ we have
\begin{align*}
 &\partial_t \ird \vf(x)\tilde{f}_i(t,dx) = \partial_t \ird \vf(\xb_i(t,x))\Mb_i(t,x)d\p(x)  \\
 &=\ird \nabla \vf(\xb_i(t,x)) \partial_t \xb_i(t,x) \Mb_i(t,x)d\p(x) + \ird  \vf(\xb_i(t,x)) \partial_t \Mb_i(t,x)d\p(x)\\
 &=\ird \nabla \vf(\xb_i(t,x)) \sum_{j=1}^N w^N_{ij} \ird K_{1}(\xb_i(t,x)-y)\tilde{f}_j(t,dy) \Mb_i(t,x)d\p(x)\\
 &+  \ird  \vf(\xb_i(t,x)) \Mb_i(t,x)\left[A-\sum_{j=1}^N w^N_{ij}\ird K_{2}(\xb_i(t,x)-y)\tilde{f}_j(t,dy)\right]d\p(x)\\
 & = \ird \nabla \vf(x) \sum_{j=1}^N w^N_{ij} \ird K_{1}(x-y)\tilde{f}_j(t,dy) \tilde{f}_i(t,dx)\\
 &+ \ird  \vf(x) \left[A-\sum_{j=1}^N w^N_{ij}\ird K_{2}(x-y)\tilde{f}_j(t,dy)\right]\tilde{f}_i(t,dx).
\end{align*}
Hence, indeed $\tilde{f}_i$ is a measure solution to (\ref{semic}), which finishes the proof of the theorem.
\end{proof}
We finish this chapter with the proposition that establishes the regularity of $\tilde{f}_i$ under the additional assumption (\ref{asf}). The proof may be found in the Appendix.
\begin{prop}\label{fireg}
Denote $\tilde{f}^0_i = \tilde{f}_i(0)$. Under the assumptions of Theorem \ref{mainprob} and additionally (\ref{asf}) function $g_i$ defined in (\ref{asf}) is a density of $\tilde{f}^0_i$ and the measures $\tilde{f}_i$ actually belong to $L^{\infty}((0,t_{*});(W^{1,1}\cap W^{1,\infty})(\R^d))$.
\end{prop}

\begin{section}{ Mean-field limit by the theory of extended graphons}
In this section we show that the approach introduced in \cite{JPS2025} may be extended to our non-conservative framework. Firstly, we recall the necessary definitions and results from \cite{JPS2025}. Subsequently, we extend some of the results from \cite{JPS2025} to the non-conservative setting and finally give a proof of Theorem \ref{maintheorem}.\\
While introducing the notion of extended graphon, we stick to the original notation from \cite{JPS2025} and employ the subscripts in the notation of spaces to indicate in which variable the kernel belongs to a certain space.

\begin{subsection}{Introduction to the setting from \cite{JPS2025}}

We begin with the crucial definition of the extended graphon.

\begin{defi}\label{space}\cite[Definition 4.5 and Definition 4.6]{JPS2025}
We denote by  $L^\infty_\xi\mathcal{M}_\zeta$ the weak$^*$ Bochner space which is dual  to $L^1_\xi([0,1]; C_\zeta([0,1]))$, i.e., $w \in L^\infty_\xi\mathcal{M}_\zeta$  if the map  $\xi\in [0,\ 1]\mapsto w(\xi,d\zeta)\in \mathcal{M}([0,\ 1])$ is  weak$^*$ measurable (for every $\phi \in C([0,1])$ the map: $ 
\xi\mapsto \izj \phi(\zeta)w(\xi,d\zeta)$ is  measurable),  essentially bounded and 
$$\Vert w\Vert_{L^\infty_\xi\mathcal{M}_\zeta}:=\sup_{\Vert \phi\Vert_{C([0, 1])}\leq 1}\esssup_{\xi\in (0, 1)}\left\vert\int_{0}^1\phi(\zeta)\,w(\xi,d\zeta)\right\vert < \infty,$$
where we identify the elements $w_1,w_2$, if for every $\phi \in C([0,1])$, we have $  \izj \phi(\zeta)w_1(\xi,d\zeta) = \izj \phi(\zeta)w_2(\xi,d\zeta) \hd a.e.$
Similarly, we define $L^\infty_\zeta\mathcal{M}_\xi$. 
We call $w$ an extended graphon if $w \in L^{\infty}_{\xi}\mathcal{M}_{\zeta}\cap L^{\infty}_{\zeta}\mathcal{M}_{\xi}$ and we introduce the notation
\[
\norm{w}:= \max\{\norm{w}_{L^{\infty}_{\xi}\mathcal{M}_{\zeta}}, \norm{w}_{L^{\infty}_{\zeta}\mathcal{M}_{\xi}}\}.
\]
\end{defi}

\begin{remark}\cite[ page 683]{JPS2025}
Since the space $C([0,1])$ is separable, the map: $
\xi \mapsto \norm{w(\xi,\cdot)}_{\mathcal{M}([0,1])}$ is measurable and essentially bounded. Hence the equivalence relation simplifies to $w_1(\xi,\cdot)=w_2(\xi,\cdot)$ in $\mathcal{M}([0,1])$ for almost all $\xi$. Also
\[
\norm{w}_{L^{\infty}_{\xi}\mathcal{M}_{\zeta}} = \esssup_{\xi\in [0,1]}\norm{w(\xi,\cdot)}_{\mathcal{M}([0,1])}.
\]
\end{remark}

One of the primary contributions in \cite{JPS2025} is extending the operator $\Phi\rightarrow \int_0^1\Phi(\zeta)w(\cdot,d\zeta)$ from continuous functions $\Phi$ to merely essentially bounded functions. The result reads as follows.

\begin{lemma}\cite[Lemma 4.7 and Lemma 4.8]{JPS2025}\label{impoest}
The bilinear,  bounded operator
\[
\begin{array}{ccl}
(L^\infty_\xi\mathcal{M}_\zeta\cap L^\infty_\zeta\mathcal{M}_\xi)\times C_\zeta([0,1]) & \longrightarrow & L^\infty_\xi((0,1)),\\
(w,\phi) & \longmapsto & \int_0^1\phi(\zeta)\,w(\cdot,d\zeta).
\end{array}
\]
 extends to a bounded operator from $(L^\infty_\xi\mathcal{M}_\zeta\cap L^\infty_\zeta\mathcal{M}_\xi)\times L^\infty_\zeta((0,1))$ to $L^\infty_\xi((0,1))$. 
Furthermore, for any $w\in L^\infty_\xi\mathcal{M}_\zeta\cap L^\infty_\zeta\mathcal{M}_\xi$ and $\phi\in L^\infty((0,1))$
\[
\left\|\int_0^1 \phi(\zeta)\,w(\cdot,d\zeta)\right\|_{L^1((0,1))}\leq \|w\|_{L^\infty_\zeta \mathcal{M}_\xi}\,\|\phi\|_{L^1((0,1))}, \hd 
\left\|\int_0^1 \phi(\zeta)\,w(\cdot,d\zeta)\right\|_{L^\infty((0,1))}\leq \|w\|_{L^\infty_\xi \mathcal{M}_\zeta}\,\|\phi\|_{L^\infty((0,1))}.
\]
In addition, for any uniformly bounded sequence $(\phi_n)_{n\in \mathbb{N}} \in L^{\infty}((0,1))$ such that $\phi_n\rightarrow \phi$ in $L^1((0,1))$ and a sequence   $w_n\overset{*}{\rightharpoonup} w$  in $L^\infty_\xi\mathcal{M}_\zeta\cap L^\infty_\zeta\mathcal{M}_\xi$ the following convergence holds  weak$^*$ topology on $L^{\infty}((0,1))$
  \[
\int_0^1\phi_n(\zeta)\,w_n(\cdot,d\zeta)\overset{*}{\rightharpoonup} \int_0^1 \phi(\zeta)w(\cdot,d\zeta).
\]
Similarly, the bilinear, bounded operator
\[
\begin{array}{ccl}
(L^\infty_\xi\mathcal{M}_\zeta\cap L^\infty_\zeta\mathcal{M}_\xi)\times C_\zeta([0,1]; B_x^*) & \longrightarrow & L^\infty_\xi((0,1); B_x^*),\\
(w,\phi) & \longmapsto & \int_0^1\phi(\cdot,\zeta)w(\cdot,d\zeta),
\end{array}
\]
where $B_x^*$ denotes a Banach space, extends to a bounded operator from $(L^\infty_\xi\mathcal{M}_\zeta\cap L^\infty_\zeta\mathcal{M}_\xi)\times L^\infty_\zeta((0,1); B_x^*)$ to $L^\infty_\xi((0,1); B_x^*)$, and  we have the following estimates
\begin{align*}
  \left\|\int_0^1 \phi(\cdot,\zeta)\,w(\cdot,d\zeta)\right\|_{L^1((0,1); B_x^*)}&\leq \|w\|_{L^\infty_\zeta \mathcal{M}_\xi}\,\|\phi\|_{L^1((0,1); B_x^*)},\\
  \left\|\int_0^1 \phi(\cdot,\zeta)\,w(\cdot,d\zeta)\right\|_{L^\infty((0,1); B_x^*)}&\leq \|w\|_{L^\infty_\xi \mathcal{M}_\zeta}\,\|\phi\|_{L^\infty((0,1); B_x^*)},
\end{align*}
for any $w\in L^\infty_\xi\mathcal{M}_\zeta\cap L^\infty_\zeta\mathcal{M}_\xi$ and $\phi\in L^\infty((0,1); B_x^*)$. 
\end{lemma}  

Another fundamental development in \cite{JPS2025} is an introduction of the new observables indexed by finite trees. Before we recall the notion of these new observables, let us present the basic definitions from graph theory.

\begin{defi}(cf. \cite[Definition 4.1]{JPS2025})
~
\begin{enumerate}
\item We call a   (simple) graph a pair $G=(V,E)$ where $V$ is a finite set (and represents vertices) and $E\subseteq \{(i,j)\in V\times V:\,i\neq j\}$ ( and represent the edges). We call $G=(V,E)$ a weighted graph if it is endowed with a weight function $W:V\times V\rightarrow \mathbb{R}$ such that $E=\{(i,j)\in V\times V:\, W(i,j)\neq 0\}$. We denote the number of vertices (or order) of $G$  by $\vert G\vert:=\# V$. 
\item A graph $T=(V,E)$ is called a rooted directed tree if:
\begin{itemize}
\item $V=\{1,2,\dots,|T|\}$ and the vertex indexed by $1$ is called the root, 
\item for every $u\in V$, $(u,1)\notin E$,
\item for every $u\in V$, $u\neq 1$, there exists exactly one $v \in V$ such that $(v,u) \in E$, 
\item for every $u\in V$, $u\neq 1$, there exists a directed path from $\{1\}$ to $u$.
\end{itemize}
A vertex in tree which has only one edge connecting it with another vertex is called a leaf.
\item We define the family $\Tree_n$ of  labeled trees of order $n$ by the following recursive formula
\begin{align*}
\Tree_1&:=\{T_1\},\\
\Tree_{n+1}&:=\{T+i:\,T\in \Tree_n,\,i\in \{1,\ldots,n\}\},\quad n\in \mathbb{N},
\end{align*}
where $T_1$ is the  tree with only one vertex $\{1\}$ and for any tree $T$ with vertices $\{1,\ldots,n\}$ by $T+i$ we denote a tree created from $T$ by adding a leaf indexed by $n+1$ to an $i$-th vertex. The family of all labeled trees of arbitrary order is then defined by $\Tree:=\cup_{n=1}^\infty\Tree_n$. 
\end{enumerate}
\end{defi}

Now we are ready to define the operator $\tau$ which applied to particular graphon $w$ and function $f$ gives rise to the new observable.  

\begin{defi}\cite[Definition 4.13]{JPS2025}
For any $T\in \Tree$,  $w\in L^\infty_\xi\mathcal{M}_\zeta \cap L^\infty_\zeta\mathcal{M}_\xi$ and  $f\in L^\infty((0,1)),L^1(\R^d)
)$ we formally define the operator
\[
\tau(T,w,f)(x_1,\ldots,x_{|T|}):=\int_{[0,\ 1]^{|T|}} \prod_{(k,l)\in E(T)} w(\xi_k,\xi_l)\,\prod_{m\in V(T)} f(x_m,\xi_m)\,d\xi_1\dots d\xi_{|T|},
\]
If $f$ also depends on variable $t$, we include time dependence in a definition in a natural way.
\end{defi}

The definition above is given formally, since at the moment, we are not able to say whether $\tau$ is well defined for $w$ being merely a graphon. One may easily see that in case $w \in L_\xi^{\infty}((0,1);L_\zeta^{1}(0,1))\cap L_\zeta^{\infty}((0,1);L_\xi^{1}(0,1))$ the definition of $\tau$ is correct. The proper extension to the case $w\in L^\infty_\xi\mathcal{M}_\zeta \cap L^\infty_\zeta\mathcal{M}_\xi$ is a substantial part of \cite[Chapter 5]{JPS2025}. Below we present one of the main results from \cite{JPS2025}, which not only justifies the general definition of $\tau$ but also states the key convergence result. 

\begin{theo}\cite[Theorem 5.1]{JPS2025}\label{lg}
The definition of $\tau(T,w,f)$ can be uniquely extended for any $w\in L^\infty_\zeta\mathcal{M}_\xi \cap L^\infty_\xi\mathcal{M}_\zeta$ and $f\in L^\infty((0,1);(L^{1}\cap L^\infty)(\R^d))$. Furthermore, for any two sequences $\{w_N\}_{N\in \mathbb{N}}$ and $\{f_N\}_{N\in \mathbb{N}}$ such that 
    \begin{enumerate}
    \item $\quad \displaystyle \sup_{N\in \mathbb{N}} \sup_{\xi\in [0,\ 1]} \int_0^1 |w_N(\xi,\zeta)|\,d\zeta\, <\infty, \quad \sup_{N\in \mathbb{N}} \sup_{\zeta\in [0,\ 1]} \int_0^1 |w_N(\xi,\zeta)|\,d\xi\, <\infty,$\\
    \item $\quad \displaystyle\sup_{N\in \mathbb{N}} \|f_N\|_{L^\infty((0,1); (W^{1,1}\cap W^{1,\infty})(\R^d))}<\infty$,\\
   \end{enumerate}
 there exists an extracted subsequence (still denoted by $N$) such that for every $T\in \Tree$ and every $1\leq p<\infty$.
    \[
\tau(T,w_N,f_N)\rightarrow \tau(T,w,f)\quad \mbox{in}\quad L^p_{loc}(\mathbb{R}^{d\,|T|}),
    \]
as $N\rightarrow \infty$, for some  $w\in L^\infty_\zeta\mathcal{M}_\xi \cap L^\infty_\xi\mathcal{M}_\zeta$ and $f\in L^\infty((0,1); (W^{1,1}\cap W^{1,\infty})(\R^d))$.  
  \end{theo}

Having properly defined $\tau$ we finish this section with the following lemma.

\begin{lemma}\label{treeb}\cite[Corollary 4.16]{JPS2025}
If $w\in L^\infty_\zeta\mathcal{M}_\xi \cap L^\infty_\xi\mathcal{M}_\zeta$, $f\in L^\infty((0,1), (L^1\cap L^{\infty})(\R^d
))$ and $T\in \Tree$, then $\tau(T,w,f)\in (L^1\cap L^{\infty})(\R^{d |T|})$ with the estimates
$$\Vert \tau(T,w,f)\Vert_{L^1(\R^{d|T|})}\leq \Vert w\Vert^{|T|-1}\Vert f\Vert_{L^\infty((0,1), L^1(\R^d
))}^{|T|}, \hd \Vert \tau(T,w,f)\Vert_{L^\infty(\R^{d|T|})}\leq \Vert w\Vert^{|T|-1}\Vert f\Vert_{L^\infty((0,1), L^\infty(\R^d
))}^{|T|}.$$
Furthermore, if $f$ belongs additionally to $L^\infty((0,1); W^{k,p}(\R^d))$ for some $k\geq 1$ and $p\in [1,\infty]$, then $\tau(T,w,f)\in W^{k,p}(\R^{d|T|})$ and
$$\Vert \tau(T,w,f)\Vert_{W^{k,p}(\R^{d|T|})}\leq \Vert w\Vert^{|T|-1}\,\Vert f\Vert_{L^\infty((0,1); W^{k,p}(\R^d))}^{\vert T\vert}.$$
\end{lemma}

\end{subsection}

\begin{subsection}{The analysis of the limiting equation (\ref{meana}).}

In this chapter we follow the path marked in Chapter 4 in \cite{JPS2025}. First we establish the solvability of the limiting equation. In order to apply in subsequent results the regularizing approximation we add the dissipative term with $\nu \geq 0$. Thus, we discuss the following problem
\eqq{
\partial_t f(t,x,\xi) + \divv ( f(t,x,\xi) V_1[f](t,x,\xi)) =  f(t,x,\xi) \left[A-V_2[f](t,x,\xi)\right] + \nu \Delta_x f(t,x,\xi),
}{meannu}
where for $i=1,2$
\[
V_i[f](t,x,\xi) = \izj \ird K_{i}(x-y)f(t,y,\zeta)dy w(\xi,d\zeta).
\]
Similarly to \cite[Definition 4.10]{JPS2025} we define a weak solution to (\ref{meannu}) as follows. We say that $f$ is a weak solution to \eqref{meannu} with $w \in L^{\infty}_\xi\mathcal{M}_\zeta \cap L^{\infty}_\zeta\mathcal{M}_\xi$,  $K_i\in L^\infty(\R^d)$ for $i=1, 2$ and any $\nu\geq 0$ if $f\in L^{\infty}((0,t_{*})\times (0,1);L^{1}(\R^d))$ and
\[
\int_0^{t_*}\int_{\mathbb{R}^d}\partial_t \varphi(t,x)\,f(t,x,\xi)\,dxdt +\int_0^{t_*}\int_{\mathbb{R}^d}\nabla_x\varphi(t,x)\cdot V_{1}[f](t,x,\xi)\,f(t,x,\xi)\,dxdt
\]
\[
+\nu\int_0^{t_*}\int_{\mathbb{R}^d}\Delta_x\varphi(t,x)\,f(t,x,\xi)\,dxdt
+\int_0^{t_*}\int_{\mathbb{R}^d}\varphi(t,x) [A-V_{2}[f](t,x,\xi)]\,f(t,x,\xi)\,dxdt =0, \hd a.e \m{ on }  (0,1),
\]
for any $\varphi\in W^{1,\infty}((0,\ t_*);C_c^2(\mathbb{R}^d)) \cap C_{c}((0,\ t_*);C_c^2(\mathbb{R}^d))$.
Analogously to \cite[Proposition 4.11]{JPS2025} we establish the existence of weak solutions.  
\begin{prop}
Let us assume $K_i\in L^1(\R^d)\cap L^\infty(\R^d)$ for $i=1,2$, $K_1 \in W^{1,1}(\R^d)$ $\divv K_1\in L^\infty(\R^d)$, $w\in L^\infty_\xi \mathcal{M}_\zeta\cap L^\infty_\zeta \mathcal{M}_\xi$, $w \geq 0
$, $K_2 \geq 0$ and  $f^0\in L^\infty((0, 1); W^{1,1}\cap W^{1,\infty}(\R^d))$ with $f^0 \geq 0$. Then for any $\nu\geq 0$, there exists nonnegative $f\in L^\infty((0, t_*)\times (0, 1); W^{1,1}\cap W^{1,\infty}(\R^d))$, which is weak$^*$ continuous in time such that  $f(t,\cdot,\cdot)\Big|_{t=0} = f^0$ in $L^\infty((0,\ 1);\ W^{1,1}\cap W^{1,\infty}(\R^d))$ and that $f$ is a weak solution to \eqref{meannu}.\label{existenceprop}
  \end{prop}
The proof of Proposition \ref{existenceprop} is the adaptation of the argumentation from  \cite[Proposition 4.11]{JPS2025} to the non-conservative problem. For the sake of completeness we provide it in the Appendix.
Let us now prove that if $f$ is a weak solution to (\ref{meannu}), then the family $\tau(T,w,f)$ solves, so called, Vlasov hierarchy. 

\begin{prop}\label{hier}
Let us assume that $K_1, K_2\in L^\infty(\R^d)$, $\nu\geq 0$ and $w\in L^\infty_\xi \mathcal{M}_\zeta \cap L^\infty_\zeta \mathcal{M}_\xi$. Let $f\in L^\infty((0,\ t_*)\times (0,1);\, L^1(\R^d))$ be a weak solution to \eqref{meannu}. Then, $\tau(T,w,f)$ solves the following  non-exchangeable Vlasov hierarchy
\[
\partial_t \tau(T,w,f)+\sum_{i=1}^{|T|} \divv_{x_i}\left(\int_{\R^d} K_1(x_i-y)\,\tau(T+i,w,f)(t,x_1,\ldots,x_{|T|},y) dy\right)
\]
\eqq{
=A|T|\tau(T,w,f) - \sum_{i=1}^{|T|}\ird K_2(x_i-y)\tau(T+i,w,f)(t,x_1,\ldots,x_{|T|},y)dy +\nu\sum_{i=1}^{|T|}\Delta_{x_i}\tau(T,w,f).
}{treeeq}
for any tree $T\in \Tree$ in the sense of distributions.
\end{prop}
\begin{proof}
On the formal level the proof is a straightforward extension of \cite[Proposition 4.17]{JPS2025}. Below we provide a rigorous argument for $w \in L^{\infty}_\xi L^1_\zeta \cap L^{\infty}_\xi L^1_\zeta$. We remark that, in case of $w\in L^\infty_\xi \mathcal{M}_\zeta \cap L^\infty_\zeta \mathcal{M}_\xi$ one shall represent the observable $\tau$ with the help of countable algebra of operators, as in Definition 5.5. and Lemma 5.4 in  \cite{JPS2025} and prove the result by induction on the size of a tree.  

First, note that the weak solution to (\ref{meannu}) is also a distributional solution. Let us convolve in space the distributional time derivative of $f$ with a standard mollifying kernel $\eta_{\ve} \in C^{\infty}_c(\R^d)$ and denote $f^\ve:=f*\eta_\ve$. Then, $f^\ve$ satisfy 
\eqq{
\partial_t f^\ve = -\divv (\eta_{\ve}*(fV_1[f])) + \nu \Delta f^{\ve} +\eta_{\ve}*(f(A-V_2[f]))
}{aprof}
in the sense of distributions and thus $\partial_t f^\ve \in L^{\infty}((0,t_{*})\times (0,1);C^{\infty}(\R^d))$.
Moreover, applying Lemma~\ref{treeb} we obtain $\tau(T,w,f)\in L^\infty((0, t_*);L^1(\R^{d |T|}))$. 
Fix $T \in \Tree$ and take arbitrary $\Phi \in C^{1}_{c}((0,t_{*}))$ and a sequence $(\phi^{i})_{i=1}^{|T|} \in C^{2}_{c}(\R^d)$. Then, as $\ve \rightarrow 0$
\begin{align*}
    \int_0^{t_*}\int_{\R^{d|T|}}\tau(T,w,f^\ve)(t,x_{1},\dots, x_{|T|})\Phi'(t)&\prod_{i=1}^{|T|}\vf^i(x_i) dx_{1}\dots dx_{|T|}dt \longrightarrow \\
    &\int_0^{t_*}\int_{\R^{d|T|}}\tau(T,w,f)(t,x_{1},\dots, x_{|T|})\Phi'(t)\prod_{i=1}^{|T|}\vf^i(x_i) dx_{1}\dots dx_{|T|}dt,
\end{align*}
where we applied the properties of mollification, i.e. $f^\ve \rightarrow f$ in $L^{1}(\R^d)$ a.e. on $(0,t_{*})\times (0,1)$ and $\norm{f^{\ve}}_{L^{\infty}((0,t_{*})\times(0,1);L^1(\R^d))} \leq \norm{f}_{L^{\infty}((0,t_{*})\times (0,1);L^1(\R^d))}$. 
Since, $f^\ve$ is Lipschitz in time we may calculate for almost all $t$
\[
\partial_t\tau(T,w,f^\ve)(t,x_1,\dots x_{|T|})=\sum_{n=1}^{|T|}\int_{[0,\ 1]^{|T|}} \prod_{(k,l)\in E(T)} w(\xi_k,\xi_l)\,\partial_t f^\ve (t,x_n,\xi_n)\,\prod_{m\neq n} f^{\ve}(t,x_m,\xi_m)\,d\xi_1\dots d\xi_{|T|}.
 \]
Using \eqref{aprof} we arrive at
\begin{align*}
\nic{
\partial_t\tau(T,w,f^\ve)=
&-\sum_{n=1}^{|T|}\int_{[0,\ 1]^{|T|}} \prod_{(k,l)\in E(T)} w(\xi_k,\xi_l)\,\divv_{x_n} \Bigg(\eta_{\ve}*\left(f(t,x_n,\xi_n)\int_0^1\int_{\R^d} K_1(x_n-y)\,w(\xi_n,\zeta)\, \,f(t,y,\zeta)\,dy\,d\zeta\right)\Bigg)\\
&\qquad \times \prod_{m\neq n} f^{\ve}(t, x_m,\xi_m)\,d\xi_1\dots d\xi_{|T|}       \\
&\quad+\nu\sum_{n=1}^{|T|}\,\int_{[0,\ 1]^{|T|}} \prod_{(k,l)\in E(T)} w(\xi_k,\xi_l)\,\Delta_{x_n} f^{\ve} (t,x_n,\xi_n)\,\prod_{m\neq n} f^\ve(t, x_m,\xi_m)\,d\xi_1\dots d\xi_{|T|}\\
& -\sum_{n=1}^{|T|}\int_{[0,\ 1]^{|T|}} \prod_{(k,l)\in E(T)} w(\xi_k,\xi_l)\,\eta_{\ve} *\Bigg(f(t,x_n,\xi_n)\int_0^1\int_{\R^d} K_2(x_n-y)\,w(\xi_n,\zeta)\, \,f(t,y,\zeta)\,dy\,d\zeta\Bigg)\\
&\qquad \times \prod_{m\neq n} f^{\ve}(t, x_m,\xi_m)\,d\xi_1\dots d\xi_{|T|}       \\
& + A \sum_{n=1}^{|T|}\int_{[0,\ 1]^{|T|}} \prod_{(k,l)\in E(T)} w(\xi_k,\xi_l)f^{\ve}(t,x_n,\xi_n)\prod_{m \neq n} f^{\ve}(t, x_m,\xi_m)\,d\xi_1\dots d\xi_{|T|}   \\
}
&\partial_t\tau(T,w,f^\ve)=
-\sum_{n=1}^{|T|}\int_{[0,\ 1]^{|T|}} \prod_{(k,l)\in E(T)} w(\xi_k,\xi_l)
\divv_{x_n} (\eta_{\ve}*(fV_1[f]))(t,x_n,\xi_n)\prod_{m \neq n} f^{\ve}(t, x_m,\xi_m)\,d\xi_1\dots d\xi_{|T|} 
\\
&+\sum_{n=1}^{|T|}\int_{[0,\ 1]^{|T|}} \prod_{(k,l)\in E(T)} w(\xi_k,\xi_l) \Big[ \nu \Delta_{x_n} f^{\ve}(t,x_n,\xi_n) +\eta_{\ve}*(f(A-V_2[f]))(t,x_n,\xi_n)\Big]\prod_{m \neq n} f^{\ve}(t, x_m,\xi_m)\,d\xi_1\dots d\xi_{|T|} \\
=&-\sum_{n=1}^{|T|}\int_{[0,\ 1]^{|T|+1}}\divv_{x_{n}}\Bigg(\eta_{\ve}*\left(\prod_{(k,l)\in E(T)}w(\xi_k,\xi_l)\int_{\R^d} K_1(\cdot-y)\,w(\xi_n,\zeta)\, f (t,\cdot,\xi_n)\,f(t,y,\zeta)\,dy\,\right)(x_n)\Bigg)\\
&\times \prod_{m\neq n}  f^{\ve}(t,x_m,\xi_m)\,d\xi_1\dots d\xi_{|T|}d\zeta+\nu\sum_{n=1}^{|T|}\Delta_{x_n} \tau(T,w,f^\ve) + A|T|\tau(T,w,f^\ve) \\
& -\sum_{n=1}^{|T|}\int_{[0,\ 1]^{|T|+1}}\eta_{\ve}*\left(\prod_{(k,l)\in E(T)}w(\xi_k,\xi_l)\int_{\R^d} K_2(\cdot-y)\,w(\xi_n,\zeta)\, f (t,\cdot,\xi_n)\,f(t,y,\zeta)\,dy\,\right)(x_n)\\
& \times \prod_{m\neq n} f^{\ve}(t, x_m,\xi_m)\,d\xi_1\dots d\xi_{|T|} d\zeta. 
\end{align*}
Inserting this in a weak formulation for $\tau(T,w,f^\ve)$ we get
\begin{align*}
    &\int_0^{t_*}\int_{\R^{d|T|}}\tau(T,w,f^\ve)(t,x_{1},\dots , x_{|T|})\Phi'(t)\prod_{i=1}^{|T|}\vf^i(x_i) dx_{1}\dots dx_{|T|}dt=  \\
&-\sum_{n=1}^{|T|}\int_0^{t_*} \Phi(t) \int_{\R^{d|T|}}\nabla \vf^n(x_n)\int_{[0,\ 1]^{|T|+1}} \eta_{\ve}*\left(\prod_{(k,l)\in E(T)}w(\xi_k,\xi_l)\int_{\R^d} K_1(\cdot-y)\,w(\xi_n,\zeta)\, f (t,\cdot,\xi_n)\,f(t,y,\zeta)\,dy\,\right)(x_n)  \\
&\times \prod_{m\neq n} f^{\ve}(t, x_m,\xi_m)\,d\xi_1\dots d\xi_{|T|}d\zeta\prod_{i\neq n}^{|T|}\vf^i(x_i) dx_{1}\dots dx_{|T|}dt\\
&-\nu \sum_{n=1}^{|T|}\int_0^{t_*}\Phi(t)\int_{\R^{d|T|}}\tau(T,w,f^\ve)(t,x_{1},\dots , x_{|T|})\Delta \vf^n(\xi_n)\prod_{i\neq n}^{|T|}\vf^i(x_i) dx_{1}\dots dx_{|T|}dt \\
&- A|T|\int_0^{t_*}\Phi(t)\int_{\R^{d|T|}}\tau(T,w,f^\ve)(t,x_{1},\dots , x_{|T|}) \prod_{i=1}^{|T|}\vf^i(x_i) dx_{1}\dots dx_{|T|}dt\\
&+\sum_{n=1}^{|T|}\int_0^{t_*} \Phi(t) \int_{\R^{d|T|}}\int_{[0,\ 1]^{|T|+1}} \eta_{\ve}*\left(\prod_{(k,l)\in E(T)}w(\xi_k,\xi_l)\int_{\R^d} K_2(\cdot-y)\,w(\xi_n,\zeta)\, f (t,\cdot,\xi_n)\,f(t,y,\zeta)\,dy\right)(x_n)  \\
&\times \prod_{m\neq n} f^{\ve}(t, x_m,\xi_m)\,d\xi_1\dots d\xi_{|T|}d\zeta \prod_{i=1}^{|T|}\vf^i(x_i) dx_{1}\dots dx_{|T|}dt.
\end{align*}
Due to the properties of convolution with mollifying kernel we may pass to the limit with $\ve$. Then, we 
recall that for any $n=1,\ldots,|T|$ the tree $T+n$ contains exactly the same edges as $T$ plus a new edge from the vertex $n$ to the new vertex $\vert T\vert+1$. Thus, recalling that we only consider edges $(k,l)\in E(T+n)$ that run from the root to the leaves, we have that
\[
\tau(T+n)=\int_{[0,\ 1]^{|T|+1}} \, w(\xi_n, \xi_{|T|+1})\prod_{(k,l)\in E(T)} w(\xi_k,\xi_l)\;f(t,x_{|T|+1},\xi_{|T|+1})\prod_{m\in V(T)} f(t,x_m,\xi_m)\,d\xi_1\ldots\,d\xi_{|T|+1}.
\]
Changing variables $\xi_{|T|+1}$ with $\zeta$ and $x_{|T|+1}$ with $y$, we arrive at
\begin{align*}
&\int_0^{t_*}\int_{\R^{d|T|}}\tau(T,w,f)(t,x_{1},\dots , x_{|T|})\Phi'(t)\prod_{i=1}^{|T|}\vf^i(x_i) dx_{1}\dots dx_{|T|}dt=  \\
&-\sum_{n=1}^{|T|}\int_0^{t_*} \Phi(t) \int_{\R^{d|T|}}\nabla \vf^n(x_n)\int_{\R^d} K_1(x_n-y)\tau(T+n,w,f)(t,x_1,\dots,x_{|T|},y)\,dy\,\prod_{i\neq n}^{|T|}\vf^i(x_i) dx_{1}\dots dx_{|T|}dt \\
&-\nu \sum_{n=1}^{|T|}\int_0^{t_*}\Phi(t)\int_{\R^{d|T|}}\tau(T,w,f)(t,x_{1},\dots , x_{|T|})\Delta \vf^n(\xi_n)\prod_{i\neq n}^{|T|}\vf^i(x_i) dx_{1}\dots dx_{|T|}dt \\
&- A|T|\int_0^{t_*}\Phi(t)\int_{R^{d|T|}}\tau(T,w,f)(t,x_{1},\dots , x_{|T|}) \prod_{i=1}^{|T|}\vf^i(x_i) dx_{1}\dots dx_{|T|}dt\\
&+\sum_{n=1}^{|T|}\int_0^{t_*} \Phi(t) \int_{\R^{d|T|}}\int_{\R^d} K_2(x_n-y)\tau(T+n,w,f)(t,x_1,\dots,x_{|T|},y)dy\,\prod_{i=1}^{|T|}\vf^i(x_i) dx_{1}\dots dx_{|T|}dt.
\end{align*}
By the standard density argument we obtain that $\tau(T,w,f)$ indeed satisfy (\ref{treeeq}) in the sense of distributions.
\end{proof}


Our next target is the proof  of stability estimate for weak solutions to (\ref{meannu}) with $\nu = 0$, with respect to the initial condition and $w$ (Lemma \ref{mainstab}). However, we firstly obtain the stability result 
for the solutions $f^\nu$ to the regularized problem with $\nu > 0$ and then comparing $f$ and $f^\nu$ we optimize $\nu$. To obtain the result for $f^\nu$ we actually study the stability of  (\ref{treeeq}).
The results below are the adaptations of Theorem 4.19, Lemma 4.22 and Theorem 4.23 from \cite{JPS2025} to non-conservative setting. Analogously to \cite{JPS2025}, we formulate the first result in terms of generic distribution solutions  $h=(h_T)_{T\in \Tree}$ with $h_T\in L^\infty(
(0, t_*);L^1(\R^{d\vert T\vert}))$ for any $T\in \Tree$ to the  hierarchy of  non-exchangeable Vlasov equations
\begin{align}\label{hiergeneq}
 &\partial_t h_T+\sum_{i=1}^{|T|} \mbox{div}_{x_i}\left(\int_{\R^d} K_1(x_i-y)\,h_{T+i}(t,x_1,\ldots,x_{|T|},y)\,dy\right)\\ \nonumber 
 =&A|T|h_T -\sum_{i=1}^{|T|} \int_{\R^d} K_2(x_i-y)\,h_{T+i}(t,x_1,\ldots,x_{|T|},y)\,dy +  \nu\,\sum_{i=1}^{|T|} \Delta_{x_i} h_T,
\end{align}
with $K_j\in L^\infty(\R^d)$ for $j=1,2$, $K_2 \geq 0$ and $\nu>0$.
Subsequently, we apply this general result to the difference $\tau(T,w,f^\nu)-\tau(T,\tilde{w},\tilde{f}^\nu)$, $T \in \mathcal{T}$, where $(f^\nu,w)$ and $(\tilde{f}^\nu,\tilde{w})$ satisfy (\ref{meannu}) with different initial conditions. Note, that we do not claim any solvability result regarding (\ref{hiergeneq}). Let us introduce the norm of the hierarchy in which our stability result holds. 

\begin{defi}{\cite[Definition 4.18]{JPS2025}}
Let $h_T\in L^2(\R^{d |T|})$ for every $T\in \Tree$ and  set $h=(h_T)_{T\in \Tree}$. Then,  we define the family of norms:
\[
 \|h\|_{\lambda}=\sup_{T\in \Tree} \lambda^{|T|/2}\,\|h_T\|_{L^2(\R^{d\,|T|})}, \hd \lambda>0.
\]
\end{defi}
Note that in our application $h_T=\tau(T,w,f)$ the choice of the norm defined above seems particularly natural. Indeed, if $f\in L^2(\R^d)$ applying Lemma~\ref{treeb} we obtain
\[
\Vert h\Vert_\lambda\leq \sup_{T\in \Tree} \lambda^{|T|/2}\Vert w\Vert_{L^\infty_\xi\mathcal{M}_\zeta}^{|T|-1}\Vert f\Vert_{L^2(\R^d)}^{|T|},
\]
for every $\lambda>0$. Thus,  choosing $\lambda  < \|w\|^{-2\frac{|T|-1}{|T|}}_{L^\infty_\xi\mathcal{M}_\zeta} \|f\|^{-2}_{L^2(\R^d)}$, we get $\norm{h}_{\lambda} < \infty$. 
%
The already announced stability result of the solutions to (\ref{hiergeneq}) reads as follows. 
\begin{lemma} \label{hlambdabound}
Let $h=(h_T)_{T\in \Tree}$ be a distributional solution to \eqref{hiergeneq} with  $\nu>0$, $K_1, K_2\in L^2(\R^d)$ and $K_2 \geq 0$. Additionally, assume that $h_T\in L^\infty((0,t_{*});(L^1\cap L^2)(\R^{d |T|}))$ for any $T\in \Tree$ and that there exists $\lambda>0$ such that $C_\lambda:=\sup_{t\in (0,\ t_*)}\|h(t,\cdot)\|_{\lambda}<\infty$. Then, for every $p>1$ and every $\theta\in (0,2^{-p'}e^{-p'(2A+1)t_*})$ with $p'=\frac{p}{p-1}$, there exists a constant $C > 0$ dependent only on $p,A,t_*,\theta$ such that 
\[
\Vert h(t,\cdot)\Vert_{\theta\lambda}\leq \max\{C_{\lambda},1\}C_\lambda C\exp\left(p^{-\frac{\left((2\nu)^{-1}\norm{K_1}_{L^{2}(\R^d)}^2  + \norm{K_2}_{L^{2}(\R^d)}^2\right)}{\theta\lambda}t}\,\log\frac{\Vert h(0,\cdot)\Vert_{\theta \lambda}}{C_\lambda}\right),
\]
for every $t\in [0, t_*]$.
\end{lemma}

\begin{proof}
We follow the proof of \cite[Theorem 4.19]{JPS2025} introducing the necessary adaptations to cover non-conservative problem.
First we note that, under the assumptions of the lemma, if $h_T$ is a distributional solution to (\ref{hiergeneq}) with $\nu>0$, then it actually is a weak solution with $\partial_th_T \in L^{2}((0,t_*);H^{-1}(\R^d)), \nabla h_T \in L^{2}((0,t_*);L^2(\R^d))$ and it satisfies for every $\Phi \in H^{1}(\R^{d|T|})$ and almost all $t \in (0,t_*)$
\[
\left\langle \Phi, \partial_t h_T \right\rangle_{H^{1}(\R^{d|T|}) \times H^{-1}(\R^{d|T|})} + \nu \sum_{i=1}^{|T|}\int_{\R^{d|T|}}\nabla_{x_{i}}h_{T} \cdot \nabla_{x_{i}}\Phi dx_{1}\dots d_{x_{|T|}} = \left\langle\Phi, g_T\right\rangle_{H^{1}(\R^{d|T|}) \times H^{-1}(\R^{d|T|})},
\]
where
\[
g_T:= -\sum_{i=1}^{|T|} \divv_{x_i}\left(\int_{\R^d} K_1(x_i-y)\,h_{T+i}(t,x_1,\ldots,x_{|T|},y) dy\right)+A|T|h_T - \sum_{i=1}^{|T|}\ird K_2(x_i-y)h_{T+i}(t,x_1,\ldots,x_{|T|},y)dy. 
\]
Indeed, $g_T \in H^{-1}(\R^{d|T|})$, since from H\"older inequality
  \eqq{
  \int_{\R^{d|T|}}\abs{\ird K_{j}(x_i-y)h_{T+i}(t,x_1,\dots,x_{T},y)dy}^{2}dx_{1}\dots d x_{T}\leq \Vert K_j\Vert_{L^2(\R^d)}^2\Vert h_{T+i}(t,\cdot)\Vert^2_{L^2(\R^{d(|T|+1)})} \hd \m{ for } j=1,2.}{hnj}
Hence, $h_T$ is a distributional solution to the Poisson equation with the source term in $L^{\infty}((0,t_{*});H^{-1}(\R^{d|T|})$. By the theory of parabolic equations it is actually a weak solution with the regularity described above.

Thus, by standard approximation  by step functions we may use  $h_T$ as a test function, then
\begin{align*}
&\frac{d}{dt}\int_{\R^{d\,|T|}} |h_T|^2\,dx_1\dots dx_{|T|} + 2\,\nu\,\sum_{i=1}^{|T|}\int_{\R^{d\,|T|}} |\nabla_{x_i}h_T|^2\,dx_1\dots dx_{|T|} \leq  \nonumber \\
&   2\,\sum_{i=1}^{|T|}\left(\int_{\R^{d|T|}} \abs{\nabla_{x_i} h_{T}}^2\ dx_1\dots dx_{|T|}\right)^{\frac{1}{2}}\left(\int_{\R^{d|T|}} \abs{\ird K_{1}(x_i-y)h_{T+i}dy}^2\ dx_1\dots dx_{|T|}\right)^{\frac{1}{2}} \nonumber \\
&+2AT\int_{\R^{d|T|}} \abs{ h_{T}}^2\ dx_1\dots dx_{|T|} \\
&+  2\,\sum_{i=1}^{|T|}\left(\int_{\R^{d|T|}} \abs{ h_{T}}^2\ dx_1\dots dx_{|T|}\right)^{\frac{1}{2}}\left(\int_{\R^{d|T|}} \abs{\ird K_{2}(x_i-y)h_{T+i}dy}^2\ dx_1\dots dx_{|T|}\right)^{\frac{1}{2}}.
\end{align*}
Applying (\ref{hnj}) and Young inequality we arrive at
\[
 \frac{d}{dt}\|h_T(t,\cdot)\|_{L^2(\R^{d\,|T|})}^2 \leq \left(\frac{\norm{K_{1}}^2_{L^{2}(\R^d)}}{2\nu} +\norm{K_{2}}^2_{L^{2}(\R^d)}\right)\sum_{i=1}^{|T|} \norm{h_{T+i}(t,\cdot)}_{L^2(\R^{d\,(|T|+1)})}^2 + (2A+1)T \norm{h_T(t,\cdot)}_{L^2(\R^{d\,|T|})}^2.
\]
Thus, by Gr\"onwall inequality
\begin{align}\label{haa}
  \norm{h_T(t,\cdot)}_{L^2(\R^{d\,|T|})}^2 & \leq e^{(2A+1)|T|(t-s)}\norm{h_T(s,\cdot)}_{L^2(\R^{d\,|T|})}^2\\ \nonumber
 &  + \left(\frac{\norm{K_{1}}^2_{L^{2}(\R^d)}}{2\nu}+\norm{K_{2}}^2_{L^{2}(\R^d)}\right)\int_s^t e^{(2A+1)|T|(t-r)} \sum_{i=1}^{|T|} \norm{h_{T+i}(r,\cdot)}_{L^2(\R^{d\,|T|})}^2 dr.
\end{align}

Define the family of  norms
\[
\abs{h(t,\cdot)}_{n}:=\sup_{|T|=n} \|h_T(t,\cdot)\|_{L^2(\R^{d\,|T|})},
\]
for any $n\in \mathbb{N}$.
Then, estimating firstly the RHS of (\ref{haa}) by the appropriate newly defined  norms and then taking supremum over $|T|=n$ we arrive at 
\eqq{
 \abs{h(t,\cdot)}_{n}^2\leq e^{(2A+1)n(t-s)}\abs{h(s,\cdot)}_n^2 + n \left(\frac{\norm{K_{1}}^2_{L^{2}(\R^d)}}{2\nu}+\norm{K_{2}}^2_{L^{2}(\R^d)}\right)\,\int_{s}^t e^{(2A+1)n
 (t-r)}\abs{h(r,\cdot)}_{n+1}^2 dr.
}{hit}
Let us introduce the following notation:
\[
\vf^n(t):=  e^{-(2A+1)nt}\abs{h(t,\cdot)}_{n}^2, \hd C(t):= \left(\frac{\norm{K_{1}}^2_{L^{2}(\R^d)}}{2\nu}+\norm{K_{2}}^2_{L^{2}(\R^d)}\right)e^{(2A+1)t}.
\]
Then the estimate (\ref{hit}) leads to
\[
\vf^n(t) \leq \vf^n(s) + n C(t)\int_s^{t}\vf^{n+1}(r)dr \m{ for every } n \in \mathbb{N}.
\]
Applying this inequality for index $n+1$ and using the fact that $C(t)$ is increasing we obtain
\[
\vf^n(t) \leq \vf^n(s) +nC(t)\vf^{n+1}(s)(t-s) +n(n+1) C^2(t)\int_s^{t}\vf^{n+2}(r)(t-r)dr.
\]
Iterating this estimate further we obtain for every $m>n$ 
\[
\vf^n(t) \leq (C(t))^{m-n} \frac{(m-1)!\,}{(n-1)!\,(m-n-1)!}\,\int_{s}^t \vf^{m}(r)(t-r)^{m-n-1}dr + \sum_{k=n}^{m-1} \vf^{k}(s) (C(t))^{k-n} \frac{(k-1)!\,(t-s)^{k-n}}{(n-1)!\,(k-n)!}.
\]
Coming back to original notation with $\vf^n$ and leaving $C(t)$ to keep the notation concise we arrive at
\begin{align*}
  &e^{-(2A+1)nt}\abs{h(t,\cdot)}_{n}^2\leq \left(C(t)\right)^{m-n}\,\frac{(m-1)!\,}{(n-1)!\,(m-n-1)!}\,\int_{s}^t\,e^{-(2A+1)mr}(t-r)^{m-n-1}  \abs{h(r,\cdot)}_{m}^2\,dr\\
  &+\sum_{k=n}^{m-1}  \left(C(t)\right)^{k-n}\, \frac{(k-1)!\,(t-s)^{k-n}}{(n-1)!\,(k-n)!}\, e^{-(2A+1)ks}\abs{h(s,\cdot)}_{k}^2
\end{align*}
for any $m>n$. From the assumption we infer that there exists $\lambda > 0$ such that 
\[
\abs{h(r,\cdot)}^2_m\leq C^2_{\lambda}\lambda^{-m},
\]
for any $r\in [s,\ t]$. Hence, estimating the exponents raised to a negative power by one and calculating the integral we obtain
\[
  \begin{split}
e^{-(2A+1)nt}\abs{h(t,\cdot)}_{n}^2&\leq \left(C(t_*)\right)^{m-n}\,\binom{m-1}{n-1}\,(t-s)^{m-n}\, C^2_{\lambda} \lambda^{-m} \\
&+\sum_{k=n}^{m-1} \left(C(t_*)\right)^{k-n}\,\binom{k-1}{n-1} \,(t-s)^{k-n}\, \abs{h(s,\cdot)}_{k}^2.
\end{split}
\]
Choose any $p>1$, any $\theta\in (0,2^{-p'}e^{-p'(2A+1)t_*})$ and fix a time step
\[
\delta:=\frac{\theta\lambda e^{(2A+
1)t_*}}{C(t_*)}.
\]
Then for every $i=0,\ldots,N$ set $t_i=\delta\,i$. Taking in the estimate above $t=t_{i+1}$ and $s=t_i$, we get

\[
\sup_{t\in [t_i,\ t_{i+1}]}e^{-(2A+1)nt_{*}}\abs{h(t,\cdot)}_{n}^2\leq C^2_{\lambda} \theta^{m-n} e^{(2A+1)t_*(m-n)}\lambda^{-n}\,\binom{m-1}{n-1}+\sum_{k=n}^{m-1}  (\theta\lambda e^{(2A+1)t_*})^{k-n} \binom{k-1}{n-1} \abs{h(t_i,\cdot)}_{k}^2.
\]
Estimating $\binom{m-1}{n-1}\leq 2^{m-1}$, which follows from $\sum_{n=1}^m \binom{m-1}{n-1}=2^{m-1}$, we arrive at
\[
\sup_{t\in [t_i,\ t_{i+1}]}(\theta\lambda)^n\abs{h(t,\cdot)}_{n}^2\leq \frac{1}{2} C^2_{\lambda} (2\theta e^{(2A+1)t_*})^m +\frac{1}{2}\sum_{k=n}^{m-1}  (2\theta\lambda e^{2(A+1)t_*})^k\abs{h(t_i,\cdot)}_{k}^2.
\]
Young estimate gives
\[
(\theta^{1/p}\lambda)^k\,\abs{h(t_i,\cdot)}_{k}^2= \left((\theta\lambda)^k\,\abs{h(t_i,\cdot)}_{k}^2\right)^{1/p}\, \left(\lambda^k\,\abs{h(t_i,\cdot)}_{k}^2\right)^{1/p'}\leq \Vert h(t_i,\cdot)\Vert_{\theta\lambda}^{2/p}\,\Vert h(t_i,\cdot)\Vert_{\lambda}^{2/p'}\leq  \Vert h(t_i,\cdot)\Vert_{\theta\lambda}^{2/p}C_{\lambda}^{\frac{2}{p'}}.
\]
Applying this estimate and recalling that due to assumptions $2\theta^{1/p'}e^{(2A+1)t_*}<1$ we obtain
\begin{align*}
   \sup_{t\in [t_i,\ t_{i+1}]}(\theta\lambda )^n\abs{h(t,\cdot)}_{n}^2&\leq\frac{C^2_{\lambda}}{2}(2\theta e^{2(A+1)t_*})^m\,+\frac{C_{\lambda}^{\frac{2}{p'}}}{2}\sum_{k=n}^{m-1}  \left(2\theta^{1/p'}e^{(2A+1)t_*}\right)^{k}\,  \|h(t_i,\cdot)\|_{\theta\lambda}^{2/p} \\
   & \leq \frac{C^2_{\lambda}}{2}(2\theta e^{2(A+1)t_*})^m +\frac{C_{\lambda}^{\frac{2}{p'}}}{2\left(1-2\theta^{1/p'}e^{(2A+1)t_*}\right)}\,\|h(t_i,\cdot)\|_{\theta\lambda}^{2/p}.
\end{align*}
 Passing to the limit $m\to \infty$ and then taking supremum over $n$ yields
\[
\sup_{t\in [t_i,\ t_{i+1}]}\|h(t,\cdot)\|_{\theta\lambda}\leq \left(\frac{C_{\lambda}^{\frac{2}{p'}}}{2\left(1-2\theta^{1/p'}e^{(2A+1)t_*}\right)}\right)^{1/2}\, \|h(t_i,\cdot)\|_{\theta\lambda}^{1/p}.
\]
Then, for $t\in[t_i, t_{i+1}]$ and $\tilde{C} = \frac{C_{\lambda}^{\frac{2}{p'}}}{2\left(1-2\theta^{1/p'}e^{(2A+1)t_*}\right)}$, we may iteratively estimate
\[
\begin{split}
\|h(t,\cdot)\|_{\theta \lambda} &\leq \tilde{C}^{\frac{1}{2}} \|h(t_i,\cdot)\|_{\theta \lambda}^{1/p} \leq \tilde{C}^{\frac{1}{2}} \tilde{C}^{\frac{1}{2p}} \|h(t_{i-1},\cdot)\|_{\theta \lambda}^{1/p^2} \leq \tilde{C}^{\frac{1}{2}} \tilde{C}^{\frac{1}{2p}} \tilde{C}^{\frac{1}{2p^2}} \|h(t_{i-2},\cdot)\|_{\theta \lambda}^{1/p^3}\\
& \leq \cdots \leq \tilde{C}^{\frac{1}{2} \sum_{j=0}^{i-1} \frac{1}{p^j}} \|h(0,\cdot)\|_{\theta \lambda}^{1/p^i} \leq C_{p}^{\frac{p'}{2}}\|h(0,\cdot)\|_{\theta \lambda}^{p^{-i}}.
\end{split}
\]
Hence,
\[
\sup_{t\in [t_i,\ t_{i+1}]}\|h(t,\cdot)\|_{\theta\lambda}\leq C_{\lambda}\left(\frac{1}{2\left(1-2\theta^{1/p'}e^{(2A+1)t_*}\right)}\right)^{p'/2}\,\|h(0,\cdot)\|_{\theta\lambda}^{p^{-i}}.
\]
In order to complete the proof, for arbitrary $t\in [0,\ t_*)$ we choose $i$ such that $t\in [t_i,\ t_{i+1})$. Furthemore, note that by the assumptions $\frac{\Vert h(0,\cdot)\Vert_{\theta\lambda}}{C_{\lambda}}\leq \frac{\Vert h(0,\cdot)\Vert_{\lambda}}{C_{\lambda}}\leq 1$, hence using $i\leq \frac{t}{\delta}$
\[
\|h(0,\cdot)\|_{\theta\lambda}^{p^{-i}} = \left(\frac{\|h(0,\cdot)\|_{\theta\lambda}}{C_{\lambda}}\right)^{p^{-i}}C_{\lambda}^{p^{-i}} \leq \left(\frac{\|h(0,\cdot)\|_{\theta\lambda}}{C_{\lambda}}\right)^{p^{-\frac{t}{\delta}}}\max\{C_{\lambda},1\},
\]
which finishes the proof.
 \end{proof}

In addition to Lemma \ref{hlambdabound}, the second component in the proof of stability estimate, is the following bound of the solution to (\ref{meannu}) with $\nu = 0$.

\begin{lemma}\label{linftyestf}
Assume that $f$ is a nonnegative weak solution to \eqref{meannu} belonging to $L^{\infty}((0,t_{*})\times (0,1);(L^{1}\cap L^{\infty}\cap H^{1})(\R^d))$ with $w\in L^\infty_\xi \mathcal{M}_\zeta \cap L^\infty_\zeta \mathcal{M}_\xi$, $w \geq 0$, $\nu \geq 0$, $K_2 \in (L^1 \cap L^{\infty})(\R^d)$ $K_1\in (W^{1,1}\cap W^{1,\infty})(\R^d)$,   $K_2 \geq 0$,   and  $f^0\in L^\infty((0,1); (L^1\cap L^\infty\cap H^1)(\R^d))$.  Then, 
  \[
\|f(t,\cdot,\cdot)\|_{L^\infty((0,1); (L^1\cap L^\infty\cap H^1)(\R^d)}\leq C \exp(\exp(\tilde{C}e^{At}t)),
  \]
for some $C,\tilde{C}\in \mathbb{R}_+$ depending only on $\|f^0\|_{L^\infty((0,1); (L^1\cap L^\infty\cap H^1(\R^d)))}, \|w\|$, $\|K_1\|_{W^{1,1}(\R^d)}$, $\|\divv K_1\|_{L^\infty}$, $\norm{K_2}_{L^{1}(\R^d)}$ and $A$.
\end{lemma}
\begin{proof}

We follow the approach from the proof of \cite[Lemma 4.22]{JPS2025} with the necessary modifications. 
Note, that testing the equation by $f^p$ for positive $p$ is allowed, since $\nabla f^p \in L^2(\R^d)$, because $f$ is bounded in space. Let us test the equation then by $pf^{p-1}$ for fixed $p>1$:
\begin{align*}
    &\frac{d}{dt} \|f(t,\cdot,\xi)\|_{L^p(\R^d)}^p +\nu p(p-1) \ird \abs{\nabla f(t,x,\xi)}^{2}f^{p-2}(t,x,\xi)dx\\
    & =p(p-1)\ird f^{p-1}(t,x,\xi)\nabla f(t,x,\xi)\cdot\izj \ird K_{1}(x-y)f(t,y,\zeta)dyw(\xi,d\zeta) dx \\
    & = p\ird f^{p}(t,x,\xi)\left[A - \izj \ird  K_{2}(x-y)f(t,y,\zeta)dyw(\xi,d\zeta) \right]dx.
\end{align*}
Hence, applying integration by parts,  Lemma \ref{impoest} and recalling that $w,K_2, f \geq 0$ we arrive at
\[
\frac{d}{dt} \|f(t,\cdot,\xi)\|_{L^p(\R^d)}^p 
\leq \left((p-1)\,\norm{w}\|\divv K_1\|_{L^\infty(\R^d)}\,\| f(t,\cdot,\cdot)\|_{L^\infty((0,1); L^1(\R^d)}\, + Ap\right)\|f(t,\cdot,\xi)\|_{L^p(\R^d)}^p,
\]
for a.a. $t\in (0,\ t_*)$ and $\xi\in (0,\ 1)$.
On the other hand, since, $(f,w)$ is a nonnegative weak solution and $K_2 \geq 0$, integrating (\ref{meannu}) in space,  we may easily estimate
\eqq{
\norm{f(t,\cdot,\cdot)}_{L^{\infty}((0,1);L^{1}(\R^d)} \leq e^{At} \norm{f^0}_{L^{\infty}((0,1);L^{1}(\R^d))}, \hd \norm{f}_{L^{\infty}((0,t_*)\times (0,1);L^{1}(\R^d)} \leq e^{At_*} \norm{f^0}_{L^{\infty}((0,1);L^{1}(\R^d))}.
}{flinfty}
Hence, by Gr\"onwall inequality
\eqq{
\|f(t,\cdot,\cdot)\|_{L^{\infty}((0,1);L^p(\R^d))}^p \leq  \|f^0\|_{L^{\infty}((0,1);L^p(\R^d))}^pe^{\left([(p-1)\,\norm{w}\|\divv K_1\|_{L^\infty(\R^d)}e^{At}\norm{f^0}_{L^{\infty}((0,1);L^{1}(\R^d))} +Ap]t\right)}.
}{flinftytwo}
Rising the inequality to the power $1/p$ and passing to the limit with $p$ we arrive at
\eqq{
\norm{f(t,\cdot,\cdot)}_{L^{\infty}((0,1);L^\infty(\R^d))} \leq \norm{f^0}_{L^{\infty}((0,1);L^\infty(\R^d))} e^{Ce^{At}t},
}{finftyinfty}
with a positive constant $C$ dependent only on $\norm{w},A, \|\divv K_1\|_{L^\infty(\R^d)},\norm{f^0}_{L^{\infty}((0,1);L^{1}(\R^d))}$. In order to justify testing the equation by $\Delta f$ we mollify the equation in space by standard smoothing convolution and apply Friedrichs commutator lemma (Lemma II.1 \cite{diPerna}), to $\nabla f$ which is possible since $f \in L^{\infty}((0,t_{*})\times (0,1);H^{1}(\R^d))$ and $K_1 \in W^{1,\infty}(\R^d;\R^d)$ . For the sake of clarity we  continue formally.

\begin{align}\label{gradba}
    \frac{d}{dt} \frac{1}{2}\|\nabla f(t,\cdot,\xi)\|_{L^2(\R^d)}^2&=\frac{1}{2}\ird \divv V_1[f] \abs{\nabla f}^{2} dx-\ird \nabla^T f\cdot \nabla V_1[f] \cdot \nabla fdx - \ird \nabla f \cdot \nabla \divv V_1[f]  f dx\\ \nonumber
    &+\ird \abs{\nabla{f}}^2[A-V_2[f]] dx - \ird f \nabla V_2[f] \cdot \nabla fdx -\nu \ird \abs{\Delta f}^2dx =:
\sum_{j=1}^6 I_j.
\end{align}
Let us estimate term by term. Obviously, $I_6 \leq 0$ and since $V_2[f] \geq 0$
\[
I_4 \leq A\norm{\nabla f(t,\cdot,\xi)}^2_{L^{2}(\R^d)}.
\]
Moreover, using the properties of convolution in $I_3$ and $I_5$ we obtain 
\begin{align*}
I_1&=\frac{1}{2}\ird |\nabla f(t,x,\xi)|^2\int_0^1\int_{\R^{d}}\divv K_1(x-y)\,\,f(t,y,\zeta)dyw(\xi,d\zeta)\,dx \\
I_2&=- \ird\int_{\R^{d}} \nabla f(t,x,\xi)^\top\cdot  \int_0^1 \nabla K_1(x-y)\, f(t,y,\zeta) w(\xi,d\zeta)\,dy\,  \cdot \nabla f(t,x,\xi) dx.\\
I_3&=- \int_{\R^{d}} f(t,x,\xi)\nabla f(t,x,\xi) \cdot\ird \izj \divv K_1(x-y)\,\nabla_y f(t,y,\zeta)\,\, w(\xi,d\zeta)dy\,dx,\\
I_5&=- \int_{\R^{d}} f(t,x,\xi)\nabla f(t,x,\xi) \cdot \ird \izj  K_2(x-y)\, \nabla_y f(t,y,\zeta)\,\, w(\xi,d\zeta)dy dx.
\end{align*}

Thus applying Lemma \ref{impoest} we estimate as follows
\begin{align*}
\vert I_1\vert&\leq \frac{1}{2}\Vert w\Vert\,\Vert \divv K_1\Vert_{L^\infty(\R^d)}\,\Vert f(t,\cdot,\cdot)\Vert_{L^\infty((0,1);L^1(\R^d))}\,\Vert \nabla f(t,\cdot,\xi)\Vert_{L^2(\R^d)}^2,\\
\vert I_2\vert &\leq \Vert w\Vert\,\Vert \nabla K_1\Vert_{L^1(\R^d)}\,\Vert f(t,\cdot,\cdot)\Vert_{L^\infty((0,1);L^\infty(\R^d))}\,\Vert \nabla f(t,\cdot,\xi)\Vert_{L^2(\R^d)}^2,\\
\vert I_3\vert &\leq \Vert w\Vert\,\Vert \divv K_1\Vert_{L^1(\R^d)}\,\Vert f(t,\cdot,\cdot)\Vert_{L^\infty((0,1); L^\infty(\R^d))}\,\Vert \nabla f(t,\cdot,\cdot)\Vert_{L^{\infty}((0,1);L^2(\R^d))}^2,\\
\vert I_5\vert &\leq \Vert w\Vert\,\Vert  K_2\Vert_{L^1(\R^d)}\,\Vert f(t,\cdot,\cdot)\Vert_{L^\infty((0,1); L^\infty(\R^d))}\,\Vert \nabla f(t,\cdot,\cdot)\Vert_{L^{\infty}((0,1);L^2(\R^d))}^2.
\end{align*}
Inserting the estimates above in (\ref{gradba}) and integrating with respect to time
\begin{align*}
 \norm{\nabla f(t,\cdot,\xi)}_{L^{2}(\R^d)}^2 &\leq \norm{\nabla f^0(\cdot,\xi)}_{L^{2}(\R^d)}^2 \\   
&+ C (\norm{f}_{L^{\infty}((0,t)\times(0,1)\times \R^d)} + \norm{f}_{L^{\infty}((0,t)\times(0,1);L^{1}(\R^d)}) \cdot \int_0^t \norm{\nabla f(s,\cdot,\cdot)}_{L^{\infty}((0,1);L^2(\R^d))}^2 ds,
\end{align*}
where $C$ is a positive constant dependent only on $A,\norm{w}$ and norms of $K_1,K_2$ used above. Taking sup with respect to $\xi$, applying Gr\"onwall lemma together with the estimates (\ref{flinfty}) and (\ref{finftyinfty})
we obtain
\[
\norm{\nabla f(t,\cdot,\cdot)}_{L^{\infty}((0,1);L^{2}(\R^d))} \leq \norm{\nabla f^0}_{L^{\infty}((0,1);L^{2}(\R^d))} \exp(\exp(\tilde{C}e^{At}t)),
\]
where $\tilde{C}$ depends on the same norms of $w,K_1,K_2$ as $C$ and also on $A,\norm{f^0}_{L^{\infty}((0,1);L^{1}(\R^d))}, \norm{f^0}_{L^{\infty}((0,1);L^{\infty}(\R^d))} $. 
\end{proof}
At last, we are ready to establish the main stability estimate which is an analogue of \cite[Theorem~4.23]{JPS2025}.
\begin{lemma}\label{mainstab}
Let $(f,w)$ and $(\tilde f,\tilde{w})$, such that $f,\tilde{f} \in L^\infty((0,t_{*})\times (0,1);  (L^1\cap L^\infty\cap H^1)(\R^d))$ and $w,\;\tilde w\in L^\infty_\xi\mathcal M_\zeta \cap L^\infty_\zeta \mathcal M_\xi$, $w,\tilde{w} \geq 0$, satisfy in a weak sense \eqref{meannu} with $\nu=0$ and  nonnegative initial conditions $f^0,\tilde{f}^0$, respectively. Assume that $K_1\in (W^{1,1}\cap W^{1,\infty})(\R^d)$,  $K_2 \in L^1(\R^d)\cap L^\infty(\R^d), K_2 \geq 0$. Then,  there exist $C,\lambda > 0$ such that for almost all  $t \in (0,t_{*})$
  \[
\left\|  \int_0^1 ( f -\tilde f) (t,\cdot, \xi)\,d\xi\,\right\|_{L^2(\R^d)}\leq {\frac{C}{\sqrt{(\log |\log \|\tau(\cdot,w, f^0)-\tau(\cdot,\tilde w,\tilde f^0)\|_{\lambda}|)_+}}.}
  \]
The constants $C$ and $\,\lambda$  depend on $t_*,A, \norm{w},\norm{\tilde{w}}$,  $\norm{K_1}_{L^{2}(\R^d)}, \norm{K_1}_{W^{1,1}(\R^d)}$, $\norm{\divv K_1}_{L^{\infty}(\R^d)}, \norm{K_2}_{L^{1}\cap L^{\infty}(\R^d)}$ and the norm of the initial data $f^0,\,\tilde f^0$ in $L^\infty((0,1);(L^1\cap L^\infty\cap H^1)(\R^d))$.
 \end{lemma}
\begin{proof}
In the proof we closely follow the reasoning from the proof of \cite[Theorem 4.23]{JPS2025}.
Let us  denote by $f^\nu,\tilde{f}^\nu$ the solutions  to the system \eqref{meannu} with $
\nu  > 0$, initial data $f^0, \tilde{f}^0$ and weights $w,\tilde{w}$, respectively,  given by Proposition \ref{existenceprop}.
We will argue applying triangle inequality. In the first step, with a help of Lemma \ref{linftyestf}, we establish how close does $f(t,\cdot,\cdot)$ lie from $f^\nu(t,\cdot,\cdot)$ in terms of $\nu$ and $t$. In the second step, we compare $f^\nu$ and $\tilde f^\nu$ using Lemma~\ref{hlambdabound}. Finally, we combine the results and choose appropriate $\nu$ to obtain the claim.

First, we observe that the difference $f^\nu-f$ satisfies the following equation in a weak sense.
\begin{align*}
 & \partial_t ( f^\nu- f)+\divv\left( ( f^\nu- f)V_1[f^\nu]\right)
     +\divv(  f V_1[f^{\nu}-f]) -\nu\,\Delta ( f^\nu- f)  \\
 & =\nu\,\Delta  f  + A( f^\nu- f) - ( f^\nu- f)V_2[f^\nu]
   - fV_2[f^\nu-f].    
\end{align*}
 Testing the equation with $(f^{\nu}-f)$ and recalling the definition of $V_1,V_2$ we get 
\begin{align*}
    &  \frac{d}{dt} \frac{1}{2}\,\| (f^\nu-f)(t, \cdot, \xi)\|_{L^2(\R^d)}^2 +{\nu}\,\int_{\R^d} |\nabla  (f^\nu-f) (t, x, \xi)|^2\,dx \\
 &   =\frac{1}{2}\ird\nabla(f^\nu-f)^2(t,x,\xi)\,\cdot\int_0^1\int_{\R^{d}}  K_1(x-y)\, f^\nu(t,y, \zeta) \,dy w(\xi, d\zeta)\,dx\\
&+ \ird\nabla(f^\nu-f)(t,x,\xi) f(t,x,\xi)\cdot \int_0^1\int_{\R^{d}}  K_1(x-y)\, (f^\nu-f)(t,y, \zeta) \,dy w(\xi, d\zeta)\,dx\\
&  - \nu \ird \nabla f(t,x,\xi)\cdot\nabla(f^\nu - f)(t,x,\xi)dx + A \ird (f^\nu - f)^2(t,x,\xi)dx  \\
&-\ird (f^\nu-f)^2(t,x,\xi)\,\int_0^1\int_{\R^{d}}  K_2(x-y)\, f^\nu(t,y, \zeta) \,dy w(\xi, d\zeta)\,dx\\
&-\ird f(t,x,\xi)(f^\nu - f)(t,x,\xi)\,\int_0^1\int_{\R^{d}}  K_2(x-y)\, (f^\nu-f)(t,y, \zeta) \,dy w(\xi, d\zeta)\,dx = \sum_{j=1}^6 I_j.
\end{align*}
Integration by parts gives
\begin{align*}
    I_1 &= -\frac{1}{2}\ird(f^\nu-f)^2(t,x,\xi)\,\int_0^1\int_{\R^{d}} \divv K_1(x-y)\, f^\nu(t,y, \zeta) \,dy w(\xi, d\zeta)\,dx,\\
    I_2 &=-  \ird(f^\nu-f)(t,x,\xi) \nabla f(t,x,\xi)\cdot \int_0^1\int_{\R^{d}}  K_1(x-y)\, (f^\nu-f)(t,y, \zeta) \,dy w(\xi, d\zeta)\,dx \\
    &- \ird(f^\nu-f)(t,x,\xi) f(t,x,\xi)\int_0^1\int_{\R^{d}} \divv K_1(x-y)\, (f^\nu-f)(t,y, \zeta) \,dy w(\xi, d\zeta)\,dx.
\end{align*}
Hence, applying Lemma \ref{impoest} together with H\"older inequality and Young inequality for convolution we arrive at
\begin{align*}
\vert I_1\vert&\leq \frac{1}{2}\Vert w\Vert\,\Vert \divv K_1\Vert_{L^\infty(\R^d)}\,\Vert f^\nu(t,\cdot,\cdot)\Vert_{L^\infty((0,1);L^1(\R^d))}\,\Vert  (f^\nu - f)(t,\cdot,\xi)\Vert_{L^2(\R^d)}^2,\\
\vert I_2\vert &\leq \Vert w\Vert\, \Vert  K_1\Vert_{L^2(\R^d)}\,\Vert f(t,\cdot,\cdot)\Vert_{L^\infty((0,1);H^1(\R^d))}\Vert (f^\nu-f)(t,\cdot,\cdot)\Vert_{L^{\infty}((0,1);L^2(\R^d))}^2\\ 
+& \norm{w}\Vert  \divv K_1\Vert_{L^1(\R^d)}\,\Vert f(t,\cdot,\cdot)\Vert_{L^\infty((0,1);L^{\infty}(\R^d))}\Vert (f^\nu-f)(t,\cdot,\cdot)\Vert_{L^{\infty}((0,1);L^2(\R^d))}^2,\\
\vert I_3\vert &\leq \frac{\nu}{2}\norm{\nabla(f^\nu - f)(t,\cdot,\xi)}^2_{L^{2}(\R^d)} + \frac{\nu}{2}\norm{f(t,\cdot,\xi)}_{H^{1}(\R^d)}^2,\\
\vert I_4\vert &\leq A \norm{(f^\nu - f)(t,\cdot,\xi)}^2_{L^{2}(\R^d)},\\
\vert I_5\vert &\leq \Vert w\Vert\,\Vert  K_2\Vert_{L^\infty(\R^d)}\,\Vert f^\nu(t,\cdot,\cdot)\Vert_{L^\infty((0,1); L^1(\R^d))}\,\Vert (f^\nu- f)(t,\cdot,\xi)\Vert_{L^2(\R^d)}^2,\\
\vert I_6\vert &\leq \Vert w\Vert\,\Vert  K_2\Vert_{L^1(\R^d)}\,\Vert f(t,\cdot,\xi)\Vert_{L^\infty(\R^d)}\,\Vert (f^\nu- f)(t,\cdot,\cdot)\Vert_{L^{\infty}((0,1);L^2(\R^d))}^2.
\end{align*}
Thus, by Lemma \ref{linftyestf} we obtain the estimate of the form
\[
\frac{d}{dt} \frac{1}{2}\,\| (f^\nu-f)(t, \cdot, \cdot)\|_{L^{\infty}(0,1;L^2(\R^d))}^2 \leq C \Vert (f^\nu- f)(t,\cdot,\cdot)\Vert_{L^{\infty}((0,1);L^2(\R^d))}^2 + \frac{\nu}{2}\norm{f(t,\cdot,\cdot)}_{L^{\infty}((0,1);H^{1}(\R^d))},
\]
where $C$ is a positive constant dependent on $A$ and the norms $\norm{w}$,  $\norm{K_1}_{W^{1,1}(\R^d)}$, $\norm{\divv K_1}_{L^{\infty}(\R^d)}, \norm{K_2}_{L^{1}\cap L^{\infty}(\R^d)}$ and $f^0,\tilde{f}^0$ in $L^{\infty}((0,1);L^1\cap L^\infty\cap H^1(\R^d))$. 
Applying Lemma \ref{linftyestf} also to the last term and using Gr\"onwall lemma we get
  \begin{equation}\label{aproxst}
\| (f^\nu - f)(t,\cdot,\cdot)\|_{L^{\infty}((0,1);L^2(\R^d))} \leq C(t)\,\sqrt{\nu},
  \end{equation}
for some continuous and non-decreasing function $C=C(t)\in \R_+$ that only depends on the same norms as $C$ above. Obviously, the same estimate may be obtained for $\tilde f^\nu -\tilde f$. 

Let us now compare $f^\nu$ and $\tilde{f}^\nu$. To that end, note that due to  Proposition \ref{hier},  $\tau(T, w, f^\nu)$ and $\tau(T, \tilde w,\tilde f^\nu)$ with $T\in \Tree$ solve the same Vlasov hierarchy \eqref{treeeq}. Hence, by linearity, $h=(h_T)_{T\in \Tree}$ defined by $h_T=\tau(T, w, f^\nu) - \tau(T, \tilde w,\tilde f^\nu)$ solves \eqref{hiergeneq}. Moreover, by Lemma~\ref{treeb} we may estimate
\begin{align*}
\Vert h_T(t,\cdot)\Vert_{L^2(\R^{d |T|})}&\leq \Vert \tau(T,w,f^\nu(t,\cdot))\Vert_{L^2(\R^{d |T|})}+\Vert \tau(T,w,\tilde{f}^\nu(t,\cdot))\Vert_{L^2(\R^{d |T|})}\\
&\leq \Vert w\Vert^{|T|-1}\,\Vert f^\nu(t,\cdot,\cdot)\Vert_{L^\infty((0,1); L^2(\R^d))}^{|T|}+\Vert \tilde w\Vert^{|T|-1}\,\Vert \tilde f^\nu(t,\cdot,\cdot)\Vert_{L^\infty((0,1); L^2(\R^d))}^{|T|}\\
& \leq w_{\text{max}}^{|T|-1}\left(\norm{f^\nu(t,\cdot,\cdot)}_{L^\infty((0,1); L^2(\R^d))}^{|T|} + \norm{\tilde{f}^\nu(t,\cdot,\cdot)}_{L^\infty((0,1); L^2(\R^d))}^{|T|} \right),
\end{align*}

where $w_{\text{max}}=\max\{\Vert w\Vert,\Vert \tilde w\Vert\}$. 
Applying the estimate (\ref{flinftytwo}) with $p=2$ and introducing $f^0_{\text{max},p}:=\max\{\norm{f^0}_{L^\infty((0,1);L^p(\R^d))},\norm{\tilde{f}^0}_{L^\infty((0,1);L^p(\R^d))} \}$, for $p \geq 1$ we arrive at
\[
\norm{h_T(t,\cdot,\cdot)}_{L^2(\R^{d |T|})}\leq 2 w_{\text{max}}^{|T|-1}(f^0_{\max,2})^{|T|} \exp\left(\left(\frac{w_{\text{max}}\Vert \divv K_1\Vert_{L^\infty(\R^d)}e^{At_*}\,f^0_{\max,1}}{2} + A\right)t|T|\right).
\]

Thus, we note that $\sup_{t\in [0,\ t_*]}\|h(t,\cdot)\|_\lambda<1$, for any $\lambda>0$ such that
\eqq{
\sqrt{\lambda}<\min\left\{\frac{1}{w_{\text{max}}},1\right\}\frac{\exp\left(-(\frac{w_{\text{max}}}{2}\Vert \divv K_1\Vert_{L^\infty(\R^d)}e^{At_*}\,f^0_{\max,1} + A)t_*\right)}{2f^0_{\max,2}},
}{lest}
where we applied $w_{\max}^{|T|-1}\leq \max\{1,w_{\max}^{|T|}\}$.
  Take arbitrary $\lambda$ satisfying (\ref{lest}). Then, by Lemma~\ref{hlambdabound} for any $p>1$ and $\theta\in (0,2^{-p'}e^{-p'(2A+1)t_*})$ with $p'=\frac{p}{p-1}$ there exists $C$ dependent on $p,A,t_*,\theta$ such that 
  \[
\sup_{t\in [0,\ t_*]}\|h(t,\cdot)\|_{\theta\lambda}\leq C \,\exp\left(p^{-\frac{\left((2\nu)^{-1}\norm{K_1}_{L^{2}(\R^d)}^2  + \norm{K_2}_{L^{2}(\R^d)}^2\right)}{\theta\lambda}t}\,\log\Vert h(0,\cdot)\Vert_{\theta \lambda}\right).
  \]
Choosing particular $T=T_1$ we arrive at 
  \begin{align*}
   \|h(t,\cdot)\|_{\theta\lambda}&\geq (\theta\lambda)^{\frac{1}{2}} \|h_{T_1}(t,\cdot)\|_{L^2(\R^d)}=\sqrt{\lambda}\sqrt{\theta} \left\|\tau(T_1,w, f^\nu)(t,\cdot)-\tau(T_1,\tilde w,\tilde f^\nu)(t,\cdot)\right\|_{L^2(\R^d)}   \\
   & =\sqrt{\lambda}\sqrt{\theta}\,\left\| \int_0^1 ( f^\nu -\tilde f^\nu) (t,\cdot, \xi)\,d\xi\,\right\|_{L^2(\R^d)}
  \end{align*}
  and consequently for every $t \in (0,t_{*})$
  \[
\left\| \int_0^1 ( f^\nu -\tilde f^\nu) (t,\cdot, \xi)\,d\xi\,\right\|_{L^2(\R^d)} \leq \frac{C}{\sqrt{\lambda}\sqrt{\theta}}\,\exp\left(p^{-\frac{\left((2\nu)^{-1}\norm{K_1}_{L^{2}(\R^d)}^2  + \norm{K_2}_{L^{2}(\R^d)}^2\right)}{\theta\lambda}t}\,\log\Vert h(0,\cdot)\Vert_{\theta \lambda}\right).\]
Combining the estimate above with (\ref{aproxst}), we obtain, by the triangle inequality

\[
\left\| \int_0^1 ( f -\tilde f) (t,\cdot, \xi)\,d\xi\,\right\|_{L^2(\R^d)}  \leq \frac{C}{\sqrt{\lambda}\sqrt{\theta}}\,\exp\left(p^{-\frac{\left((2\nu)^{-1}\norm{K_1}_{L^{2}(\R^d)}^2  + \norm{K_2}_{L^{2}(\R^d)}^2\right)}{\theta\lambda}t}\,\log\Vert h(0,\cdot)\Vert_{\theta \lambda}\right)+2C(t)\,\sqrt{\nu}.
\]
In order to finish the proof, it remains to choose appropriate $\nu$. To simplify the calculations we take $p=e$ and we choose 
\[
\nu=\frac{\Vert K_1\Vert_{L^2(\R^d)}^2}{\theta\lambda}\,t_*\,(\log |\log \Vert h(0,\cdot)\Vert_{\theta\lambda}|)_+^{-1}.
\]
Then,
\begin{align}\label{fest7}
   \left\| \int_0^1 ( f -\tilde f) (t,\cdot, \xi)\,d\xi\,\right\|_{L^2(\R^d)}&\leq \frac{C}{\sqrt{\lambda}\sqrt{\theta}}\,\exp(e^{-\frac{t}{2t_*}(\log|\log\norm{h(0,\cdot)}_{\theta\lambda}|)_{+}}e^{-\frac{\norm{K_2}_{L^{2}(\R^d)}^2 }{\theta \lambda}t}\log \norm{h(0,\cdot)}_{\theta\lambda})  \\ \nonumber
   &+ 2 C(t)\frac{\Vert K_1\Vert_{L^2(\R^d)}}{\sqrt{\theta\lambda}}\,\sqrt{t_*}\,\frac{1}{\sqrt{(\log |\log \Vert h(0,\cdot)\Vert_{\theta\lambda}|)_+}}.
\end{align}
Note that $\lambda$ was chosen in such a way that $\norm{h(0,\cdot)}_{\lambda} < 1$. Since $\theta < 1$ also $\norm{h(0,\cdot)}_{\theta\lambda} < 1$, which gives $\log \norm{h(0,\cdot)}_{\theta\lambda}$ negative. Thus, we  estimate as follows
\begin{align*}
    &e^{-\frac{t}{2t_*}(\log|\log\norm{h(0,\cdot)}_{\theta\lambda}|)_{+}}\log \norm{h(0,\cdot)}_{\theta\lambda} \leq e^{-(\log|\log\norm{h(0,\cdot)}_{\theta\lambda}|^\frac{1}{2})_{+}}\log \norm{h(0,\cdot)}_{\theta\lambda}\\
    &= \frac{\log\norm{h(0,\cdot)}_{\theta\lambda}}{\abs{\log\norm{h(0,\cdot)}_{\theta\lambda}}^{\frac{1}{2}}} = -\abs{\log\norm{h(0,\cdot)}_{\theta\lambda}}^{\frac{1}{2}}.
\end{align*}   
Hence,
    \[
\exp\left(e^{-\frac{t}{2t_*}(\log|\log\norm{h(0,\cdot)}_{\theta\lambda}|)_{+}}e^{-\frac{\norm{K_2}_{L^{2}(\R^d)}^2 t}{\theta \lambda}}\log \norm{h(0,\cdot)}_{\theta\lambda}\right) \leq \exp\left(-\abs{\log\norm{h(0,\cdot)}_{\theta\lambda}}^{\frac{1}{2}} e^{-\frac{\norm{K_2}_{L^{2}(\R^d)}^2 t_*}{\theta \lambda}}\right).
\]
Note that for every $a > 1, b \in (0,1)$
\[
\exp(-\sqrt{a}b) \leq \frac{1}{\sqrt{2b} \sqrt{ \ln a}}
\]
Applying this inequality with
\[
a=\abs{\log\norm{h(0,\cdot)}_{\theta\lambda}}, \hd \hd b = e^{-\frac{\norm{K_2}_{L^{2}(\R^d)}^2 t_*}{\theta \lambda}}
\]
we finally obtain
\[
\exp(e^{-\frac{t}{2t_*}(\log|\log\norm{h(0,\cdot)}_{\theta\lambda}|)_{+}}e^{-\frac{\norm{K_2}_{L^{2}(\R^d)}^2 t}{\theta \lambda}}\log \norm{h(0,\cdot)}_{\theta\lambda}) \leq \frac{1}{\sqrt{2}}  e^{\frac{\norm{K_2}_{L^{2}(\R^d)}^2 t_*}{2\theta \lambda}} \frac{1}{\sqrt{(\log |\log \Vert h(0,\cdot)\Vert_{\theta\lambda}|)_+}} . 
\]
Inserting the above estimate in (\ref{fest7}) and noting that $C(t)$ is an increasing function of time we arrive at the claim of lemma. We remark that if it happens that $K_1 = 0$, then the term $\norm{K_1}^2_{L^2(\R^d)}/2\nu$ disappears in Lemma \ref{hlambdabound} and the estimate still holds with $\nu=0$, since $u^{\kappa} \leq \frac{C(\kappa)}{\sqrt{(\log\abs{\log u})_{+}}}$ for small $u$ and any $\kappa \in (0,1)$. 
\end{proof}

\end{subsection}

\begin{subsection}{Proof of the main result}
At last we are ready to prove  Theorem~\ref{maintheorem}. The proof closely follows the approach introduced in \cite{JPS2025}. The new ingredient is to show that the assumption (\ref{asf}) implies that $f_N^0$ satisfies the assumption $2.$ of Theorem~\ref{lg}, where
\[
f_N^0(x,\xi):=\sum_{i=1}^N \tilde{f}_i^0(x)\,\mathbb{I}_{[\frac{i-1}{N},\frac{i}{N})}(\xi), \hd \hd x\in \R^d,\xi\in [0,\ 1], \hd \tilde{f}_i^0:=X^{0}_{i\#}[M_i^0 d\p].
\]
Despite of that, we present the whole argument for the sake of completeness.
\begin{proof}[Proof of Theorem \ref{maintheorem}]
Let us define
\[
w_N(\xi,\zeta):=\sum_{i,j=1}^N N\,w^N_{ij}\,\mathbb{I}_{[\frac{i-1}{N},\frac{i}{N})}(\xi)\, \mathbb{I}_{[\frac{j-1}{N},\frac{j}{N})}(\zeta), \hd \hd \xi,\zeta\in [0,\ 1].
\]
Then, thanks to (\ref{aswb})
 $w_N$ satisfies the assumption $1.$ of Theorem~\ref{lg}. 
Indeed, we have
\[
\sup_{N \in \mathbb{N}} \sup_{\xi \in [0,1]}\int_0^1 |w_N(\xi,\eta)|d\eta \leq  \sup_{N \in \mathbb{N}}  \max_{1\leq i \leq N}\sum_{j=1}^N|w_{ij}^N| \leq C_w,
\]
\[
 \sup_{N \in \mathbb{N}} \sup_{\eta \in [0,1]}\int_0^1 |w_N(\xi,\eta)|d\xi \leq \sup_{N \in \mathbb{N}}  \max_{1\leq j \leq N}\sum_{i=1}^N|w_{ij}^N| \leq C_w.
\]
To show that $f_N^0$ satisfies the assumption $2.$ of Theorem \ref{lg} we invoke Proposition \ref{fireg}.
Then, by the assumption (\ref{asf})
\[
\sup_{N \in \mathbb{N}}\norm{f^0_N}_{L^{\infty}((0,1);(W^{1,1}\cap W^{1,\infty})(\R^d))}\hspace{-0.1 cm}=\sup_{N \in \mathbb{N}} \max_{1\leq i \leq N}\norm{\tilde{f}^0_i}_{(W^{1,1}\cap W^{1,\infty})(\R^d)} \hspace{-0.2 cm}= \sup_{N \in \mathbb{N}} \max_{1\leq i \leq N}\norm{g_i}_{(W^{1,1}\cap W^{1,\infty})(\R^d)}:=C_0 < \infty.
\]


Thus, from Theorem~\ref{lg}, we infer that there exist $w\in L^\infty_\xi \mathcal{M}_\zeta\cap L^\infty_\zeta \mathcal{M}_\xi$, $f^0\in L^\infty((0,1); (W^{1,1}\cap W^{1,\infty})(\R^d))$ and a subsequence (still denoted by $N$) such that for all $T$
\[
\tau(T,w,f^0)=\lim_{N \to \infty} \tau(T,w_N,f_N^0) \hd \m{ in } \hd L^p_{loc}(\R^{d|T|}), \hd p \in [1,\infty).
\]
Furthermore, since $f^0_N$ and $w_N$ are nonnegative, also the limiting $f^0$ and $w$ are nonnegative. Indeed,  according to Remark \ref{wlrem} (see also the Proof of \cite[Theorem 5.1]{JPS2025}) $w$ is a weak$^*$ limit in $L^\infty_\xi \mathcal{M}_\zeta\cap L^\infty_\zeta \mathcal{M}_\xi$ of properly selected subsequence of $w^N$ after proper relabeling and $f^0$ is a strong $L^{p}_{loc}$ limit of rearranged $f^0_N$. Hence,  nonnegativity holds also for the limits $f^0$ and $w$.
Our aim is to improve this convergence to hold globally in $\R^d$.
First, note that,  by Lemma \ref{treeb}
\[
\norm{\tau(T,w,f^0)}_{(L^1\cap L^\infty)(\R^{d|T|})} \leq \norm{w}^{|T|-1}\norm{f^0}_{L^{\infty}((0,1);(L^1\cap L^\infty)(\R^d))}^{|T|} < \infty \hd \hd  \m{for every } T
\]
and 
\[
\sup_{N\in \mathbb{N}}\Vert \tau(T,w_N,f^0_N)\Vert_{(W^{1,1}\cap W^{1,\infty})(\R^{d|T|})}\leq C_w^{|T|-1}\, C_0^{\vert T\vert} < \infty.
\]

Fix $R>0$, a straightforward extension of Lemma~\ref{treeb} leads to
\eqq{
\|\tau(T,w_N,f^0_N)\|_{L^1(\R^{d|T|}\setminus B(0,R)^{|T|})}\leq |T| \|w_N\|^{|T|-1}\,\|f^0_N\|_{L^\infty((0,1); L^1(\R^{d}))}^{|T|-1}\,\|f^0_N\|_{L^\infty((0,1); L^1(\R^{d}\setminus B(0,R)))}.
}{out}
On the other hand, by the assumptions (\ref{M}) and (\ref{secmo}) there exists a positive constant $C$ independent of $N$ such that
\[
\int_{\R^d} |x|^2\,df_N^0(x,\xi)
= \sum_{i=1}^N \ird |X^0_i(x)|^2 M^0_i(x)d\p(x) \,\mathbb{I}_{[\frac{i-1}{N},\frac{i}{N})}(\xi)
\leq  M \max_{1\leq i \leq N } \mathbb{E} |X^0_i|^2 \leq C.
\]

Inserting this result in (\ref{out}) leads to
\[
\|\tau(T,w_N,f^0_N)\|_{L^1(\R^{d|T|}\setminus B(0,R)^{|T|})}\leq \frac{C|T|}{R^2}\,\,\|w_N\|\|f^0_N\|_{L^\infty((0,1); L^1(\R^{d}))}^{|T|-1} \leq \frac{C|T|}{R^2}C_w^{|T|-1}C_0^{|T|-1}.
\]
Hence, the sequence $\tau(T,w_N,f^0_N)$ is equicontinuous and equitight in $L^{1}(\R^{d|T|})$.
Combining it with the local convergence in $L^1$, it implies that $\tau(T,w_N,f_N^0)$ converges strongly to $\tau(T,w,f^0)$ in  $L^1(\R^{d\,|T|})$. Since $\tau(T,w_N,f_N^0)$ is also bounded in $L^{\infty}(\R^{d|T|})$, applying the interpolation inequality we obtain that $\tau(T,w_N,f_N^0)$ converges strongly to $\tau(T,w,f^0)$ in $L^p(\R^{d\,|T|})$ for any $p\in [1,\infty)$ and in particular in $L^2(\R^{d\,|T|})$.
Besides, applying again Lemma~\ref{treeb}, we may estimate as follows
\[
\|\tau(T,w,f^0)\|_{L^2(\R^d)}\leq \|w\|^{|T|-1}\,\|f^0\|_{L^\infty((0,1); L^2(\R^d))}^{|T|}.
\]
Thus, combining the convergence for any fixed tree together with the estimate above we obtain that there exists some $\lambda>0$ small enough such that
\begin{equation}
\|\tau(\cdot,w,f^0)-\tau(\cdot,w_N,f^0_N)\|_{\lambda}\to 0,\quad\mbox{as}\ N\to\infty.\label{finalconv}
  \end{equation}
By Proposition~\ref{existenceprop}, there exists a  weak solution $f\in L^\infty((0,t_*)\times(0,1); ( W^{1,1}\cap W^{1,\infty})(\R^d))$ to \eqref{meannu} with $\nu=0$  associated with $w$.
Define 
\[
\bar{f}_N(t,x,\xi)=\sum_{i=1}^N \tilde f_i(t,x)\,\mathbb{I}_{[\frac{i-1}{N},\frac{i}{N})}(\xi),\qquad t\in (0,t_{*}),\;x\in \R^d,\;\xi\in [0,\ 1],
\]
where $\tilde{f}_i(t,x) = \xb_i(t)_{\#}[\Mb_i(t)d\p]$ and $(\xb_i,\Mb_i)_{i=1}^N$ is a solution to (\ref{part2}) given by Lemma \ref{exi}.  
We observe that $\bar{f}_N$ is a weak solution to \eqref{meannu} with $w_N$ and $\nu=0$. Moreover,  $\bar{f}_N\in L^\infty((0,t_*)\times(0,1);H^{1}(\R^d))$ due to Proposition \ref{fireg}. 
Hence, both $\bar{f}_N$ and $f$ satisfy the assumptions of Lemma~\ref{mainstab}. Combining Lemma~\ref{mainstab} with \eqref{finalconv}, we obtain
\[
\esssup_{t \in (0,t_{*})}\left\|\int_0^1 (f-\bar{f}_N)(t,\cdot,\xi)\,d\xi\right\|_{L^2(\R^d)}\leq \frac{C}{{(\log |\log \|\tau(\cdot,w_N,f^0_N)-\tau(\cdot,w,f^0)\|_{\lambda}|)^{1/2}_+}}\to 0,\quad \mbox{as}\ N\to \infty.
\]
In order to show that the convergence holds also in $L^1(\R^d)$, note that for any $R>0$
\eqq{
\left\|\int_0^1 (f-\bar{f}_N)(t,\cdot,\xi)\,d\xi\right\|_{L^1(\R^d)}\hspace{-0.2cm}\leq \left\|\int_0^1 (f-\bar{f}_N)(t,\cdot,\xi)\,d\xi\right\|_{L^2(B_{R})} \hspace{-0.2cm}|B_{R}|^{\frac{1}{2}} + \frac{1}{R^2}\int_{\R^d \setminus B_R}|x|^2 \abs{\int_0^1 (f+\bar{f}_N)(t,\cdot,\xi)\,d\xi}dx.  
}{l1conv}
Recall that there exists constant $C$ independent on $N$ and $t$ such that
\[
\ird |x|^2 \izj \bar{f}_N(t,x,\xi)d\xi dx = \sn \ird |x|^2 \tilde{f}_i(t,x)dx \leq \max_{1\leq i\leq N}\ird |\xb_i(t)|\Mb_i(t) d\p \leq C,
\]
where in the last estimate we applied (\ref{Mbound}) and (\ref{secmobound}). On the other hand, by Fatou lemma
\[
\ird |x|^2 \izj f(t,x,\xi)d\xi dx \leq \liminf_{N\rightarrow \infty} \ird |x|^2 \izj \bar{f}_N(t,x,\xi)d\xi dx \leq C.
\]
Thus, passing to the limit with $N$ in (\ref{l1conv}) we obtain
\[
\lim_{N\rightarrow \infty}\esssup_{t \in (0,t_{*})}\left\|\int_0^1 (f-\bar{f}_N)(t,\cdot,\xi)\,d\xi\right\|_{L^1(\R^d)} \leq \frac{C}{R^2}
\]
for some constant $C>0$. Since $R$ is arbitrary we obtain the $L^1$ convergence and hence also convergence in flat norm: 
\[
\esssup_{t \in (0,t_{*})}d_F\left(\int_0^1 \bar{f}_N(t,\cdot,\xi)\,d\xi,\int_0^1 f(t,\cdot,\xi)\,d\xi\right)\to 0,\quad \mbox{as}\ N\to \infty.
\]
Since $\int_0^1 \bar{f}_N(t,x,\xi)\,d\xi=\frac{1}{N}\,\sum_{i=1}^N \tilde f_i(t,x)$, we conclude the proof of Theorem~\ref{maintheorem}, recalling \eqref{finalp} from Theorem \ref{mainprob}.
\end{proof}
\end{subsection}

\end{section}

\section{Appendix}

\subsection{Proof of Lemma \ref{odestab}}

We begin the Appendix with a proof of auxiliary result from Chapter 2.
\begin{proof}[Proof of Lemma \ref{odestab}]
Clearly, under the assumptions (\ref{asK}) and (\ref{Mz}) we have
\eqq{
\max_{1\leq i\leq N} \sup_{t \in (0,t_{*})}M_i(t) \leq Me^{At_{*}}.}
{cest}
Let us subtract the corresponding equations for $M_i$ and $\tilde{M}_i$, then
\begin{align*}
&\partial_t(M_{i}(t) - \tilde{M}_{i}(t)) 
=\\
&A(M_{i}(t)- \tilde{M}_i(t)) + \frac{1}{N}\left[\tilde{M}_i(t)\sum_{j=1}^N \tilde{M}_{j}(t)\tilde{w}^{N}_{ij}K_{2}(\tx_i(t) -\tx_{j}(t))-M_i(t)\sum_{j=1}^N M_{j}(t)w^{N}_{ij}K_{2}(\nx_i(t) -\nx_{j}(t))\right].
\end{align*}
Multiplying by $\frac{1}{N}(M_{i}(t)- \tilde{M}_i(t))$ and summing over $i$ we get
\begin{align*}
&\frac{1}{2}\frac{d}{dt} \frac{1}{N}\sum_{i=1}^N(M_{i}(t)- \tilde{M}_i(t))^2 = A\frac{1}{N}\sum_{i=1}^N(M_{i}(t)- \tilde{M}_i(t))^2 \\
&+ \frac{1}{N^2}\sum_{i,j=1}^N (M_i(t) - \tilde{M}_i(t)) \left[\tilde{M}_i(t)\tilde{M}_{j}(t)\tilde{w}^{N}_{ij}K_{2}(\tx_i(t) -\tx_{j}(t))-M_i(t) M_{j}(t)w^{N}_{ij}K_{2}(\nx_i(t) -\nx_{j}(t))\right]\\
&\leq \frac{A}{N}\sum_{i=1}^N(M_{i}(t)- \tilde{M}_i(t))^2  +  \frac{\norm{K_2}_{L^{\infty}(\R^d)}}{N^2}\sum_{i,j=1}^N \abs{M_i(t) - \tilde{M}_i(t)} \abs{\tilde{w}^{N}_{ij} - w^{N}_{ij}}\tilde{M}_i(t)\tilde{M}_{j}(t)\\
&+ \frac{1}{N^2}\sum_{i,j=1}^N (M_i(t) - \tilde{M}_i(t))w^{N}_{ij}\left[\tilde{M}_i(t)\tilde{M}_j(t)K_{2}(\tx_i(t) -\tx_{j}(t))-M_i(t) M_{j}(t)K_{2}(\nx_i(t) -\nx_{j}(t))\right]=:\sum_{l=1}^3I_l.
\end{align*}
We note that
\begin{align*}
I_1+I_2&\leq \frac{A}{N}\sum_{i=1}^N(M_{i}(t)- \tilde{M}_i(t))^2  + \frac{\norm{K_2}_{L^{\infty}(\R^d)}}{N^2}\left(\sum_{i,j=1}^N \abs{\tilde{w}^{N}_{ij} - w^{N}_{ij}}^2 + \sum_{i,j=1}^N\abs{M_{i}(t)- \tilde{M}_i(t)}^2\abs{\tilde{M}_i(t)\tilde{M}_{j}(t)}^2\right)\\
&
\leq \frac{\norm{K_2}_{L^{\infty}(\R^d)}}{N^2}\sum_{i,j=1}^N \abs{\tilde{w}^{N}_{ij} - w^{N}_{ij}}^2 + \frac{1}{N}\sum_{i=1}^N(M_{i}(t)- \tilde{M}_i(t))^2 \left(A + \norm{K_2}_{L^{\infty}(\R^d)}e^{4At_{*}}M^4\right),
\end{align*}
where in the last estimate we used (\ref{cest}). Let us decompose $I_3$ as follows.
\begin{align*}
&I_3 \leq \frac{1}{N^2}\sum_{i,j=1}^N  w^{N}_{ij}(M_i(t) - \tilde{M}_i(t))(\tilde{M}_i(t)-M_i(t))\tilde{M}_j(t)K_{2}(\tx_i(t) -\tx_{j}(t))\\
&+ \frac{1}{N^2}\sum_{i,j=1}^N  w^{N}_{ij}(M_i(t) - \tilde{M}_i(t))(\tilde{M}_j(t) - M_j(t))M_i(t)K_{2}(\tx_i(t) -\tx_{j}(t))\\
&+  \frac{1}{N^2}\sum_{i,j=1}^N  w^{N}_{ij}(M_i(t) - \tilde{M}_i(t)) M_i(t)\nc_j(t)\left(K_2(\tx_i(t) - \tx_j(t)) - K_2(\nx_i(t) - \nx_j(t))\right)=\sum_{l=1}^3I_3^l.
\end{align*}
We note that, since $K_2,w^N_{ij},M_i \geq 0$, $I_3^1$ is nonpositive.
Moreover, applying again (\ref{cest})   we have
\[
I_3^2 \leq \frac{M \norm{K_2}_{L^{\infty}(\R^d)}C_w}{N^2} e^{At_{*}}  \left(\sum_{i=1}^N \abs{\nc_i(t) - \tilde{M}_i(t)}\right)^2 
\leq M \norm{K_2}_{L^{\infty}(\R^d)} e^{At_{*}} C_w \frac{1}{N}\sum_{i=1}^N\abs{\nc_i(t) - \tilde{M}_i(t)}^2.
\]
Next, by (\ref{asK})
\[
I^3_3 \leq \frac{\norm{K_2}_{W^{1,\infty}(\R^d)}}{N^2}\sum_{i,j=1}^N  w^{N}_{ij}\abs{\nc_i(t) - \tilde{M}_i(t)}\nc_i(t)\nc_j(t)\left(\abs{\tx_i(t) - \nx_i(t)} +\abs{\tx_j(t) - \nx_j(t)}\right)
\]
 and  hence,
\[
I_3^3 \leq e^{2At_{*}} 
M^{2} \norm{K_2}_{W^{1,\infty}(\R^d)} C_w\frac{1}{N}\left(\sum_{i=1}^N\abs{\nc_i(t) - \tilde{M}_i(t)}^2 + \sum_{i=1}^N\abs{\nx_i(t) - \tx_i(t)}^2\right).
\]
Combing the estimates above we arrive at
\begin{align}\label{nac}
&\frac{1}{2}\frac{d}{dt} \frac{1}{N}\sum_{i=1}^N(\nc_{i}(t)- \tilde{M}_i(t))^2 \leq  \frac{\norm{K_2}_{L^{\infty}(\R^d)}}{N^2}\sum_{i,j=1}^N \abs{\tilde{w}^{N}_{ij} - w^{N}_{ij}}^2 +
\\ \nonumber
&\left(A + \norm{K_2}_{L^{\infty}(\R^d)}e^{4At_{*}}M^4  + C_w M \norm{K_2}_{L^{\infty}(\R^d)} e^{At_{*}} + e^{2At_{*}} 
M^{2} \norm{K_2}_{W^{1,\infty}(\R^d)} C_w\right)\frac{1}{N}\sum_{i=1}^N(\nc_{i}(t)- \tilde{M}_i(t))^2 
\\ \nonumber
&
+  e^{2At_{*}} 
M^{2} \norm{K_2}_{W^{1,\infty}(\R^d)} C_w\frac{1}{N} \sum_{i=1}^N |\nx_i(t) - \tx_i(t)|^2. 
\end{align}
Similarly we subtract the corresponding equations for $X_i$ and $\tilde{X}_i$
\[
\partial_t(\nx_i(t) - \tx_i(t)) = \frac{1}{N}\sum_{j=1}^N\left(\nc_j(t)w^N_{ij}K_1(\nx_i(t) - \nx_j(t)) - \tilde{M}_j(t)\tilde{w}^N_{ij}K_1(\tx_i(t) - \tx_j(t))\right)
\]
We multiply the identity above by $\frac{1}{N}(\nx_{i}(t)- \tx_i(t))$, sum over $i$ and perform the estimates as follows
\begin{align*}
&\frac{1}{2}\frac{1}{N}\frac{d}{dt}\sum_{i=1}^N (\nx_i(t) - \tx_i(t))^2 \leq \frac{1}{N^2}\sum_{i,j=1}^N \abs{\nx_i(t) - \tx_i(t)}\nc_j(t)K_1(X_i(t)-X_j(t))\abs{w^N_{ij}-\tilde{w}^N_{ij}}\\
&+\frac{1}{N^2}\sum_{i,j=1}^N\abs{\nx_i(t) - \tx_i(t)} \tilde{w}^N_{ij}\abs{\nc_j(t) K_1(\nx_i(t) - \nx_j(t))- \tilde{M}_j(t)K_1(\tx_i(t) - \tx_j(t))}\\
&\leq e^{At_{*}} M\norm{K_1}_{L^{\infty}(\R^d)}\left(\frac{1 }{N}\sum_{i=1}^N \abs{\nx_i(t) - \tx_i(t)}^2  +  \frac{1}{N^2}\sum_{i,j=1}^N\abs{w^N_{ij}-\tilde{w}^N_{ij}}^2\right)\\
&+ \frac{\norm{K_1}_{W^{1,\infty}(\R^d)}}{N^2}\sum_{i,j=1}^N\abs{\nx_i(t) - \tx_i(t)} \tilde{w}^N_{ij}\left(\abs{\nc_j(t) - \tilde{M}_j(t)} + \tilde{M}_j(t)\abs{\nx_i(t) - \tx_i(t)} +\tilde{M}_j(t)\abs{\nx_j(t) - \tx_j(t)} \right). 
\end{align*}
Note that
\[
\frac{1}{N^2}\sum_{i,j=1}^N\tilde{w}^N_{ij}\abs{\nx_i(t) - \tx_i(t)} \abs{\nc_j(t) - \tilde{M}_j(t)} \leq C_{\tilde{w}}\frac{1}{N}\left( \sum_{i=1}^N \abs{\nx_i(t) - \tx_i(t)}^2  + \sum_{i=1}^N \abs{\nc_i(t) - \tilde{M}_i(t)}^2\right),
\]
and by (\ref{cest}) 
\[
\frac{1}{N^2}\sum_{i,j=1}^N\tilde{w}^N_{ij}\tilde{M}_j(t)\abs{\nx_i(t) - \tx_i(t)}^2  \leq e^{At_{*}}M C_{\tilde{w}}\frac{1}{N}\sum_{i=1}^{N}\abs{\nx_i(t) - \tx_i(t)}^2. 
\]
At last, 
\[
\frac{1}{N^2}\sum_{i,j=1}^N\tilde{w}^N_{ij}\tilde{M}_j(t)\abs{\nx_i(t) - \tx_i(t)}\abs{\nx_j(t) - \tx_j(t)} \leq e^{At_{*}}MC_{\tilde{w}}\frac{1}{N}\sum_{i=1}^{N}\abs{\nx_i(t) - \tx_i(t)}^2.
\]
Combining the estimates above we arrive at
\begin{align}\label{nax1}
&\frac{1}{2}\frac{1}{N}\frac{d}{dt}\sum_{i=1}^N \abs{\nx_i(t) - \tx_i(t)}^2 \leq \frac{e^{At_{*}}M \norm{K_1}_{L^{\infty}(\R^d)}}{N^2}\sum_{i,j=1}^N\abs{w^N_{ij}-\tilde{w}^N_{ij}}^2\\ \nonumber
&+C_{\tilde{w}}\norm{K_1}_{W^{1,\infty}(\R^d)}\frac{1}{N}\sum_{i}^N\abs{\nc_i(t) - \tilde{M}_i(t)}^2\\ \nonumber
&+ \left( e^{At_{*}}M\left(2C_{\tilde{w}}\norm{K_1}_{W^{1,\infty}(\R^d)}+\norm{K_1}_{L^{\infty}(\R^d)}\right) +C_{\tilde{w}}\norm{K_1}_{W^{1,\infty}(\R^d)}\right)\frac{1}{N}\sum_{i=1}^N \abs{\nx_i(t) - \tx_i(t)}^2. 
\end{align}
Adding (\ref{nac}) and (\ref{nax1}), and applying Gr\"onwall lemma, we obtain the claim.
\end{proof}

\subsection{Proof of Lemma \ref{exi} and Proposition \ref{fireg}}

Let us firstly give a proof of Lemma \ref{exi}.
\begin{proof}[Proof of Lemma \ref{exi}]
Fix $N \in \mathbb{N}$. We will prove the unique existence of the solutions to (\ref{part2}), since the proof for (\ref{part}) is analogous. First, we show that there exists exactly one solution to (\ref{part2}) in the space $(E,d_{\alpha})$ for $\al >0$ big enough, where
\begin{align*}
   E:=\Bigg\{&(\bar{X},\bar{M}) = ((\xb_i)_{i=1}^N,(\Mb_i)_{i=1}^N),\hd  \xb_i,\Mb_i \in C([0,t_*];L^{1}(\R^d,\mathbb{P})), \\
 &\Mb_i(t) \geq 0, \hd \max_{1\leq i\leq N}\sup_{t\in (0,t_*)} \Mb_i(t) \leq Me^{At_*}\hd  \hd \p \hd  a.e., \hd \xb_i(0)=X^0_i,\hd  \Mb_i(0) = M^0_i \Bigg\}  
\end{align*}
and
\[
d_{\alpha}((\xb^1,\Mb^1),(\xb^2,\Mb^2)) :=  \sup_{t\in (0,t_*)}e^{-\alpha t} \left(\max_{1\leq i \leq N} \E|\xb^1_i(t)-\xb^2_i(t)| + \max_{1\leq i \leq N}\E|\Mb^1_i(t)-\Mb^2_i(t)|\right). 
\]
Note that $(E,d_{\alpha})$ is a complete metric space. In order to apply the Banach fixed point theorem we define an operator $G:E\rightarrow E$ as
\begin{align*}
&(G(\xb,\Mb))_{i=1}^N=\\
&\left(X^0_i + \sum_{j=1}^N w^N_{ij} \int_0^{t}\ird K_{1}(\xb_i(\tau)-y)\tilde{f}_j(\tau,dy)d\tau, M^{0}_i\exp\left[\int_0^t \left(A-\sum_{j=1}^N w^N_{ij}\ird K_{2}(\xb_i(\tau)-y)\tilde{f}_j(\tau,dy)\right) d\tau \right]\right)_{i=1}^N.
\end{align*}
Due to the imposed assumptions, we see that, indeed the image of $E$ under the operator $G$ is included in $E$.
Let us show that $G$ is a contraction on $(E,d_{\alpha})$ for properly chosen $\alpha>0$. Take $(\xb^1,\Mb^1), (\xb^2,\Mb^2) \in E$ and denote 
$\tilde{f}^1_i(t,x) = \xb^1_i(t,x)_{\#}[\Mb^1_i(t,x)d\mathbb{P}(x)], \hd \tilde{f}^2_i(t,x) = \xb^2_i(t,x)_{\#}[\Mb^2_i(t,x)d\mathbb{P}(x)]$.
Then,
\begin{align} \label{firsteq}
&\E\abs{X^0_i + \sum_{j=1}^N w^N_{ij} \int_0^{t}\ird K_{1}(\xb_i^1(\tau)-y)\tilde{f}^1_j(\tau,dy)d\tau - X^0_i - \sum_{j=1}^N w^N_{ij} \int_0^{t}\ird K_{1}(\xb_i^2(\tau)-y)\tilde{f}^2_j(\tau,dy)d\tau}\\ \nonumber
    \leq &\E\abs{\sum_{j=1}^N w^N_{ij} \int_0^{t}\ird [K_{1}(\xb_i^1(\tau)-y)-K_{1}(\xb_i^2(\tau)-y)]\tilde{f}^1_j(\tau,dy)d\tau}\\ \nonumber
    + &\E\abs{ \sum_{j=1}^N w^N_{ij} \int_0^{t}\ird K_{1}(\xb_i^2(\tau)-y)(\tilde{f}^1_j(\tau,dy)-\tilde{f}^2_j(\tau,dy))d\tau}\\ \nonumber
    \leq & \norm{K_{1}}_{W^{1,\infty}(\R^d)} \sum_{j=1}^N w^N_{ij} \E\left[\int_0^{t}\abs{ \xb_i^1(\tau)-\xb_i^2(\tau)}\ird\Mb^{1}(\tau,y)d\p(y)d\tau\right]\\ \nonumber
+&\norm{K_{1}}_{W^{1,\infty}(\R^d)} \sum_{j=1}^N w^N_{ij} \int_0^{t}\left(\ird\abs{\Mb_j^1(\tau,y) - \Mb_j^2(\tau,y)}d\p(y) + \ird \Mb^2_j(\tau,y)\abs{\xb_j^1(\tau,y) - \xb_j^2(\tau,y)}d\p(y)\right) d\tau\\ \nonumber
\leq &\norm{K_{1}}_{W^{1,\infty}(\R^d)}  C_w \left(2Me^{At_*}\int_0^t \max_{1\leq i \leq N} \E|\xb^1_i(\tau) - \xb^2_i(\tau)| d\tau+\int_0^t \max_{1\leq i \leq N} \E|\Mb^1_i(\tau) - \Mb^2_i(\tau)| d\tau\right), 
\end{align}
where in the last estimate we used the assumption (\ref{aswb}) and the properties of the set $E$.
Moreover, again from the definition of $E$ and since for $k=1,2$ the expression $\sum_{j=1}^N w^N_{ij}\ird K_{2}(\xb^k_i(\tau)-y)\tilde{f}^k_j(\tau,dy)$ is nonnegative ( due to the assumptions (\ref{asK}), (\ref{aswl})), using properties of exponent we may estimate as follows
\begin{align*}
  &\E \Bigg| M^{0}_i\exp\left[\int_0^t \left(A-\sum_{j=1}^N w^N_{ij}\ird K_{2}(\xb^1_i(\tau)-y)\tilde{f}^1_j(\tau,dy)\right) d\tau \right] \\
& - M^{0}_i\exp\left[\int_0^t \left(A-\sum_{j=1}^N w^N_{ij}\ird K_{2}(\xb^2_i(\tau)-y)\tilde{f}^2_j(\tau,dy)\right) d\tau \right] \Bigg| \\
&\leq M e^{At_*} \E \int_0^t \sum_{j=1}^N
w_{ij}^N\abs{\ird K_2(\xb^2_i(\tau) - y)\tilde{f}^2_j(\tau,dy) -\ird K_2(\xb^1_i(\tau) - y)\tilde{f}^1_j(\tau,dy) }d\tau =: (*)
\end{align*}
Adding and subtracting the mixed term $\ird K_{2}(\xb^2_i(\tau,x) - \xb^2_j(\tau,y))\Mb^1_j(\tau,y)d\p(y)$, by triangle inequality we obtain
\begin{align*}
  &\abs{\ird K_2(\xb^2_i(\tau,x) - y)\tilde{f}^2_j(\tau,dy) -\ird K_2(\xb^1_i(\tau,x) - y)\tilde{f}^1_j(\tau,dy) }\\
 &\leq 
\norm{K_2}_{W^{1,\infty}(\R^d)} \left[\ird \Big(\abs{\xb^2_i(\tau,x) - \xb^1_i(\tau,x)} +  \abs{\xb^2_j(\tau,y) - \xb^1_j(\tau,y)}\Big)\Mb^2_j(\tau,y)d\p(y) + \E\abs{\Mb^2_j(\tau) - \Mb^1_j(\tau)}\right]. 
\end{align*}
Hence, again by (\ref{aswb}) and the definition of $E$ we get
\begin{align}\label{secondeq}
(*) \leq & 2\norm{K_2}_{W^{1,\infty}(\R^d)}M^2e^{2At_*}C_w \int_0^t \max_{1\leq i \leq N}\E\abs{\xb_i^2(\tau) - \xb_i^1(\tau)} d\tau \\ \nonumber
+&\norm{K_2}_{W^{1,\infty}(\R^d)}Me^{At_*}C_w \int_0^t \max_{1\leq i \leq N}\E\abs{\Mb_j^2(\tau) - \Mb_j^1(\tau)}d\tau\\ \nonumber
\leq &C(A,M,t_*,C_w,\norm{K_2}_{W^{1,\infty}(\R^d)} )\int_0^{t}\max_{1\leq i \leq N}\E \abs{\xb_i^2(\tau) - \xb_i^1(\tau)} + \max_{1\leq i \leq N}\abs{\Mb_i^2(\tau) - \Mb_i^1(\tau)}d\tau.
\end{align}
Combining (\ref{firsteq}) and (\ref{secondeq}) we obtain that there exists $C$ dependent only on $A,M,t_*,C_w,\norm{K_1}_{W^{1,\infty}(\R^d)}, \norm{K_2}_{W^{1,\infty}(\R^d)}$ such that
\begin{align*}
    d_{\alpha}(G(\xb^1,\Mb^1), G(\xb^2,\Mb^2)) & \leq C \sup_{t\in (0,t_*)}e^{-\alpha t}\int_{0}^{t}\max_{1\leq i \leq N}\E \abs{\xb_i^2(\tau) - \xb_i^1(\tau)} + \max_{1\leq i \leq N}\abs{\Mb_i^2(\tau) - \Mb_i^1(\tau)}d\tau \\
    &\leq \frac{C}{\alpha} d_{\alpha}((\xb^1,\Mb^1), (\xb^2,\Mb^2)).
\end{align*}
Choosing $\alpha = 2 C$ we obtain that $G$ is a contraction on $(E,d_{2C})$. Hence, by the Banach fixed point theorem there exists exactly one solution $(\xb_i,\Mb_i)_{i=1}^N \in E$ which satisfies
\[
\xb_i(t) = X^0_i + \sum_{j=1}^N w^N_{ij} \int_0^{t}\ird K_{1}(\xb_i(\tau)-y)\tilde{f}_j(\tau,dy)d\tau, 
\]
\[
\Mb_i(t) = M^{0}_i\exp\left[\int_0^t \left(A-\sum_{j=1}^N w^N_{ij}\ird K_{2}(\xb_i(\tau)-y)\tilde{f}_j(\tau,dy)\right) d\tau \right] \hd \m{ for } \hd i = 1 , \dots, N.
\]
Since $(\xb_i,\Mb_i)_{i=1}^N \in E$ the right hand sides in the system above belong actually to $C^{1}([0,t_*];L^{1}(\R^d,\p))$. Thus, also $\xb_i,\Mb_i \in C^{1}([0,t_*];L^{1}(\R^d,\p))$ and differentiating the equations we arrive at (\ref{part2}).
It remains to show the independence. Note that $\tilde{f}_i(t)$ for every $t\geq 0$ and $i=1,\dots,N$ is a deterministic measure. Indeed, for any $U \in \mathcal{B}(\mathbb{R}^d)$ we have
\[
\tilde{f}_i(t)(U) = \ird I(\{x:\xb_{i}(t,x)\in U\})\Mb_i(t,x)d\p(x) = \E \left(\Mb_i(t)I_U(\xb_i(t))\right).
\]
Hence, we may rewrite (\ref{part2})$_{(i)}$ as $\partial_t \xb_i(t)=a_i(t,\xb_i(t))$, where, due to (\ref{asK}) and (\ref{aswl}) function $a_i(t,x)$ is Lipschitz continuous in $x$, continuous in $t$ and bounded. Thus, there exists a flow $h_i:([0,t_{*});\R^d)$ continuous in time and Lipschitz in space such that $\xb_i(t) = h_i(t,X^0_i)$. Similarly, $\Mb_i(t)$ satisfies $\Mb_i(t)=M_i^0\exp(\int_0^t\tilde{a}_i(s,\xb_i(s))ds)$, with $\tilde{a}_i(t,x)$ also Lipschitz continuous in $x$, continuous in $t$ and bounded. Thus, $\Mb_i(t)=M_i^0\exp(\int_0^t\tilde{a}_i(s,h_i(s,X^0_i))ds)$. All in all, we may write 
\[
(\xb_i(t),\Mb_i(t)) = \Phi^i_t(X^0_i,M^0_i), \hd \hd \Phi^i_t(x,m):= \left(h_i(t,x),m\exp(\int_0^t\tilde{a}_i(s,h_i(s,x))ds)\right)
\]
and $\Phi^i_t:\R^d \times \R_{+}$ is a deterministic measurable map. Thus, the family $(\xb_i(t),\Mb_i(t))_{i=1}^N$ is jointly independent for any $t \in (0,t_{*})$ for every $N \in \mathbb{N}$, which finishes the proof.
\end{proof}

\begin{proof}[Proof of Proposition \ref{fireg}]
For any set $U \in \mathcal{B}(\R^d)$ we have
\[
\tilde{f}^0_i(U)= \ird I(\{x:X^0_{i}(x)\in U\})M^0_i(x)d\p(x)  = \int_{(X_i^0)^{-1}(U)} M^0_id\p 
= \int_{(X_i^0)^{-1}(U)} \mathbb{E}(M^0_i|X^0_i)d\p,
\]
where in the last equality we used the definition of conditional expectation. Denoting\\ $h_i(x) = \mathbb{E}(M_i^0|X_i^0(x)~=~x)$, we may write
\[
\tilde{f}^0_i(U)= \int_{(X_i^0)^{-1}(U)} h_i(X^{0}_i)d\p = \int_{U} h_i(x)d\p_{X^{0}_i} = \int_{U} h_i(x)f_{X^{0}_i}dx= \int_U g_i(x)dx,
\]
where $g_i$ is defined in (\ref{asf}). Hence $g_i$ is indeed the density of $\tilde{f}^0_i$. Moreover,  by the assumption (\ref{asf}), the density of $\tilde{f}^0_i$ (which below we denote also by $\tilde{f}^0_i$) belongs to $(W^{1,1}\cap W^{1,\infty})(\R^d)$.

Now we will show that this regularity propagates with time. We recall that  $\tilde{f}_i$ satisfies (\ref{semic}), which  may be written as $\partial_t \tilde{f_i} + \divv (\tilde{f}_i v^1_i[\tilde{f}]) = \tilde{f}_i v^2_i[\tilde{f}]$, where
\[
v^1_i[\tilde{f}](t,x) :=  \sum_{j=1}^Nw_{ij}^N\ird K_{1}(x-y)\tilde{f}_j(t,dy), \hd \hd  v^2_i[\tilde{f}](t,x):= A - \sum_{j=1}^Nw_{ij}^N\ird K_{2}(x-y)\tilde{f}_j(t,dy).
\]
Then, using the fact that $v^1_i$ is Lipschitz and $\tilde{f}_i, w, K_2$ are nonnegative we may estimate 
\[
\norm{\tilde{f}_i(t)}_{L^{1}(\R^d)}\leq e^{At}\norm{\tilde{f}^0_i}_{L^{1}(\R^d)}.
\]
Furthermore by (\ref{aswb})
\[
\norm{v^1_i[\tilde{f}]}_{L^1(\R^d)} \leq C_w e^{At}\norm{K_1}_{L^{1}(\R^d)}\max_{1\leq i\leq N}\norm{\tilde{f}^0_i}_{L^{1}(\R^d)}, \norm{v^1_i[\tilde{f}]}_{L^\infty(\R^d)} \leq C_w e^{At}\norm{K_1}_{L^{\infty}(\R^d)}\max_{1\leq i\leq N}\norm{\tilde{f}^0_i}_{L^{1}(\R^d)}
\]
and
\[
\norm{\nabla v^1_i[\tilde{f}]}_{L^1(\R^d)} \leq C_w e^{At}\norm{\nabla K_1}_{L^{1}(\R^d)}\max_{1\leq i\leq N}\norm{\tilde{f}^0_i}_{L^{1}(\R^d)},  \norm{\nabla v^1_i[\tilde{f}]}_{L^\infty(\R^d)} \leq C_w e^{At}\norm{\nabla K_1}_{L^{\infty}(\R^d)}\max_{1\leq i\leq N}\norm{\tilde{f}^0_i}_{L^{1}(\R^d)}.
\]
Clearly, we obtain the analogous estimates for $v^2_i$. Thus, similarly to (\ref{finfty}) we get 
\[
\norm{\tilde{f}_i(t)}_{L^{\infty}(\R^d)}\leq \exp\left(t\left(C_w e^{At}\norm{\divv K_1}_{L^{\infty}(\R^d)}\max_{1\leq i\leq N}\norm{\tilde{f}^0_i}_{L^{1}(\R^d)} + A\right)\right)\norm{\tilde{f}^0_i}_{L^{\infty}(\R^d)}.
\]
We may obtain a priori bounds in $W^{1,1}$ and $W^{1,\infty}$   using the classical estimates for transport equations (\ref{propW11}) and (\ref{gradinfty}), with $v=v^1_i[\tilde{f}]$ and $R = \tilde{f}_i v^2_i[\tilde{f}]$.
Since the terms involving $\nabla R$ involve also $\nabla \tilde{f}_i$ and $\norm{\divv v_i^1[\tilde{f}]}_{W^{1,\infty}(\R^d)}$ involves also $\nabla \tilde{f}_j$, we get the estimates of the form
\begin{align*}
&\norm{\nabla \tilde{f}_i(t)}_{L^{1}(\R^d)}\leq C\left(\norm{\nabla \tilde{f}^0_i(t)}_{L^{1}(R^d)} + \int_{0}^t \norm{\divv v_i^1[\tilde{f}](s,\cdot)}_{W^{1,\infty}(\R^d)} + \norm{\nabla \tilde{f}_i(s)}_{L^{1}(\R^d)}ds\right),\\ 
&\norm{\nabla \tilde{f}_i(t)}_{L^{\infty}(\R^d)}\leq C\left(\norm{\nabla \tilde{f}^0_i(t)}_{L^{\infty}(R^d)} + \int_{0}^t \norm{\divv v_i^1[\tilde{f}](s,\cdot)}_{W^{1,\infty}(\R^d)} + \norm{\nabla \tilde{f}_i(s)}_{L^{\infty}(\R^d)}ds\right),
\end{align*}
where we included in the constant $C$ the constants dependent on $t_{*}$ as well as $L^{\infty}((0,t_{*});(W^{1,1}\cap W^{1,\infty})(\R^d))$ norms of $v^1_i[\tilde{f}],v^2_i[\tilde{f}]$ and $(L^1\cap L^{\infty})(\R^d)$ norms of $\tilde{f}_i^0$. Applying the estimate 
\[
\norm{\divv v_i^1[\tilde{f}](s,\cdot)}_{W^{1,\infty}(\R^d)} \leq C_w \norm{\divv K_1}_{L^{\infty}(\R^d)} \max_{1\leq j\leq N} \norm{\nabla \tilde{f}_j(s,\cdot)}_{L^{1}(\R^d)}
\]
together with the Gr\"onwall inequality applied to the function  $u(t)=\max_{1\leq i \leq N}\norm{\tilde{f}_i(t)}_{L^{1}(\R^d)}$, we obtain the desired bound in $W^{1,1}(\R^d)$, then similarly we obtain the result in $W^{1,\infty}(\R^d)$.  To justify the formal  calculations above we may discuss the approximate system $\partial_t \tilde{f}^\ve_i + \divv (\tilde{f}_i^\ve v^{1,\ve}_i[\tilde{f}^\ve] = \tilde{f}^\ve_i v^{2,\ve}_i[\tilde{f}^\ve]$  where $v^{1,\ve}_i$ and $v^{2,\ve}_i$ are created from $v^1_i$ and $v^2_i$ by replacing kernels $K_1$ and $K_2$ by their convolution with mollifying kernel and the initial condition $\tilde{f}_i^{\ve}(0)$ is also a convolution of $\tilde{f}_i(0)$ with mollifying kernel. Then, since $\tilde{f}_i^{\ve}(0), v^{1,\ve}_i[\tilde{f}^\ve]$ and $v^{2,\ve}_i[\tilde{f}^\ve]$ are $C^{\infty}(\R^d)$ so is $\tilde{f}_i^\ve$ and the above a priori estimates are rigorous for $\tilde{f}_i^{\ve}$. Furthermore, by the properties of convolution with mollifier the estimates are $\ve$- independent. Testing the equation against $\vf \in H^{1}(\R^d)$ and controlling the $\ve$ - independent bound of $\tilde{f}_i^\ve v^{1,\ve}_i[\tilde{f}^\ve]$ and $\tilde{f}_i^\ve v^{2,\ve}_i[\tilde{f}^\ve]$ in $L^{2}(\R^d)$ we obtain the $\ve$ - independent bound for $\partial_t \tilde{f}^\ve_i$ in $L^{\infty}((0,t_{*});H^{-1}(\R^d))$. Thus, by Aubin-Lions lemma $\tilde{f}^\ve_i$ is relatively compact in $C([0,t_{*}];L^{2}(B_s))$ for every $s>0$. This, together with weak-compactness,  allows us to pass to the limit  in the weak formulation of the equation. By weak compactness the limiting $\chi_i$ also satisfies the uniform bounds in $(W^{1,1}\cap W^{1,\infty})(\R^d)$. Finally, we conclude that the limiting $\chi_i$ coincides with  $\tilde{f}_i$ applying the $L^{\infty}((0,t_{*});(\mathcal{M}_{+}(\R^d))^N)$ uniqueness of measure solutions  to system of transport equations with non-local flow  \cite[Theorem 1.1]{Crippa}. Admittedly, in \cite{Crippa} the Authors discuss probabilistic measures being a solution to the problem without reaction term, however considering the flat metric instead of $1$-Wasserstein metric and applying Gr\"onwall lemma extends the result (at least in our simple setting) to accommodate systems with reaction term. 

\end{proof}

\begin{remark}[An alternative proof of Proposition \ref{fireg}]
One may provide the proof of Proposition \ref{fireg} without referring to the equation (\ref{semic}). Recalling the notation $a_i,h_i$ and $\tilde{a}_i$ from the proof of Lemma \ref{exi}, we may write
\[
\partial_t \xb_i(t) = a_i(t,\xb_i(t)), \hd \xb_i(t) = h_i(t,X^0_i), \hd \Mb_i(t)=M_i^0 G_{i}(t,X^0_i), \hd G_i(t,x):=\exp(\int_0^t \tilde{a}_i(s,h_i(s,x)))ds.
\]
Since 
\[
\norm{\nabla a_i(t,\cdot)}_{L^{\infty}(\R^d)}\leq \norm{\nabla K_1}_{L^{\infty}(\R^d)}  \max_{1\leq j\leq N}\norm{\tilde{f}_j(t)}_{L^{1}(\R^d)}\sum_{j=1}^N w_{ij} \leq C,
\]
 we obtain that for every $i \in \{1,\dots,N\}$ $h_i$ is a bi-Lipschitz diffeomorphism and 
\[
e^{-Ct} \leq \det \nabla h_i(t,x)\leq e^{Ct}.
\]
Thus, we may write the density of $\tilde{f}_i(t)$ as
\eqq{
\tilde{f}_i(t,x) = \frac{G_i(t,y)f_i^0(t,y)}{\det \nabla h_i(t,y)}\Big|_{y=h_i^{-1}(t,x)}.
}{fidensform}
Since $\det \nabla h_i(t,x)$ fulfills the equation $\partial_t b(t,x) = \divv a_i(t,x)b(t,x)$ differentiating it with respect to $x$ and applying Gr\"onwall inequality provides the appropriate bound for $\nabla \det \nabla h_i(t,x)$. Estimating the term $\nabla G_i(t,x)$ and differentiating in space formula (\ref{fidensform}) leads to the desired estimates for $\nabla \tilde{f}_i(t)$ in $L^{1}(\R^d)$ and $L^{\infty}(\R^d)$.  
\end{remark}

\subsection{The extension of Glivenko-Cantelli lemma}

In this section we give a proof of Lemma \ref{GC}. We mimic the classical proof for empirical distribution from \cite{Dudley}, however we deal with weighted and not probabilistic measure, which demands careful modifications. First we establish the following covering lemma (cf. \cite[Proposition 3.4.]{Dudley}).

\begin{lemma}\label{cover}
For any measure $\mu$ on $\R^d$, let us denote by $\mathcal{N}(\mu,\ve,\delta)$ the minimal number of disjoint sets of diameter not greater than $2\ve$, which cover $\mathbb{R}^d$ except the set A with $\mu(A) \leq \delta$.
Then, under the assumptions of Lemma \ref{GC}, there exists $K=K(C,M,d) > 0$ such that for every $k\geq 2+\frac{3}{2}d$, every $\ve >0$ and every $N$ there holds $\mathcal{N}(\nu_N,\ve,\ve^{\frac{k}{k-2}}) \leq K \ve^{-k}$, (i.e. there exists  $K > 0$ such that for every $k\geq 2+\frac{3}{2}d$,  $\ve >0$  there exists $A_N^k$ with $\nu_{N}(A_N^k) \leq \ve^{\frac{k}{k-2}}$ such that the minimal number of disjoint sets of diameter not greater than $2\ve$ which cover $\mathbb{R}^{d} \setminus A_N^k$ is not greater than $K \ve^{-k}$). 
\end{lemma}
\begin{proof}
Let us choose arbitrary $k > \max\{2,d\}$, such that
\eqq{
\frac{k}{2(k-2)} \leq \frac{k-d}{d}.
}{kdin}
Note that this inequality is satisfied for every  $k\geq 2 + \frac{3}{2}d$. Indeed  for $k > \max\{2,d\}$ the above inequality is equivalent with $2k^2-(4+3d)k+4d \geq 0$.
Then $\Delta = (4+3d)^2 - 32d >0$ for all $d \geq 1$ and the inequality is fulfilled by any $k$ satisfying
\[
\frac{(4+3d) + \sqrt{(4+3d)^2 - 32d}}{4} \leq \frac{2}{4}(4+3d)  =2 + \frac{3}{2}d \leq k. 
\]
Fix $N \in \mathbb{N}$ and $\ve>0$, using (\ref{xas}) and (\ref{mas}) we have
\[
\ird |x|^2d\nu_N = \sn \ird|X_i(x)|^2 M_i(x)d\mathbb{P}(x) \leq CM.
\]
If $\nu_N(\R^d)\leq \ve^{\frac{k}{k-2}}$ the claim holds trivially with $A = \R^d$ and $\mathcal{N}(\nu_N,\ve,\ve^{\frac{k}{k-2}}) = 0$. If $\nu_N(\R^d) > \ve^{\frac{k}{k-2}}$ we may choose  $r_{N,k}$ such that 
\[
\nu_N(B_{r_{N,k}}) > \nu_N(\mathbb{R}^d) - \ve^{\frac{k}{k-2}} \geq \nu_N(B_\frac{{r_{N,k}}}{2}).
\]
Then $\nu_N(\mathbb{R}^d \setminus B_{r_{N,k}}) < \ve^{\frac{k}{k-2}}$ and $\nu_N(\mathbb{R}^d \setminus B_{\frac{r_{N,k}}{2}}) \geq \ve^{\frac{k}{k-2}}$.
Thus, such choice of $r_{N,k}$ gives
\[
CM \geq \ird |x|^2d\nu_N  \geq \left(\frac{r_{N,k}}{2}\right)^2 \int_{\mathbb{R}^d \setminus B_{\frac{r_{N,k}}{2}}} d\nu_N \geq \left(\frac{r_{N,k}}{2}\right)^2  \ve^{\frac{k}{k-2}}.
\]
Hence, applying inequality (\ref{kdin})
\[
r_{N,k
}\leq 2\sqrt{CM} \ve^{-\frac{k}{2(k-2)}} \leq 2\sqrt{CM} \ve^{-\frac{k-d}{d}}. 
\]
Let $q_N^k$ denote the maximal number of points in $B_{r_{N,k}}$ which are separated for each other by the distance greater than $\ve$ and let us denote by $Q_N^k$ the set of such points (notice that the choice of $Q_N^k$ is not  unique). 
 Then, $B_{r_{N,k}} \subset \cup_{x\in Q_N^k}B_{\ve}(x)$.
 On the other hand, since the balls centered in points belonging to $Q_N^k$ with radius $\ve/2$ are disjoint we have
\[
q_N^k \left(\frac{\ve}{2}\right)^d \leq (r_{N,k}+\ve)^d. 
\]
Summing up, 
\[
\mathcal{N}(\nu_N,\ve,\ve^{\frac{k}{k-2}}) \leq q_N^k \leq  2^d(1+r_{N,k}\ve^{-1})^d \leq 2^d(1+2\sqrt{CM}\ve^{-\frac{k}{d}})^d \leq K \ve^{-k}, 
\]
where $K =K(C,M,d)$.
\end{proof}
Now we are ready to give a proof of Lemma \ref{GC}.
\begin{proof}[Proof of Lemma \ref{GC}]
First we note that for any $U\in \mathcal{B}(\mathbb{R}^d)$
\eqq{
\mathbb{E}\mu_N(U) = \sn \ird M_i(x) I_{\{x:X_i(x) \subset U\}}d\mathbb{P}(x)
=\sn\ird I_{\{x \subset U\}} d\tilde{f}_i(x) = \int_{U}d\nu_{N} = \nu_{N}(U).
}{wo}
Furthermore,
\[
\mathbb{E}(\mu_N(U))^2 = \ird \hspace{-0.1cm}\left(\sn  M_i(x) I_{\{x:X_i(x) \subset U\}}\right)^2\hspace{-0.2cm} d\mathbb{P}(x)
=\frac{1}{N^2}\hspace{-0.1cm}\ird \sum_{i,j=1}^NM_i(x)I_{\{x:X_i(x) \subset U\}}M_j(x)I_{\{x:X_j(x) \subset U\}}d\mathbb{P}(x).
\]
For fixed Borel measurable set $U$ define the random variable $Y_i^U:=M_iI_{\{x:X_i(x) \subset U\}}$. Then we may rewrite it as $Y_i^U = \Phi((M_i,X_i))$ with Borel function $\Phi:\R \times \R^d \rightarrow \R$ defined by $\Phi((m,x)):=m I_U(x)$.
Since by the assumption for $i \neq j$ $(M_i,X_i)$ and $(M_j,X_j)$ are independent, then also $Y^U_i$ and $Y^U_j$ are independent for $i \neq j$. Thus, we may continue as follows
\begin{align*}
   \mathbb{E}(\mu_N(U))^2 &=
\frac{1}{N^2}\sum_{i=1}^N\sum_{j\neq i}^N \ird Y^U_i(x)d\mathbb{P}(x)\ird Y^U_j(x)d\mathbb{P}(x) +\frac{1}{N^2}\sum_{i=1}^N \ird M_i^2(x)I_{\{x:X_i(x) \subset U\}}d\mathbb{P}(x)\\
&= \frac{1}{N^2}\sum_{i=1}^N\sum_{j\neq i}^N d\tilde{f}_i(U)d\tilde{f}_j(U) + \frac{1}{N^2}\sum_{i=1}^N \ird M_i^2(x)I_{\{x:X_i(x) \subset U\}}d\mathbb{P}(x)\\
&= (\nu_N(U))^2 - \frac{1}{N^2}\sum_{i=1}^N (d\tilde{f}_i(U))^{2}+ \frac{1}{N^2}\sum_{i=1}^N \ird M_i^2(x)I_{\{x:X_i(x) \subset U\}}d\mathbb{P}(x).
\end{align*}
From the calculations above (\ref{wo}) and assumption (\ref{mas}) we obtain
\begin{align*}
\mathbb{E}|\mu_N(U) - \mathbb{E}\mu_N(U)|^2 &= \mathbb{E}(\mu_N(U))^{2} - (\mathbb{E}\mu_N(U))^2
= \frac{1}{N^2}\sum_{i=1}^N \ird M_i^2(x)I_{\{x:X_i(x) \subset U\}}d\mathbb{P}(x)- \frac{1}{N^2}\sum_{i=1}^N (d\tilde{f}_i(U))^{2}\\
& \leq \frac{1}{N^2}\sum_{i=1}^N \ird M_i^2(x)I_{\{x:X_i(x) \subset U\}}d\mathbb{P}(x) \leq \frac{M \nu_N(U)}{N}.
\end{align*}
Hence, taking a disjoint finite sum $U = \cup_{j=1}^{m}U_j$, applying firstly H\"older's inequality and subsequently Jensen inequality we arrive at
\[
\mathbb{E}\sum_{j=1}^m|(\mu_N - \nu_N)(U_j)| \leq
\mathbb{E}\left(\sum_{j=1}^m|(\mu_N - \nu_N)(U_j)|^2\right)^{\frac{1}{2}}m^{\frac{1}{2}} 
\leq \left(\mathbb{E}\sum_{j=1}^m|(\mu_N - \nu_N)(U_j)|^2\right)^{\frac{1}{2}}m^{\frac{1}{2}} 
\]
\eqq{
\leq \left(\sum_{j=1}^m\frac{M\nu_N(U_j)}{N}\right)^{\frac{1}{2}}m^{\frac{1}{2}} = \left(\frac{M\nu_N(U)m}{N}\right)^{\frac{1}{2}}.
}{momest}
Let us take $k=2+\frac{3}{2}d$ and fix $N\in \mathbb{N}$. Taking $\ve=3^{-(r+2)}$ for any $r>0$ in Lemma \ref{cover}  we may decompose $\mathbb{R}^d$ for a finite sum of pairwise disjoint sets $S^N_{r,j}$ as follows
\eqq{
\mathbb{R}^d = \cup_{j=0}^{m_r}S^N_{r,j}, \hd  m_r \leq 3^{k(r+2)}K, \hd \diam|S^N_{r,j}| \leq 3^{-r-1} \m{ for } j \geq 1, \hd \nu_{N}(S^N_{r,0}) \leq 3^{-\frac{k(r+2)}{k-2}}.
}{coverprop}
Let us now set $\ve = N^{-\frac{1}{k}}$ and choose as $r_1$ the smallest integer such that $3^{-r_1} < \ve$. Then $3^{-(r_1-1)} \geq \ve$, which gives $3^{r_1} \leq 3 \ve^{-1}$.
Let $r_0$ be the smallest integer such that $3^{-r_0} < \ve^{\frac{k-2}{k}}$, then $3^{r_0}\leq 3 \ve^{\frac{2-k}{k}}$ and $r_0\leq r_1$.
Let us construct for any $r=r_0,\dots, r_1$ the pairwise disjoint family of sets $\{A_{r,j}\}_{j}$. For clarity of the notation we omit index $N$ in the notation of sets $\{A_{r,j}\}_{j}$. We define $A_{r_1,j} = S^N_{r_1,j}$ for $j=1,\dots,m_{r_1}$. Now having $A_{r,j}$ we construct the family $A_{r-1,j}$ in the following way. If $A_{r,j}$ is not included in $S^N_{r-1,0}$, then it intersects some $S^N_{r-1,q}$ for $q \geq 1$. We choose arbitrarily exactly one $S^N_{r-1,q}$ which it intersects. In this way we may construct a function $q^r(j)$, which for every $j=1,\dots,m_r$ specifies the index $q = 1,\dots,m_{r-1}$ for which $A_{r,j}$ intersects $S^N_{r-1,q}$, choosing arbitrarily only one such set each time. Then we define 
\eqq{
A_{r-1,z} = \cup\{A_{r,j}: q^r(j)=z\}.
}{Adef}
Thus, each $A_{r-1,q}$ is a disjoint union of $A_{r,j}$, which are intersected by $S^N_{r-1,q}$. Hence, for $j=1,\dots,m_r$
\[
\diam (A_{r,j}) \leq \diam (S^N_{r,j}) + 2\max_{q^{r+1}(p)=j} \diam(A_{r+1,p}) \leq 3^{-r-1} + 2\max_{q^{r+1}(p)=j} \diam(A_{r+1,p}).
\]
Iterating this estimate we obtain
\eqq{
\diam (A_{r,j}) \leq 3^{-r}\sum_{l=1}^{r_1-r+1}3^{-l} 2^{l-1} \leq 3^{-r}\frac{1}{2}\frac{2}{3}\frac{1}{1-\frac{2}{3}} \leq 3^{-r}.
}{diam}
Furthermore, we note that for each $r=r_0+1,\dots,r_1$
\[
\bigcup_{j=1,\dots,m_r}A_{r,j} \subset \left(S^N_{r-1,0} \cup \bigcup_{j = 1,\dots,m_{r-1}}A_{r-1,j}\right). 
\]
For each $r =r_0,\dots,r_1$ we introduce a random variable 
\[
M_r:=\sum_{j=1}^{m_r}|(\mu_N(\cdot) - \nu_{N})(A_{r,j})|.
\]
Thanks to the construction above, we are able to estimate the flat metric distance between $\mu_N$ and $\nu_N$. To that end,
fix arbitrary $\phi \in W^{1,\infty}(\R^d)$ with $\norm{\phi}_{L^{\infty}(\R^d)} \leq 1, \hd \norm{\nabla \phi}_{L^{\infty}(\R^d)} \leq 1$. For each $r=r_0,\dots,r_1$ and $j=1,\dots,m_r$, we choose (if possible) $x_{r,j} \in A_{r,j}$ and denote $\phi_{r,j}:=\phi(x_{r,j})$. Then, for $r=r_0+1,\dots,r_1$, whenever $A_{r,j} \not\subset S^N_{r-1,0}$ we have
\eqq{
\abs{\phi_{r,j} - \phi_{r-1,q^r(j)}} \leq |x_{r,j} - x_{r-1,q^r(j)}| \leq 3^{1-r},
}{est1}
where in the last inequality we used that $A_{r,j} \subset A_{r-1,q^r(j)}$ and estimate (\ref{diam}). Hence,
\begin{align*}
&\abs{\ird \phi d(\mu_N - \nu_{N})} =
\abs{\int_{\cup_{j=0}^{m_{r_{1}}}S^N_{r_1,j}} \phi d(\mu_N - \nu_{N})}\\
& \leq (\mu_{N} + \nu_{N})(S^N_{r_1,0}) + 
\abs{\sum_{j=1}^{m_{r_1}}\int_{A_{r_1,j}}(\phi(x)-\phi_{r_1,j}+\phi_{r_1,j}) d(\mu_N - \nu_{N})(x)} \\
& \leq (\mu_{N}(x) + \nu_{N})(S^N_{r_1,0}) + 
\abs{\sum_{j=1}^{m_{r_1}}\phi_{r_1,j} (\mu_N - \nu_{N})(A_{r_1,j})}
+ \sum_{j=1}^{m_{r_1}}\sup_{x\in A_{r_1,j}}|x-x_{r_1,j}|\abs{\mu_N - \nu_{N}}(A_{r_1,j})\\
& \leq (\mu_{N} + \nu_{N})(S^N_{r_1,0})  + 
\abs{\sum_{j=1}^{m_{r_1}}\phi_{r_1,j} (\mu_N - \nu_{N})(A_{r_1,j})} + 2M 3^{-r_1},
\end{align*}
where in the last estimate we used (\ref{diam}), the fact that $A_{r_1,j}$ are disjoint and the assumption (\ref{mas}).
Further, applying (\ref{Adef}) and (\ref{est1}) and recalling $3^{-r_1} < \ve$ we get
\begin{align*}
 & \abs{\ird \phi d(\mu_N - \nu_{N})} \leq (\mu_{N} + \nu_{N})(S^N_{r_1,0}\cup S^N_{r_1-1,0}) + 
 2M \ve\\
 & +\abs{\sum_{q=1}^{m_{r_1-1}}\sum_{j:q^r_1(j) = q}(\phi_{r_1,j} - \phi_{r_1-1,j} +\phi_{r_1-1,j} )(\mu_N - \nu_{N})(A_{r_1,j})} \\
 & \leq (\mu_{N} + \nu_{N})(S^N_{r_1,0}\cup S^N_{r_1-1,0}) + 
 2M \ve
+\abs{\sum_{q=1}^{m_{r_1-1}}\phi_{r_1-1,j}(\mu_N - \nu_{N})(A_{r_1-1,q})}
+ 3^{1-r_1}M_{r_1}.
\end{align*}
Iterating this procedure up to $r=r_0$ and recalling that $|\phi| \leq 1 $, we arrive at
\[
\abs{\ird \phi d(\mu_N - \nu_{N})} \leq 2M\ve + \sum_{r=r_0}^{r_1}\left((\mu_N + \nu_N)(S^N_{r,0}) + 3^{1-r}M_r\right) + M_{r_0}. 
\]
Hence, applying (\ref{wo}) and  (\ref{momest})
\begin{align*}
    \mathbb{E}d_{F}(\mu_N,\nu_N) &\leq\\
    &2M\ve +\sum_{r=r_0}^{r_1}\mathbb{E}(\mu_N + \nu_N)(S^N_{r,0}) + \sum_{r=r_0}^{r_1} 3^{1-r}\sum_{j=1}^{m_r}\mathbb{E}|(\mu_N - \nu_{N})(A_{r,j})|+ \sum_{j=1}^{m_{r_0}}\mathbb{E}|(\mu_N - \nu_{N})(A_{r_0,j})|\\
    & \leq 2M\ve +2\sum_{r=r_0}^{r_1} \nu_{N}(S^N_{r,0}) + \sum_{r=r_0}^{r_1}3^{1-r} \left(\frac{m_{r}M\nu_N(\mathbb{R}^d)}{N}\right)^{\frac{1}{2}} + \left(\frac{m_{r_0}M\nu_N(\mathbb{R}^d)}{N}\right)^{\frac{1}{2}}.
\end{align*} 
Due to (\ref{coverprop}) and $\nu_N(\R^d)\leq M$ we further get
\[
\mathbb{E}d_{F}(\mu_N,\nu_N) \leq 2M\ve +2\sum_{r=r_0}^{r_1} 3^{-\frac{k(r+2)}{k-2}} + M\sqrt{\frac{K}{N}} \sum_{r=r_0}^{r_1} 3^{1-r}3^{\frac{k(r+2)}{2}} +M\sqrt{\frac{K}{N}}3^{\frac{k(r_0+2)}{2}}.
\]
Since
\[
\sum_{r=r_0}^{r_1}a^r = a^{r_0}\frac{1-a^{r_1-r_0+1}}{1-a} \m{ for } a\neq 1,
\]
we have
\[
\mathbb{E}d_{F}(\mu_N,\nu_N) \leq 2M\ve + 2\cdot3^{-\frac{2k}{k-2}}\frac{3^{-\frac{kr_0}{k-2}}}{1-3^{-\frac{k}{k-2}}}+ M\sqrt{\frac{K}{N}}3^{(k+1+\frac{(k-2)r_0}{2})}\frac{3^{(\frac{k-2}{2})(r_1-r_0+1)}-1}{3^{\frac{k-2}{2}}-1} + M\sqrt{\frac{K}{N}}3^{\frac{k(r_0+2)}{2}}.
\]
Applying $3^{r_0} \leq 3\ve^{\frac{2-k}{k}}$, $3^{-r_0} < \ve^{\frac{k-2}{k}}$ and $3^{r_1} \leq 3\ve^{-1}$ we get
\[
\mathbb{E}d_{F}(\mu_N,\nu_N) \leq 2M\ve + c(k)\ve + M\sqrt{\frac{K}{N}}c(k) \ve^{-\frac{k-2}{2}} +  M\sqrt{\frac{K}{N}}c(k) \ve^{-\frac{k-2}{2}}.
\]
Finally, recalling $\ve = N^{-\frac{1}{k}}$ we arrive at
\[
\mathbb{E}d_{F}(\mu_N,\nu_N) \leq C(K,M,k)\ve,
\]
which finishes the proof since $k=2+\frac{3}{2}d$.
\end{proof}

\subsection{Proof of Proposition \ref{existenceprop}}
In this section, in order to obtain the existence of solutions to (\ref{meannu}), we repeat the Banach fixed point argument, presented in \cite{JPS2025}. The only modification is the presence of the additional lower order term, which results in a bit lengthy computations.

\begin{proof}[Proof of Proposition \ref{existenceprop}]
Let us define
\[
E:=\{f \in L^\infty((0,\ t_*)\times (0,\ 1);\, W^{1,1}\cap W^{1,\infty}(\mathbb{R}^d)): f \geq 0\},
\]
and for any $\Upsilon>0$ to be determined later  denote
\[
E_\Upsilon:=\left\{f\in E:\, \norm{f}_{E}:=\Vert f\Vert_{L^\infty((0,t_{*})\times (0,1));W^{1,1}\cap W^{1,\infty}(\R^d)}\leq \Upsilon\right\}.
\]
Then, $E_\Upsilon$ is a closed subset of $L^1((0,\ t_*)\times (0,\ 1)\times \mathbb{R}^d)$ under the $L^1$ norm and $(E_\Upsilon,\norm{\cdot
}_{L^{1}((0,t_*)\times(0,1)\times \R^d)})$ is a complete and convex metric space.
For a given $g\in E_\Upsilon$ discuss the linear problem
\begin{equation}\label{independ2-linear}
\left\{
\begin{array}{l}
\displaystyle \partial_t f(t,x,\xi) + \divv ( f(t,x,\xi) V_1[g](t,x,\xi)) =  f(t,x,\xi) (A -V_2[g](t,x,\xi)) + \nu \Delta_x f(t,x,\xi),\\
f(0,\cdot,\cdot)=f^0
\end{array}
\right.
\end{equation}
and define its solution operator $P$ as follows $Pg = f$, where $f$ is a weak solution to (\ref{independ2-linear}). We will show that $P$ is well defined and $P:E_\Upsilon\rightarrow E_\Upsilon$ for appropriately chosen $\Upsilon>0$ and $t_{*} > 0$.

First, we note that for any $g \in E_\Upsilon$, we obtain $V_i[g] \in L^{\infty}((0,t_*)\times (0,1);W^{1,\infty}\cap W^{1,1}(\R^d))$ for $i=1,2$. Furthermore, Lemma \ref{impoest} together with the change of variables in convolution gives the following estimates for $i=1,2$:
\begin{align}
&\abs{V_i[g]} \leq \norm{w}\norm{K_i}_{L^{\infty}(\R^d)}\norm{g}_{L^{\infty}((0,t_*)\times (0,1);L^{1}(\R^d))},\label{v1}\\
 &\norm{ V_1[g]}_{ L^{\infty}((0,t_*)\times (0,1);W^{1,\infty}(\R^d))} \leq \norm{w}\norm{K_1}_{W^{1,1}(\R^d)}\norm{g}_{L^{\infty}((0,t_*)\times (0,1);L^{\infty}(\R^d))},
\label{v2}\\
&\norm{ V_i[g](t,\cdot,\cdot)}_{ L^{\infty}( (0,1);W^{1,1}(\R^d))} \leq \norm{w}\norm{K_i}_{L^{1}(\R^d)}\norm{g(t,\cdot,\cdot)}_{L^{\infty}( (0,1);W^{1,1}(\R^d))},
\label{v3}\\
&\norm{ \nabla V_2[g](t,\cdot,\cdot)}_{ L^{\infty}( (0,1);L^{\infty}(\R^d))} \leq \norm{w}\norm{K_2}_{L^{1}(\R^d)}\norm{\nabla g(t,\cdot,\cdot)}_{L^{\infty}( (0,1);L^{\infty}(\R^d))},
\label{v3prim}\\
&\norm{ \divv V_1[g]}_{ L^{\infty}((0,t_*)\times (0,1);L^{\infty}(\R^d))} \leq \norm{w}\norm{\divv K_1}_{L^{\infty}(\R^d)}\norm{g}_{L^{\infty}((0,t_*)\times (0,1);L^{1}(\R^d))}\label{vdiv}\\
&\norm{ \divv V_1[g](t,\cdot,\cdot)}_{ L^{\infty}( (0,1);W^{1,\infty}(\R^d))} \leq \norm{w}\norm{\divv K_1}_{L^{\infty}(\R^d)}\norm{g(t,\cdot,\cdot)}_{L^{\infty}( (0,1);W^{1,1}(\R^d))}.\label{v4}
\end{align}
Problem \eqref{independ2-linear} is a linear equation parametrized by $\xi$, hence, for $\nu=0$ by Di Perna-Lions theory (\cite{diPerna}) there exists a unique solution to (\ref{independ2-linear}) in $L^{\infty}((0,t_{*})\times (0,1)\times \R^{d})$, which 
may be obtained by the method of characteristics, namely
\eqq{
f(t,\cdot,\xi)=X_g(t,\cdot\,,\xi)_{\#}\left[f^0(\cdot,\xi)\exp\left(At - \int_0^tV_2[g](s,X_g(s,x,\xi),\xi))ds\right)\right],
}{integro}
where $X_g$ denotes the flow solving the equation of  characteristic for almost all $\xi \in (0,1)$
\[
\left\{
\begin{array}{l} 
\displaystyle \frac{dX_g}{dt}(t,x,\xi)=V_1(t,X_g(t,x,\xi),\xi),\\
X_g(0,x,\xi)=x.
\end{array}
\right.
\]
Since we assume $f^0 \geq 0$, the formula (\ref{integro}) guarantees that $f \geq 0$ almost everywhere.
For $\nu>0$, we may instead use the heat kernel $G_\nu(t,x)$ and we can again classically obtain a solution through the mild formulation
\begin{equation}\label{mildformulation}
  \begin{split}
    &f(t,x,\xi)=G_\nu(t,.)\star_x f^0 (x,\xi)\\
    &\qquad+\int_0^t\int_{\R^d} G_\nu(t-s,x-y)\,\left[-\divv_y\left(f(s,y,\xi)\,V_1[g](s,y,\xi)\right) + f(s,y,\xi)(A-V_2[g](s,y,\xi))\right]\,dy\,ds.
\end{split}
  \end{equation}
  To show that $f$ is nonnegative also in the case $\nu > 0$ we test the equation with the negative part of f which we denote $f_{-}$. Then using the fact that $g,K_2,w \geq 0$ and  integrating by parts we get
\[
\frac{1}{2}\partial_{t}\int_{\R^d} |f_{-}|^2(t,x,\xi)dx \leq \frac{1}{2}\norm{\divv V_1[g]}_{L^{\infty}((0,t_{*})\times (0,1)\times \R^d)}\int_{\R^d} |f_{-}|^2(t,x,\xi)dx +  A \int_{\R^d} |f_{-}|^2(t,x,\xi)dx.
\] 
Applying  (\ref{vdiv}) and subsequently Gr\"onwall inequality we  get
\[
\int_{\R^d} |f_{-}|^2(t,x,\xi)dx \leq \exp(2A+\norm{w}\norm{\divv K_1}_{L^{1}(\R^d)}\norm{g}_{E})\int_{\R^d} |f^0_{-}|^2(x,\xi)dx
\]
and since $f^0 \geq 0$ we obtain $f \geq 0$ also in a case of $\nu > 0$.
In order to show that the solution belongs to $E_\Upsilon$ we firstly recall classical estimates for advection-diffusion equations. Consider any $v(t,x)\in L^\infty((0,\ t_*);\;W^{1,\infty}(\R^d))$, any weak solution $u$ in $L^1_{loc}$ to
\begin{equation}
\left\{\begin{array}{l}
\partial_t u+\divv(u\,v)=\nu\,\Delta u+R(t,x),\\
u(0,x)=u^0(x).
\end{array}\right.
\label{advectiondiffusion}\end{equation}
Then if $R\in L^1((0,t_*)\times \R^d)$, we have that $u\in L^\infty((0,t_*); L^1(\R^d))$ with
\begin{equation}
\|u(t,\cdot)\|_{L^1(\R^d)}\leq \|u^0\|_{L^1(\R^d)}+\int_0^t \|R(s,\cdot)\|_{L^1(\R^d)}\,ds.\label{propL1}
\end{equation}
If $R\in L^1((0,t_*);L^{\infty}(\R^d))$, then $u\in L^\infty((0,t_*)\times \R^d)$ with
\begin{equation}
\begin{split}
  \|u(t,\cdot)\|_{L^\infty(\R^d)}\leq &\|u^0\|_{L^\infty(\R^d)}\,\exp\left(t\, \|\divv v\|_{L^\infty((0,t_*)\times \R^d)}\right)\\
  &+\int_0^t \|R(s,\cdot)\|_{L^\infty(\R^d)}\,\,\exp\left((t-s)\, \|\divv v\|_{L^\infty((0,t_*)\times \R^d)}\right)\,ds.\label{propLinfty}
\end{split}
\end{equation}
Furthermore, better regularity of $u^0$, $v$ and $R$, implies better regularity of $u$. Indeed, if $v$ is $C^2$ and $R\in C^1$, then $u$ is a classical solution to \eqref{advectiondiffusion} and we may differentiate to the result
\[
\partial_t \nabla u+\divv(v\,\nabla u)+\divv(\nabla v\,u)=\nu\,\Delta \nabla u+\nabla R.
\]
Hence, $\nabla u$ is a solution to \eqref{advectiondiffusion} with $\tilde R=\nabla R-\divv(\nabla v\,u)$. Assume that $u^0\in W^{1,1}\cap W^{1,\infty}(\R^d)$,  then
\[
\|\tilde R(t,\cdot)\|_{L^1(\R^d)}\leq \|\divv v(t,\cdot)\|_{W^{1,\infty}(\R^d)}\,\|u(t,\cdot)\|_{L^1(\R^d)}+\| v(t,\cdot)\|_{W^{1,\infty}(\R^d)}\,\|\nabla u(t,\cdot)\|_{L^1(\R^d)} + \norm{\nabla R(t,\cdot)}_{L^{1}(\R^d)}
\]
and
\[
\|\tilde R(t,\cdot)\|_{L^\infty(\R^d)}\leq \|\divv v(t,\cdot)\|_{W^{1,\infty}(\R^d)}\,\|u(t,\cdot)\|_{L^\infty(\R^d)}+\| v(t,\cdot)\|_{W^{1,\infty}(\R^d)}\,\|\nabla u(t,\cdot)\|_{L^\infty(\R^d)} + \norm{\nabla R(t,\cdot)}_{L^{\infty}(\R^d)}.
\]
Hence, applying (\ref{propL1}) we arrive at
\begin{align*}
    \norm{\nabla u(t,\cdot)}_{L^{1}(\R^d)} &\leq \norm{\nabla u^0}_{L^{1}(\R^d)}+\int_0^t \norm{\divv v(s,\cdot)}_{W^{1,\infty}(\R^d)}\,\norm{u(s,\cdot)}_{L^{1}(\R^d)}
+\norm{\nabla R(s,\cdot)}_{L^{1}(\R^d)}ds\\
&+ \| v\|_{L^{\infty}((0,t_{*});W^{1,\infty}(\R^d))}\int_0^t \,\|\nabla u(s,\cdot)\|_{L^1(\R^d)}ds
\end{align*}
and by Gr\"onwall lemma 
\begin{equation}
  \begin{split}
    &\|\nabla u(t,\cdot)\|_{L^1(\R^d)}\leq \|\nabla u^0\|_{L^1(\R^d)}\,\exp\left(t\,\| v\|_{L^\infty((0,t_{*});W^{1,\infty} (\R^d))}\right)\\
        &+\int_0^t \exp\left((t-s)\,\| v\|_{L^\infty((0,t_*); W^{1,\infty}(\R^d))}\right)\left(\|\divv v(s,\cdot)\|_{W^{1,\infty}(\R^d)}\| u(s,\cdot)\|_{L^1(\R^d)}+ \norm{\nabla R(s,\cdot)}_{L^{1}(\R^d)} \right)\,ds.
\end{split}\label{propW11}
  \end{equation}
Furthermore, by (\ref{propLinfty})
\begin{align*}
&\norm{\nabla u(t,\cdot)}_{L^{\infty}(\R^d)} \leq \exp\left(t\norm{\divv v}_{L^{\infty}((0,t_*)\times \R^d)}\right)  \norm{\nabla u^0}_{L^{\infty}(\R^d)} \\
&+\int_0^t \exp\left((t-s)\norm{\divv v}_{L^{\infty}((0,t_*)\times \R^d)}\right) \|\divv v(s,\cdot)\|_{W^{1,\infty}(\R^d)}\,\|u(s,\cdot)\|_{L^\infty(\R^d)}ds\\
&+\int_0^t \exp\left((t-s)\norm{\divv v}_{L^{\infty}((0,t_*)\times \R^d)}\right)\left[
\|v\|_{L^{\infty}((0,t_*);W^{1,\infty} (\R^d)}\,\|\nabla u(s,\cdot)\|_{L^\infty(\R^d)} + \norm{\nabla R(s,\cdot)}_{L^{\infty}(\R^d)}\right]ds
\end{align*}
and again by Gr\"onwall lemma we arrive at,
\begin{align}\label{gradinfty}
 \|\nabla u(t,\cdot)\|_{L^\infty(\R^d)}&\leq \|\nabla u^0\|_{L^\infty(\R^d)}\,\exp\left(t \left( \|\divv v\|_{L^\infty((0,t_{*})\times \R^d)} +\norm{v}_{L^{\infty}((0,t_*);W^{1,\infty}(\R^d))} \right)\right)\\ \nonumber
 &+\int_0^t \exp\left((t-s)\,\left(\| v\|_{L^\infty((0,t_*); W^{1,\infty}(\R^d))}+\|\divv v\|_{L^\infty((0,t_{*})\times \R^d)}\right)\right)\alpha(s) ds, 
\end{align}
where
\[
\alpha(s) := \|\divv v(s,\cdot)\|_{W^{1,\infty}(\R^d)}\|u(s,\cdot)\|_{L^{\infty}(\R^d)} + \norm{\nabla R(s,\cdot)}_{L^{\infty} (\R^d)}\,.
\]
If $v$ belongs only to $(W^{2,\infty}\cap W^{1,1})(\R^d)$ and $R \in (W^{1,1}\cap W^{1,\infty})(\R^d)$, then in case $\nu > 0$, by the smoothing properties of $\Delta$ the a priori estimates (\ref{propW11}) and (\ref{gradinfty}) are justified. In case $\nu=0$, in order to justify (\ref{propW11})-(\ref{gradinfty}) we may apply the same reasoning as in the proof of Proposition \ref{fireg}. Namely, obtain $\ve$- independent estimates for solutions $u^\ve$ to problem with regularized  velocity $v^\ve$. Deduce the uniform bound for $ \partial_t u^{\ve}$ in $L^{\infty}((0,t_{*});H^{-1}(\R^d))$ and pass to the limit in the weak formulation of the equation with the help of Aubin-Lions lemma. Since, unlike in Proposition \ref{fireg}, here the problem is linear, it is sufficient to refer to \cite[Corollary II.1]{diPerna} for the uniqueness of weak solutions in $L^{\infty}((0,t_{*});L^1(\R^d))$. Thus, the weak$^*$ compactness argument justifies the gradient bounds also in case $v\in (W^{1,1}\cap W^{1,\infty})(\R^d)$ with $\divv v \in W^{1,\infty}(\R^d)$. 
We may now use the estimates above for the solutions to (\ref{independ2-linear}). Indeed, for a given $g\in E$ equation (\ref{independ2-linear}) may be written in the form \eqref{advectiondiffusion} where $\xi$ is only a parameter and $R$ is linear w.r.t. $f$, i.e. 
\[
v_\xi=V_1[g], \hd R_\xi(t,x) = f (A -V_2[g] ). 
\]
Using the special form of our equation and since
 we already established that the solution $f\geq 0$, we may obtain simpler estimates for $L^1$ and $L^
 \infty$ - norms of $f$. Indeed, since we assume $g,w,K_2 \geq 0$ we obtain that $R_{\xi} \leq Af$ and integrating (\ref{independ2-linear}) we arrive at 
\eqq{
\norm{f}_{L^{\infty}((0,t_*)\times (0,1); L^{1}(\R^d))} \leq e^{At_{*}}\norm{f^0}_{L^{\infty}( (0,1); L^{1}(\R^d))}.
}{l1bound}
Moreover, testing (\ref{independ2-linear}) with $f^{p-1}$ and passing to the limit with $p\rightarrow \infty$ we obtain
\[
\norm{f}_{L^{\infty}((0,t_*)\times (0,1)\times \R^d)} \leq \exp\left((\norm{\divv V_1[g]}_{L^{\infty}((0,t_{*})\times (0,1)\times \R^d)}+A)t_{*}\right)\norm{f^0}_{L^{\infty}( (0,1)\times \R^d)}.
\]

Thus, from (\ref{vdiv}) we infer
\eqq{
\norm{f}_{L^{\infty}((0,t_{*})\times (0,1)\times \R^d)}
\leq \exp\left(t_{*}\left(\norm{\divv K_{1}}_{L^{\infty}(\R^d)}\norm{w}\norm{ g}_{L^{\infty}((0,t_{*})\times (0,1);L^{1}(\R^d))} +A\right)\right)\norm{f^0}_{L^{\infty}((0,1)\times \R^d)}.
}{finfty}
It remains to estimate gradient terms. 

Let us estimate the gradient terms of $R$. By (\ref{v1}) and (\ref{v3})
\begin{align}\label{Rgrad1}
  \norm{\nabla R(s,\cdot)}_{L^{\infty}((0,1);L^{1}(\R^d))} & \leq \norm{\nabla f(s,\cdot,\cdot)}_{L^{\infty}((0,1);L^{1}(\R^d))}\left(A+\norm{w}\norm{K_{2}}_{L^{\infty}(\R^d)}\norm{g(s,\cdot,\cdot)}_{L^{\infty}((0,1);L^{1}(\R^d))}\right)\\ \nonumber  
  &+\norm{f(s,\cdot,\cdot)}_{L^{\infty}( (0,1)\times\R^d)}\norm{\nabla g(s,\cdot,\cdot)}_{L^{\infty}((0,1);L^1(\R^d))}   \norm{w}\norm{K_{2}}_{L^{1}(\R^d)}
\end{align}
and similarly 
by (\ref{v1}) and (\ref{v3prim})
\begin{align}\label{Rgradinfty}
\norm{\nabla R(s,\cdot)}_{L^{\infty}((0,1);L^{\infty}(\R^d))} &\leq \norm{\nabla f(s,\cdot,\cdot)}_{L^{\infty}((0,1);L^{\infty}(\R^d))}\left(A+\norm{w}\norm{K_{2}}_{L^{\infty}(\R^d)}\norm{g}_{L^{\infty}((0,t_*)\times(0,1);L^{1}(\R^d))}\right) \\ \nonumber
& +\norm{f(s,\cdot,\cdot)}_{L^{\infty}( (0,1)\times\R^d)}\norm{\nabla g(s,\cdot,\cdot)}_{L^{\infty}((0,1);L^\infty(\R^d))}      \norm{w}\norm{K_{2}}_{L^{1}(\R^d)}.
\end{align}
Hence, combining
(\ref{propW11}) with (\ref{v2}), (\ref{v4}) and  (\ref{Rgrad1}) leads to

\[
  \begin{split}
    &\|\nabla f(t,\cdot,\cdot)\|_{L^{\infty}((0,1); L^1(\R^d))}\leq \|\nabla f^0\|_{L^{\infty}((0,1);L^1(\R^d))}\,\exp\left(t \norm{w}\norm{K_1}_{W^{1,1}(\R^d)}\norm{g}_{L^{\infty}((0,t_{*})\times (0,1)\times \R^d)}\right)\\
        &+\int_0^t \exp\left((t-s)\,\norm{w}\norm{K_1}_{W^{1,1}(\R^d)}\norm{g}_{L^{\infty}((0,t_{*})\times (0,1)\times \R^d)}\right)\\
        &\times\left( \norm{\nabla f(s,\cdot,\cdot)}_{L^{\infty}((0,1);L^{1}(\R^d))}\left(A+\norm{w}\norm{K_{2}}_{L^{\infty}(\R^d)}\norm{g(s,\cdot,\cdot)}_{L^{\infty}((0,1);L^{1}(\R^d))}\right)+ \beta(s) \right)\,ds,
\end{split}
\]
where
\begin{align*}
    \beta(s) &:= \norm{w}\norm{\divv K_1}_{L^{\infty}(\R^d)}\norm{g(s,\cdot,\cdot)}_{L^{\infty}((0,1);W^{1,1}(\R^d))} \|f(s,\cdot.\cdot)\|_{L^{\infty}((0,1);L^1(\R^d))}\\
    &+\norm{f(s,\cdot,\cdot)}_{L^{\infty}( (0,1)\times\R^d)}\norm{\nabla g(s,\cdot,\cdot)}_{L^{\infty}((0,1);L^1(\R^d))}   \norm{w}\norm{K_{2}}_{L^{1}(\R^d)}.
\end{align*}
Using (\ref{l1bound}) and (\ref{finfty}) we obtain  
\[
\beta(s) \leq \left(C_1e^{At_{*}} + C_2e^{C_{3}t_{*}\norm{g}_{L^{\infty}((0,t_*)\times (0,1);L^{1}(\R^d))}} \right)\norm{ g(s,\cdot,\cdot)}_{L^{\infty}((0,1);W^{1,1}(\R^d))},
\]
where $C_{1},C_{2},C_{3}$ are positive constants dependent only on norms of $K_{1},K_2,w$ and $f^0$.
Applying Gr\"onwall estimate we arrive at
\begin{align}\label{gradl1}
&\|\nabla f(t,\cdot,\cdot)\|_{L^{\infty}((0,1); L^1(\R^d))} \leq \left( \|\nabla f^0\|_{L^{\infty}((0,1); L^1(\R^d))} + \int_0^t \beta(s)ds \right)\\ \nonumber
& \times \exp\left(\left(A+\norm{w}(\norm{K_{1}}_{W^{1,1}(\R^d)}+\norm{K_{2}}_{L^{\infty}(\R^d)})\norm{g}_{L^{\infty}((0,t_*)\times (0,1);(L^{1}\cap L^{\infty})(\R^d))}\right)t\right).
\end{align}

Similarly we deal with the $L^{\infty}$ - norm of the gradient. Combining (\ref{gradinfty}) with (\ref{Rgradinfty}) and using (\ref{v2}), (\ref{v4}) we obtain
\[
  \begin{split}
    &\|\nabla f(t,\cdot,\cdot)\|_{L^{\infty}((0,1)\times \R^d )}\leq \|\nabla f^0\|_{L^{\infty}((0,1)\times \R^d )}\,\exp\left(2t \norm{w}\norm{K_1}_{W^{1,1}(\R^d)}\norm{g}_{L^{\infty}((0,t_*)\times (0,1);L^{\infty}(\R^d))}\right)\\
        &+\int_0^t \exp\left(2(t-s)\norm{w}\norm{K_1}_{W^{1,1}(\R^d)}\norm{g}_{L^{\infty}((0,t_*)\times (0,1);L^{\infty}(\R^d))}\right)\\
        &\times\left( \norm{\nabla f(s,\cdot,\cdot)}_{L^{\infty}((0,1);L^{\infty}(\R^d))}\left(A+\norm{w}\norm{K_{2}}_{L^{\infty}(\R^d)}\norm{g(s,\cdot,\cdot)}_{L^{\infty}((0,1);L^{1}(\R^d))}\right)+ \gamma(s) \right)\,ds,
\end{split}
\]
where
\begin{align*}
\gamma(s) &:= \norm{w}\norm{\divv K_1}_{L^{\infty}(\R^d)}\norm{g(s,\cdot,\cdot)}_{L^{\infty}((0,1);W^{1,1}(\R^d))} \|f(s,\cdot.\cdot)\|_{L^{\infty}((0,1);L^{\infty}(\R^d))}\\
    &+\norm{f(s,\cdot,\cdot)}_{L^{\infty}( (0,1)\times\R^d)}\norm{\nabla g(s,\cdot,\cdot)}_{L^{\infty}((0,1);L^\infty(\R^d))}   \norm{w}\norm{K_{2}}_{L^{1}(\R^d)}.
\end{align*}

Note that by (\ref{finfty}) 
\[
\gamma(s) \leq C_4e^{ C_5t_{*}\norm{g}_{L^{\infty}((0,t_*)\times (0,1);L^{1}(\R^d))}}\left[\norm{ g(s,\cdot,\cdot)}_{L^{\infty}((0,1);W^{1,1}(\R^d))}+\norm{ g(s,\cdot,\cdot)}_{L^{\infty}((0,1);W^{1,\infty}(\R^d))}\right],
\]
where $C_{4},C_{5}$ are positive constants dependent only on norms of $K_{1},K_2,w$ and $f^0$.
Again by Gr\"onwall estimate 
\begin{align}\label{gradlinfty}
    &\|\nabla f(t,\cdot,\cdot)\|_{L^{\infty}((0,1)\times\R^d)} \leq \left( \|\nabla f^0\|_{L^{\infty}((0,1)\times \R^d)} + \int_0^t \gamma(s)ds\right) \\ \nonumber
&\times\exp\left(2t\left(A+\norm{w}\left[\norm{K_{1}}_{W^{1,1}(\R^d)}\norm{g}_{L^{\infty}((0,t_*)\times (0,1)\times(\R^d))}+\norm{K_{2}}_{L^{\infty}(\R^d)}\norm{g}_{L^{\infty}((0,t_{*})\times (0,1);L^{1}(\R^d))}\right]\right)\right).    
\end{align}
 Using the estimates (\ref{l1bound}), (\ref{finfty}), (\ref{gradl1}) and (\ref{gradlinfty}) we may show that if if $\Upsilon> \norm{f^0}_{L^{\infty}((0,1);W^{1,1}\cap W^{1,\infty}(\R^d))}e^{A}$, then for $g \in E_{\Upsilon}$, $Pg \in E_{\Upsilon}$ for sufficiently small $t_{*}$ which depends on $\Upsilon$, $A,\norm{w}$ and norms of $K_{1}$ and $K_{2}$.
 Indeed, from estimates (\ref{l1bound}), (\ref{finfty}) we infer that $\norm{f}_{L^{\infty}((0,t_{*})\times(0,1);L^{1}\cap L^{\infty}(\R^d))} \leq \Upsilon$ if 
 \eqq{
 t_{*} \leq \frac{A}{A+\norm{\divv K_1}_{L^{\infty}(\R^d)}\norm{w}\Upsilon}.
 }{t1}
 Estimate (\ref{gradl1}) gives  $\norm{\nabla f}_{L^{\infty}((0,t_{*})\times(0,1);L^{1}(\R^d))} \leq \Upsilon$, if $t_{*}$ satisfies
 \begin{align*}
 &\Upsilon e^{-A} \exp\left(\left(A+\norm{w}(\norm{K_{1}}_{W^{1,1}(\R^d)}+\norm{K_{2}}_{L^{\infty}(\R^d)})\Upsilon\right)t\right) \leq e^{-\frac{A}{2}}\Upsilon, \\
 &t_{*}\left(C_1e^{At_{*}} + C_2e^{C_{3}\Upsilon t_{*}} \right)\norm{g}_{E}\exp\left(\left(A+\norm{w}(\norm{K_{1}}_{W^{1,1}(\R^d)}+\norm{K_{2}}_{L^{\infty}(\R^d)})\Upsilon\right)t\right) \leq (1-e^{-\frac{A}{2}})\Upsilon,
 \end{align*}
 which is true for 
 \eqq{
 t_{*}\leq \min\Big\{\frac{A}{2\left(A+\norm{w}(\norm{K_{1}}_{W^{1,1}(\R^d)}+\norm{K_{2}}_{L^{\infty}(\R^d)})\Upsilon\right)},\frac{(1-e^{-\frac{A}{2}})e^{-\left(A+\norm{w}(\norm{K_{1}}_{W^{1,1}(\R^d)}+\norm{K_{2}}_{L^{\infty}(\R^d)})\Upsilon\right)}}{C_1e^{A} + C_2e^{C_{3}\Upsilon } }\Big\}, 
 }{t2}
 where $C_1,C_2,C_3$ depends on norms of $K_1,K_2,w$ and $\Upsilon$.
 Similarly from (\ref{gradlinfty}) we obtain\\ $\norm{\nabla f}_{L^{\infty}((0,t_{*})\times(0,1);L^{\infty}(\R^d))} \leq \Upsilon$ if
 \eqq{
 t_{*}\leq \min\Big\{\frac{A}{4\left(A+\norm{w}(\norm{K_{1}}_{W^{1,1}(\R^d)}+\norm{K_{2}}_{L^{\infty}(\R^d)})\Upsilon\right)},\frac{(1-e^{-\frac{A}{2}})e^{-2\left(A+\norm{w}(\norm{K_{1}}_{W^{1,1}(\R^d)}+\norm{K_{2}}_{L^{\infty}(\R^d)})\Upsilon\right)}}{C_4e^{C_5\Upsilon} } \Big\},
 }{t3}
where $C_4,C_5$ depends on norms of $K_1,K_2,w$ and $\Upsilon$. \\ Furthermore, from the representation (\ref{integro}) (respectively \ref{mildformulation}) we obtain that $f$ is weak$^*$ continuous in $L^{\infty}((0,1);(W^{1,1}\cap W^{1,\infty}(\R^d))$ and satisfies the initial condition.

Our final step is to show that $P$ is a contraction in $(E_{\Upsilon}, L^{1}((0,t_{*})\times (0,1)\times \R^d))$ for $t_*$ small enough. Let us consider $g_1,g_2\in E_\Upsilon$ and note that $Pg_1 - Pg_2$ satisfies
\[
\partial_t(Pg_1-Pg_2)+\divv(V_1[g_2]\,(Pg_1-Pg_2))+\divv((V_1[g_1]-V_1[g_2])Pg_1)
\]
\[
=\nu\Delta (Pg_1-Pg_2) + (A-V_2[g_1])(Pg_1-Pg_2) - Pg_2(V_2[g_1] - V_2[g_2])
\]
with zero initial condition.
Applying (\ref{propL1}) and taking $L^{1}$- norm with respect to $\xi$ leads
\begin{align*}
   &\Vert (Pg_1-Pg_2)(t,\cdot,\cdot)\Vert_{L^1( (0,1)\times \R^d)}\leq \norm{\divv((V_1[g_1]-V_1[g_2])Pg_1)}_{L^1((0,t)\times (0,1)\times \R^d)} \\
 &  + \norm{(A-V_2[g_1])(Pg_1-Pg_2)}_{L^1((0,t)\times (0,1)\times \R^d)}
+\norm{Pg_2(V_2[g_1] - V_2[g_2])}_{L^1((0,t)\times (0,1)\times \R^d)} =:\sum_{j=1}^3 J_j.
\end{align*}
Let us estimate term by term. From the linearity of $V_1$ and Lemma \ref{impoest} we have
\begin{align*}
J_{1}&\leq \Vert \divv(V_1[g_1-g_2])\Vert_{L^1((0,t)\times (0,1)\times \R^d)}\Vert P[g_1]\Vert_{{L^{\infty}((0,t)\times (0,1)\times \R^d)}} \\
& +\Vert V_1[g_1-g_2]\Vert_{L^1((0,t)\times (0,1)\times \R^d)}\Vert \nabla P[g_1]\Vert_{{L^{\infty}((0,t)\times (0,1)\times \R^d)}}  \\
 & \leq \norm{w}(\norm{\divv K_1}_{L^{1}(\R^d)}+\norm{ K_1}_{L^{1}(\R^d)})
 \norm{g_1-g_2}_{L^1((0,t)\times (0,1)\times \R^d)}\Vert P[g_1]\Vert_{E}.   
\end{align*}
Applying (\ref{v1}) we have
\[
J_2 \leq \norm{Pg_1-Pg_2}_{L^1((0,t)\times (0,1)\times \R^d)}\left(A + \norm{w}\norm{K_{2}}_{L^{\infty}(\R^d)}\norm{g_1}_{L^{\infty}((0,t)\times (0,1);L^{1}(\R^d))},\right)
\]
Finally, by linearity of $V_2$ and Lemma \ref{impoest}
\begin{align*}
   J_3 &\leq \norm{V_2[g_1-g_2]}_{L^1((0,t)\times (0,1)\times \R^d)}\Vert P[g_2]\Vert_{{L^{\infty}((0,t)\times (0,1)\times \R^d)}} \\
   &\leq \norm{w}\norm{K_{2}}_{L^{1}(\R^d)}\norm{g_1-g_2}_{L^1((0,t)\times (0,1)\times \R^d)}\Vert P[g_2]\Vert_{{L^{\infty}((0,t)\times (0,1)\times \R^d)}}.
\end{align*}
Thus,
\begin{align*}
    &\Vert (Pg_1-Pg_2)(t,\cdot,\cdot)\Vert_{L^1( (0,1)\times \R^d)}\leq  \left(A + \norm{w}\norm{K_{2}}_{L^{\infty}(\R^d)}\Upsilon\right)\norm{Pg_1-Pg_2}_{L^1((0,t)\times (0,1)\times \R^d)} \\ 
    &+\Upsilon\norm{w}\left(\norm{\divv K_1}_{L^1(\R^d)} + \norm{ K_1}_{L^1(\R^d)}+ \norm{ K_2}_{L^1(\R^d)}\right)\norm{g_1-g_2}_{L^1((0,t)\times (0,1)\times \R^d)}.
\end{align*}
Applying Gr\"onwall inequality we get
\begin{align*}
  &\Vert (Pg_1-Pg_2)(t,\cdot,\cdot)\Vert_{L^1( (0,1)\times \R^d)} \leq \exp\left(A + \norm{w}\norm{K_{2}}_{L^{\infty}(\R^d)}\Upsilon\right)\\
&\times \Upsilon\norm{w}\left(\norm{\divv K_1}_{L^1(\R^d)} +  \norm{ K_1}_{L^1(\R^d)}+\norm{ K_2}_{L^1(\R^d)}\right)\norm{g_1-g_2}_{L^1((0,t_{*})\times (0,1)\times \R^d)} 
\end{align*}
and thus
\begin{align*}
 &\Vert (Pg_1-Pg_2)\Vert_{L^1( (0,t_{*})\times (0,1)\times \R^d)} \leq t_{*}\exp\left(A + \norm{w}\norm{K_{2}}_{L^{\infty}(\R^d)}\Upsilon\right)\\  
&\times \Upsilon\norm{w}\left(\norm{\divv K_1}_{L^1(\R^d)} + \norm{ K_1}_{L^1(\R^d)}+ \norm{ K_2}_{L^1(\R^d)}\right)\norm{g_1-g_2}_{L^1((0,t_{*})\times (0,1)\times \R^d)}. 
\end{align*}
Choosing $t_{*}$  satisfying the bounds (\ref{t1})-(\ref{t3}) and 
\[
t_{*} \leq \frac{1}{2} e^{-\left(A + \norm{w}\norm{K_{2}}_{L^{\infty}(\R^d)}\Upsilon\right)}\left(\Upsilon\norm{w}\left(\norm{\divv K_1}_{L^1(\R^d)} + \norm{ K_1}_{L^1(\R^d)}+ \norm{ K_2}_{L^1(\R^d)}\right)\right)^{-1}
\]
we obtain that $P$ is contractive on $(E_{\Upsilon}, L^{1}((0,t_{*})\times (0,1)\times \R^d))$. By the Banach fixed point theorem there exists a unique fixed point  of $P$, leading to a weak solution of \eqref{meannu} in $[0, t_*]$. Note that replacing $g$ with $f$ in the estimates (\ref{l1bound}), (\ref{finfty}), they do not blow up in the finite time. The estimates (\ref{gradl1}) and (\ref{gradlinfty}) also do not blow up in a finite time by Gr\"onwall argument. Thus, we may extend the solution to any finite time interval, iterating the contraction argument step by step.

\end{proof}

\end{document}